\documentclass[11pt]{article}

\usepackage[margin=1in]{geometry}
\usepackage{amsmath,amssymb,amsthm,mathtools}
\usepackage{enumitem}
\usepackage{graphicx}
\usepackage{mathrsfs}
\usepackage{hyperref}
\usepackage{xcolor}
\usepackage{booktabs}
\usepackage{tabularx}
\usepackage{caption}
\usepackage{calc}
\usepackage{ragged2e}
\usepackage{setspace}
\newcommand{\FundingLogos}{%
  \raisebox{0pt}{\includegraphics[height=1.5cm]{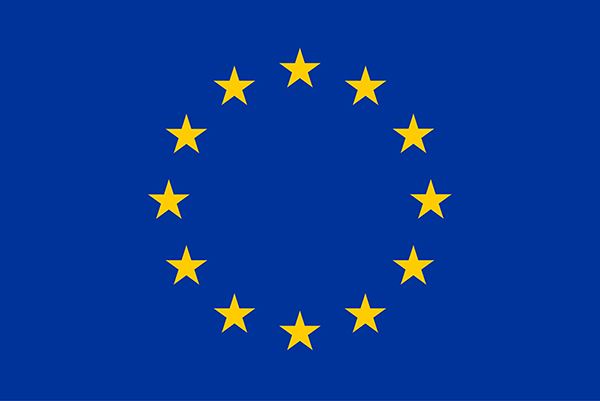}}%
  \hspace{1em}%
  \raisebox{0pt}{\includegraphics[height=1.5cm]{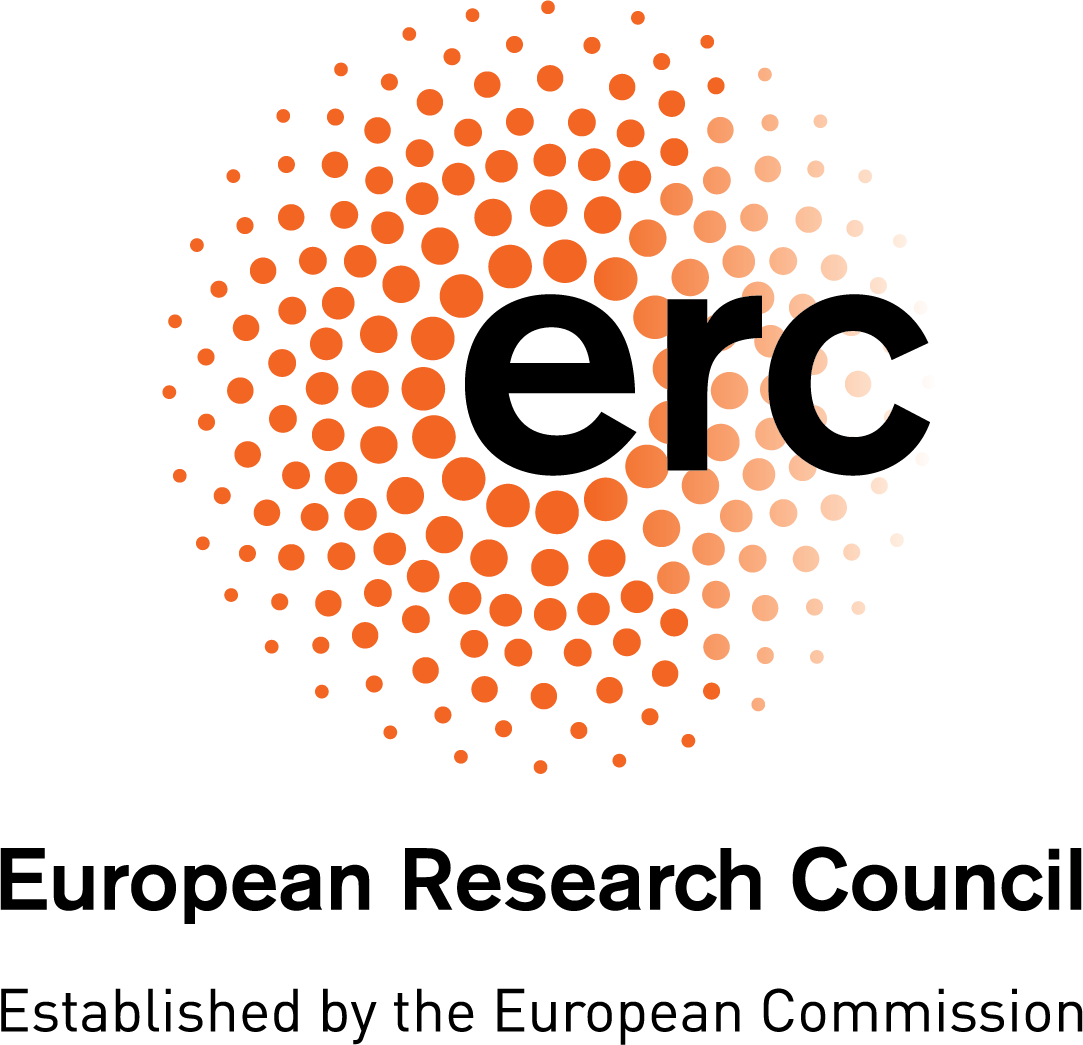}}%
}

\usepackage{tcolorbox}
\tcbset{colback=white,boxrule=0.5pt,arc=2pt,left=4pt,right=4pt,top=4pt,bottom=4pt}

\hypersetup{colorlinks=true,linkcolor=blue,citecolor=blue,urlcolor=blue}

\newtheorem{theorem}{Theorem}[section]
\newtheorem{proposition}[theorem]{Proposition}
\newtheorem{lemma}[theorem]{Lemma}
\newtheorem{corollary}[theorem]{Corollary}
\newtheorem{assumption}[theorem]{Assumption}
\theoremstyle{definition}
\newtheorem{definition}[theorem]{Definition}
\newtheorem{example}[theorem]{Example}
\newtheorem{remark}[theorem]{Remark}

\newcommand{\R}{\mathbb{R}}
\newcommand{\N}{\mathbb{N}}
\newcommand{\eps}{\varepsilon}
\newcommand{\dd}{\,\mathrm{d}}
\newcommand{\supp}{\operatorname{supp}}

\newcommand{\Lip}{\operatorname{Lip}}

\newcommand{\one}{\mathbf{1}}
\newcommand{\calP}{\mathcal{P}}

\newcommand{\calE}{\mathcal{E}}

\newcommand{\calL}{\mathcal{L}}
\newcommand{\calS}{\mathcal{S}}

\newcommand{\Hh}{\mathcal{H}}
\newcommand{\norm}[1]{\left\|#1\right\|}

\newcommand{\ip}[2]{\left\langle #1,#2\right\rangle}
\newcommand{\weak}{\rightharpoonup}
\newcommand{\grad}{\nabla}

\newcommand{\diver}{\operatorname{div}}
\newcommand{\sgn}{\operatorname{sgn}}

\newenvironment{revision}{}{}
\newcommand{\sectioneightedit}[1]{#1}
\newenvironment{sectioneightchanges}{}{}

\usepackage[colorinlistoftodos,bordercolor=black,backgroundcolor=white,linecolor=black,textsize=scriptsize]{todonotes}

\numberwithin{equation}{section}

\usepackage{xurl}
\AddToHook{env/thebibliography/begin}{%
  \small\urlstyle{same}\setlength{\itemsep}{0pt}\setlength{\parsep}{0pt}%
}

\title{The Nonlocal Attraction-Repulsion Transport Equation with Power Kernels}
\author{
Massimo Fornasier \thanks{Department of Mathematics \& Munich Data Science Institute, Technical University of Munich \& Munich Center for Machine Learning
  Email: \texttt{massimo.fornasier@ma.tum.de}}
\and
  Hui Huang \thanks{Department of Mathematics,  Hunan University
  Email: \texttt{huihuang1@hnu.edu.cn}
} \and 
  Lukang Sun\thanks{Department of Mathematics, Technical University of Munich \& Munich Center for Machine Learning
  Email: \texttt{lukang.sun@tum.de}}}

\date{\today}
\begin{document}
\maketitle

\begin{abstract}
We study a nonlocal continuity equation on $\mathbb{R}^d$ in which a probability 
density is driven by the competition between attraction toward a prescribed background 
measure $\omega$ and self-repulsion among particles, governed respectively by the 
power-law kernels $\psi_a(x) = |x|^{1+a}$ and $\psi_r(x) = |x|^{1+r}$ with exponents 
$a, r \in [0,1)$. We establish global Lagrangian well-posedness via a squared-radius 
regularization, obtaining {finite-time $L^\infty$ and moment bounds},  
$W^{n,\infty}$ regularity, and uniqueness in the Lagrangian class. When the initial 
data is compactly supported and attraction dominates ($a > r$, or $a = r$ with 
$\omega(\mathbb{R}^d) > 1$), we prove that the support remains uniformly bounded at
all time; a counterexample shows this fails for $a = r > 1$. For the 
attractive-dominant nonquadratic range $0 \leq r \leq a < 1$, we characterize 
zero-flux stationary states via a free-boundary problem involving a fractional Laplacian operator, reducing the stationarity condition to a fractional exterior 
Dirichlet problem. This characterization allows us to exhibit explicit examples of stationary measures in dimensions $d \in \{1,2,3\}$. Numerical particle simulations {are consistent} with the theoretical stationary profiles. 
Finally, for $0 \le r \le a < 1$, we prove that every global solution whose
energy is bounded from below and whose moments are bounded uniformly in time
converges, along a sequence of times, to a zero-flux stationary state.  Full narrow convergence follows whenever the omega-limit set contains a single stationary state.
\end{abstract}

{\it Keywords}: {well-posedness of gradient flows of singular attraction-repulsion potentials, 
nonlocal interactions, steady states, large-time behavior and asymptotics} \\

AMS Classification: 35Q70, 35A01, 35B40, 49Q22\\

\tableofcontents

\section{Introduction}\label{sec:intro}

We study the nonlocal transport equation
\begin{equation}\label{eq:main}
    \partial_t \phi_t = \operatorname{div}\!\bigl(\phi_t(\nabla V - K_r * \phi_t)\bigr),
    \qquad t > 0,
\end{equation}
where the background potential is $V = \psi_a * \omega$ for a prescribed
nonnegative density $\omega \in L^1(\mathbb{R}^d) \cap L^\infty(\mathbb{R}^d)$,
and the interaction and attraction kernels are given by the power-law
family
\begin{equation}
    \psi_s(x) := |x|^{1+s}, \qquad K_s(x) := \nabla \psi_s(x) =
    (1+s)|x|^{s-1}x, \quad x \neq 0,
\end{equation}
for exponents $a, r \in [0,1]$. {As case $a=r=1$ reduces to a simple ODE with explicit solutions (see, e.g. \cite[Section 4.1]{difrancesco2015asymptotic}), in this paper we actually consider the range $a, r \in [0,1)$ only.
The background measure $\omega$ may have arbitrary values of mass $\omega(\mathbb R^d)$.}
The unknown $\phi_t$ is a time-dependent
probability density on $\mathbb{R}^d$, and \eqref{eq:main} is a continuity
equation driven by the competition between attraction toward the
background measure $\omega$ and self-repulsion among particles.
Equation \eqref{eq:main} is the formal Wasserstein gradient flow of the
interaction energy \eqref{eq:energy1} below.
\\

A primary motivation for the study of \eqref{eq:main} comes from two
distinct but related modelling perspectives. In the equal-exponent case
$a = r$ with $\omega \in \mathcal{P}(\mathbb{R}^d)$, the dynamics
\eqref{eq:main} formally correspond to the Wasserstein gradient flow of
the Maximum Mean Discrepancy $\operatorname{MMD}^2(\cdot, \omega)$ with
negative-distance kernel \cite{gretton2012kernel}. {We recall below the background on MMD and related results on its gradient flow.}

The attractive-dominant case $a > r$ gives rise to a richer and equally
compelling model. Here the population density $\phi_t$ is drawn toward
the resource distribution $\omega$ by a long-range attraction kernel
$K_a$, while being held back by an internal repulsion $K_r * \phi_t$
that models an internal competition among agents. Because the attraction exponent
strictly dominates the repulsion exponent, the dynamics {are expected to} drive $\phi_t$
toward a stationary state $\phi_\infty$ which does not in general coincide with $\omega$.
This mismatch between the equilibrium population $\phi_\infty$ and the
resource distribution $\omega$ can be interpreted as a prototype model of \emph{unequal access}:
agents are attracted to resources but their internal competition prevents
full redistribution. The resulting inequality in the distribution of
wealth can be quantified by the cumulative distribution function of
$\phi_\infty$ with respect to the Radon--Nikodym density of $\omega$
relative to $\phi_\infty$, namely
\begin{equation}\label{eq:CDF}
    F_{\phi_\infty, \omega}(w) := \phi_\infty\!\left(\left\{x \in
    \mathbb{R}^d : \frac{\mathrm{d}\omega}{\mathrm{d}\phi_\infty}(x)
    \leq w\right\}\right), \qquad w \geq 0,
\end{equation}
which measures the fraction of the population whose per-capita share of
resources does not exceed the level $w$. The explicit characterization
of $\phi_\infty$ developed in Sections~\ref{sec:stationary}
and~\ref{sec:examples} thus provides a concrete tool for analyzing the
structure of such inequality distributions.
\\
{Well-posedness of the model \eqref{eq:main} applies also in the repulsion-dominant
case $a < r$. Yet, in this
repulsion dominant regime the driving energy \eqref{eq:energy1} does not in general admit
a minimizer, as shown in \cite{fornasier2016consistency}, and the
repulsive term $K_r * \phi_t$ dominates the confining effect of the
background potential $\nabla V$ at large distances, preventing the
formation of stationary states. Therefore, the large-time behavior in
this regime poses significantly greater challenges than in the attraction-dominant regime. A local-in-space
description may still be possible under suitable assumptions on the initial data, but it is beyond the scope of this paper.}
\\

{The analysis of
\eqref{eq:main} poses significant mathematical challenges in all cases {$a,r\in [0,1)$}:
the kernel $K_s$ lacks a global Lipschitz bound, and for $s = 0$ it is
singular at the origin, requiring a careful regularization strategy
that we develop in Section~\ref{sec:wellposedness}. In particular, while \eqref{eq:main} is formally a Wasserstein gradient flow, the corresponding energy is not $\lambda$-convex, as shown in \cite{hertrich2023generative,hertrich2024wasserstein} for $a=r$, and the general theory \cite{savare2008gradientflows} does not apply. 
 }

\subsection*{Gradient flow structure and energy}

Equation \eqref{eq:main} is the formal Wasserstein gradient flow of the
interaction energy
\begin{equation}\label{eq:energy1}
    E(\phi) := \int_{\mathbb{R}^d} (\psi_a * \omega)(x)\, \mathrm{d}\phi(x)
    - \frac{1}{2}\iint_{\mathbb{R}^d \times \mathbb{R}^d}
    \psi_r(x-y)\, \mathrm{d}\phi(x)\, \mathrm{d}\phi(y).
\end{equation}
More precisely, the first variation of $E$ is
\[
    F_\phi(x) := \frac{\delta E}{\delta \phi}(x)
    = (\psi_a * \omega)(x) - (\psi_r * \phi)(x),
\]
and \eqref{eq:main} takes the form $\partial_t \phi_t =
\operatorname{div}(\phi_t \nabla F_{\phi_t})$, which is the standard
structure of a Wasserstein gradient flow \cite{savare2008gradientflows}.
The energy $E$ is monotonically nonincreasing along solutions, with
dissipation identity
\[
    \frac{\mathrm{d}}{\mathrm{d}t} E(\phi_t)
    = -\int_{\mathbb{R}^d} |K_a * \omega - K_r * \phi_t|^2\,
    \mathrm{d}\phi_t \leq 0.
\]
Energies of the form \eqref{eq:energy1} belong to the broader class of
nonlocal power repulsion-attraction energies. The Wasserstein gradient
flow of such energies was first studied by Fornasier, Ha\v{s}kovec, and
Steidl \cite{fornasier2013consistency}, who considered the mean-field
kinetic limit of discrete particle systems used for image halftoning in
the sense of continuous-domain quantization.  In one dimension,
Di Francesco, Fornasier, H{\"u}tter, and Matthes
\cite{difrancesco2015asymptotic} subsequently developed a comprehensive
well-posedness theory for the Wasserstein gradient flow of power-law
repulsion-attraction energies, proving uniform confinement of the
support and characterizing the asymptotic behavior of solutions in terms
of stationary states for a general class of power-law potentials. The one-dimensional case has been more recently studied also in 
\cite{duong2026wasserstein} by Duong, Stein, Beinert, Hertrich and
Steidl,  giving a comprehensive description of
Wasserstein gradient flows for $a=r$  on the real line, yielding
explicit piecewise linear solution formulas for discrete targets and
proving instantaneous regularization of initial point masses.\\

The present paper extends several of these results to arbitrary
dimension $d \geq 1$ and to {the discordant regime $a \neq r$,
where the attraction and repulsion exponents are allowed to differ},  and to $a=r, \omega(\mathbb R^d) \geq 1$,  allowing the mass of the target to exceed $1$ (unbalanced matching).

\subsection*{Relation to the broader literature on confinement and
nonlocal repulsion}

The energy \eqref{eq:energy1} belongs to the well-studied class of
\emph{confinement-repulsion} functionals of the form
\begin{equation}\label{eq:general_energy}
    \mathcal{E}(\phi) = \int_{\mathbb{R}^d} V(x)\, \mathrm{d}\phi(x)
    + \frac{1}{2} \iint_{\mathbb{R}^d \times \mathbb{R}^d}
    W(x-y)\, \mathrm{d}\phi(x)\, \mathrm{d}\phi(y),
\end{equation}
where $V$ is a confining potential and $W$ is a (possibly singular)
repulsive-attractive interaction kernel. {A cornerstone of the}
Wasserstein gradient flow theory for such energies {is the work of}
Carrillo, McCann, and Villani
\cite{carrillo2003kinetic, carrillo2006contractions}, who studied the
granular media equation (a paradigmatic example of
\eqref{eq:general_energy} with quadratic confinement $V(x) = |x|^2/2$
and a smooth interaction kernel) and proved long-time convergence to
equilibrium via displacement convexity and entropy dissipation
inequalities in the Wasserstein metric. The global-in-time well-posedness
theory for measure-valued solutions of the pure interaction equation
(without the confinement term) was then developed by Carrillo, Di
Francesco, Figalli, Laurent, and Slep\v{c}ev \cite{carrillo2011global},
who established existence and uniqueness of weak measure solutions for a
general class of {$\lambda$-convex} interaction potentials and
characterized conditions under which solutions aggregate in finite time.

The companion paper \cite{carrillo2012confinement} by the same authors
identifies sufficient conditions on the potential $W$ for the support of
the solution to remain uniformly confined for all time{; these
results were subsequently extended to a broader class of
repulsive-attractive kernels by Balagu\'e, Carrillo, and Yao
\cite{balague2014confinement}. Both works are directly relevant to
Section~\ref{sec:confinement}, where we prove analogous confinement for
\eqref{eq:energy1} in the attractive-dominant regime, namely $a>r$, or
$a=r$ with $\omega(\mathbb{R}^d)>1$. In the pure interaction setting the
center of mass is conserved and serves as a natural anchor for the support.
Here the background $\omega$ breaks translation invariance, so we instead
anchor the support through an energy-sublevel estimate and control its
diameter.}

In our setting, the confinement potential $V = \psi_a * \omega =
|\cdot|^{1+a} * \omega$ is not prescribed externally but is itself
generated by the interaction of the evolving density with the fixed
background $\omega$ via the power-law kernel $\psi_a$, while the
repulsive part is given by the self-interaction kernel $W = -\psi_r =
-|\cdot|^{1+r}$, {a power-law (Riesz-type) kernel of order
$1+r \in [1,2)$}. The structure is therefore that of a
\emph{background-driven confinement} paired with \emph{Riesz
self-repulsion}. {Since $-|x|^{1+r}$ is not $\lambda$-convex
near the origin, this combination falls outside the measure-valued theory
of \cite{carrillo2011global}. The local singularity of the repulsive force
is, however, of the type covered by the theory of bounded-density
solutions of aggregation equations developed by Bertozzi, Laurent, and
Rosado \cite{bertozzi2011lp} and by Carrillo, Choi, and Hauray
\cite{carrillo2014derivation}; our maximum-principle bound and our
Lagrangian uniqueness argument follow the same strategy. The new
ingredients here are the background-driven attraction, the growth of both
kernels at infinity, and the absence of a compact-support assumption on the
data. These are handled by} the regularization and compactness arguments
developed in Sections~\ref{sec:wellposedness} and~\ref{sec:confinement}.

For the {interaction energy $\mathcal{E}(\phi) =
\frac{1}{2}\iint W(x-y)\,\mathrm{d}\phi(x)\,\mathrm{d}\phi(y)$ with
repulsion of Riesz or Newtonian type at the origin}, the regularity of
stationary states (local minimizers) has been analyzed by Carrillo,
Delgadino, and Mellet \cite{carrillo2016regularity}, who showed that
Newtonian or stronger repulsion forces local minimizers to be locally
bounded densities, by reducing the Euler--Lagrange stationarity
condition to a classical obstacle problem for the fractional Laplacian.
This connection between stationarity conditions and obstacle problems
is {analogous in spirit to} the free-boundary characterization
developed in Section~\ref{sec:stationary} of the present paper, where the
stationarity condition $\phi\,(K_a * \omega - K_r * \phi) = 0$ is
reformulated as a fractional exterior Dirichlet problem via the fractional
Laplacian. {The two settings differ in two respects. First,
zero-flux stationarity imposes no condition off the support, so one obtains
a Dirichlet problem rather than an obstacle problem. Second, the repulsion
$|x|^{1+r}$ is weaker than Newtonian (except for $d=1$, $r=0$), so the
inverse operator $(-\Delta)^{(d+1+r)/2}$ has order larger than two.}

{In the case $a=r=0$ and $d=3$, the repulsion $-|x|$ is, up to
a constant, the fundamental solution of $\Delta^2$. The resulting explicit
stationary states of Section~\ref{sec:examples}, which consist of a bulk
density together with a uniform layer on a sphere, have a close counterpart
in the equilibrium measures computed by Chafa\"i, Saff, and Womersley
\cite{chafai2023threshold} for Riesz kernels with power-law external fields
in dimension $d=s+4$; for $s=-1$ this is exactly $d=3$ with kernel $-|x|$.
Our examples differ in that the confinement is generated by a background
measure rather than by an external power field, and in that we construct
zero-flux stationary states rather than minimizers.}

The existence of compactly supported global minimizers
for interaction energies with Riesz-type repulsion was established by
Ca\~{n}izo, Carrillo, and Patacchini \cite{canizo2015existence} under
general repulsive-attractive potentials; our explicit
{constructions in Sections~\ref{sec:stationary}
and~\ref{sec:examples} are complementary: they concern zero-flux stationary
states of the background-driven energy \eqref{eq:energy1}, which need not
be minimizers, and identify their density profile in specific geometries.}
For a comprehensive overview of the broader theory of aggregation-diffusion
equations with confinement and nonlocal repulsion, including dynamics,
asymptotics, and singular limits, we refer to the survey by Carrillo,
Craig, and Yao \cite{carrillo2019aggregation}.

\subsection*{Maximum Mean Discrepancy: background}
The Maximum Mean Discrepancy was introduced in the machine learning
literature as a kernel-based two-sample test statistic
\cite{gretton2012kernel}. Given a positive-definite kernel $K$, the MMD
metrizes weak convergence for smooth kernels such as the Gaussian
$\exp(-|x-y|^2/2\sigma^2)$ or the inverse multiquadratic $(|x-y|^2 +
\sigma^2)^{-1/2}$, and enjoys the dimension-independent empirical
convergence rate $\operatorname{MMD}(\phi^N, \omega)  \asymp N^{-1/2}$, in
sharp contrast to the Wasserstein distance whose empirical rate
deteriorates as $N^{-1/d}$ \cite{fournier2015rate}. Its dual
representation $\operatorname{MMD}(\phi, \omega) = \sup_f \int f\,
\mathrm{d}(\phi - \omega)$ over the unit ball of the associated reproducing
kernel Hilbert space provides interpretable approximation guarantees for
ensemble statistics \cite{steinwart2008support}\footnote{The notations $\operatorname{MMD}$ and $\operatorname{MMD}^2$ are used consistently below;  mind that the notation $\operatorname{MMD}^2$ is actually the square of $\operatorname{MMD}$ and they may appear differently in the text, whether we identify $\operatorname{MMD}^2$ with the energy $E$ (for $a=r$) up to an additive constant or we identify $\operatorname{MMD}$ with a distance expressed in terms of a norm of the associated homogeneous Sobolev space}.

For the negative-distance kernel $K(x) = -|x|^q$ with $q \in (0,2)$ ($q=1+a$),
the MMD coincides with the classical energy distance and is equivalent
to a homogeneous Sobolev seminorm:
\begin{equation}\label{MMDonRd}
     \operatorname{MMD}(\phi, \omega) = c_{d,q}\|\phi - \omega\|_{\dot{H}^{-(d+q)/2}}.
\end{equation}
      
This connection to fractional Sobolev spaces, established in multiple contributions along many years, see
\cite{colasanto2026sharp} and references therein\footnote{The study of integral probability metrics such as the MMD is deeply rooted 
in classical work on Energy Distances in statistics and conditionally positive definite 
kernels, with a history spanning several decades. A thorough account of 
this literature falls outside the scope of the present paper; we refer 
instead to~\cite{colasanto2026sharp} for a detailed historical discussion.}, {makes the negative-distance MMD
particularly well-suited as an objective for generative sampling, since
convergence in MMD controls moment convergence up to order $q/2$, as explored  in \cite{hertrich2023generative,hertrich2024wasserstein}}.
\subsection*{Wasserstein gradient flows of MMD}
{As recalled at the beginning of this introduction, for $a=r$ the equation \eqref{eq:main} corresponds to the gradient flow of the squared Maximum Mean Discrepancy with kernel $K(x) = -|x|^q$ with $q \in (0,2)$ ($q=1+a$).}
The study of Wasserstein gradient flows of MMD-type energies has
attracted considerable attention, both from the perspective of
generative models in machine learning and from that of mathematical analysis. 
To our knowledge, the first work on the description of this type of Wasserstein gradient flows may be traced back to \cite{fornasier2013consistency}.
{For smooth kernels with Lipschitz gradient, such as the Gaussian
kernel, Arbel, Korba, Salim, and Gretton \cite{arbel2019maximum} constructed
the Wasserstein gradient flow of $\operatorname{MMD}^2$, proving existence and uniqueness
through the associated McKean--Vlasov equation and the energy-dissipation
identity, and analyzed its forward-Euler time discretization and its particle
approximation. For such kernels $\operatorname{MMD}^2$ is displacement semiconvex, that is,
$\lambda$-convex for some $\lambda<0$, which places the flow within the
framework of \cite{savare2008gradientflows}. However, it is not displacement
convex, and the flow may converge to equilibria other than the target; global
convergence is obtained only under the condition that a weighted negative
Sobolev distance between the solution and the target remains bounded along
the flow, and a noise-injection regularization is proposed to escape
non-global equilibria.}  

For the singular negative-distance kernel $K(x) = -|x|^q$ with $q \in
(0,2)$, the situation is fundamentally different. Hertrich et al.\
\cite{hertrich2023generative,hertrich2024wasserstein} showed that for $d \geq 2$ the interaction
energy is \emph{not} displacement $\lambda$-convex for any $\lambda \in
\mathbb{R}$, ruling out the classical  framework \cite{savare2008gradientflows}. Despite this, they
demonstrated strong empirical performance of the resulting generative
sampler. \\
Recent work by Chizat et al.\ \cite{chizat2026quantitative} studies the  Wasserstein gradient flow of the MMD energy on the torus $\mathbb{T}^d$, modeled by
squared homogeneous Sobolev energy
\[
  \operatorname{MMD}(\phi,\omega)
  := \tfrac{1}{\sqrt 2}\|\phi-\omega\|_{H_0^{-s}(\mathbb{T}^d)},
  \quad s \geq 1.
\]
The paper establishes three main results:
(i)~global well-posedness in weak regularity classes without absolute
continuity of $\phi_0$, via a Yudovich-type argument;
(ii)~for $s=1$ (Coulomb kernel), global \emph{exponential} convergence
to $\omega$ under minimal assumptions; and
(iii)~for $s>1$, local \emph{polynomial} convergence at sharp rates
$O(t^{-\alpha})$, with $\alpha$ depending explicitly on $s$ and the
Sobolev regularity of the data.
As an application, the training dynamics of infinite-width shallow ReLU
networks, recast as a Wasserstein--Fisher--Rao gradient flow on $\mathbb{S}^d$,
correspond to the case $s=(d+3)/2$, yielding explicit polynomial convergence
rates. Let us stress that for the Wasserstein gradient flow of the MMD energy on $\mathbb R^d$ \eqref{MMDonRd}, one cannot apply the same techniques as in \cite{chizat2026quantitative}, 
{in particular, a global Polyak–Łojasiewicz inequality fails in this case \cite{rosenzweig2025mmd}.}
\\
{
The rigorous analysis of the Wasserstein gradient flow of the MMD energy on
$\mathbb{R}^d$ with kernel $-|x|^q$, $0<q<2$ (in our notation $q=1+a$), has
recently been carried out by Rosenzweig, Slepčev, and Wang \cite{rosenzweig2025mmd}. When $d+q-2>0$,
they prove global well-posedness for source and target densities in
$\mathcal{P}_r(\mathbb{R}^d)\cap L^p(\mathbb{R}^d)$, $r\ge 1$, $p>d/(d+q-2)$,
with a spatially Lipschitz velocity field and an energy--dissipation identity;
the endpoint $d=q=1$ is included. For the particle system they prove global
noncollision and a modulated-energy estimate controlling
$\mathrm{MMD}^2(\phi^N_t,\phi_t)$ by $\mathrm{MMD}^2(\phi^N_0,\phi_0)$ on every
finite time interval. For $1\le q<2$ and every $d$, absolutely continuous
zero-flux stationary states with finite $q$-moment coincide with the target,
every $\omega$-limit point of the flow is a zero-flux stationary state, and the
flow converges to the target under uniform-in-time moment and $L^p$ bounds.
Finally, no initial-data-independent MMD decay rate exists on $\mathbb{R}^d$,
and global Polyak--{\L}ojasiewicz inequalities fail.
}


\subsection*{Contributions and organization of this paper}
{{The present work focuses on} the discordant regime $a \neq r$,
in which the background attraction and self-repulsion exponents are allowed
to differ.} In Section~\ref{sec:wellposedness}, we establish global
Lagrangian well-posedness via a squared-radius regularization, obtaining
uniform {finite-time} $L^\infty$ and moment bounds, propagation of $W^{n,\infty}$
regularity, and uniqueness in the Lagrangian class. When the initial data
is compactly supported and attraction dominates, either $a > r$, or
$a = r$ with $\omega(\mathbb{R}^d) > 1$, we prove in
Section~\ref{sec:confinement} that the support of $\phi_t$ remains
confined to a fixed ball for all time; a counterexample shows this fails
for $a = r > 1$. For the nonquadratic attractive-dominant range $0 \leq r
\leq a < 1$, Section~\ref{sec:stationary} develops a free-boundary
characterization of zero-flux stationary states via a fractional Laplacian, reducing the stationarity condition
$\phi\,\nabla F_\phi = 0$ to a fractional exterior Dirichlet problem;
{this characterization yields explicit examples of stationary measures in one dimension and in two dimensions via the
Boggio--Kelvin Green kernel on disks, and in three dimensions for the uniform measure on a ball and Gaussian targets $\omega$. 
{For the one- and two-dimensional examples, particle simulations
produce profiles consistent with the predicted stationary states,
while the three-dimensional examples are verified analytically. Finally,
Section~\ref{sec:longtime} establishes, for $0\le r\le a<1$, subsequential
convergence to zero-flux stationary states for every global solution whose
energy is bounded from below and whose moments are bounded uniformly in
time. {For $r=0$, a three-point symmetrization yields uniform inverse-distance moment bounds and uniform H\"older continuity of the dissipation defect. Consequently every omega-limit point is zero-flux stationary and the trajectory approaches its omega-limit set; for $r>0$, an energy-dissipation bootstrap gives the same conclusions. Full narrow convergence holds when the omega-limit set contains a single stationary state.}}
Because of its relevance for applications in machine learning and statistics, the MMD gradient flow case $a=r$ is discussed in a concluding Section \ref{sec:a=r}, where 
 we clarify more explicitly which of our results holds in this case\footnote{
   The MMD gradient flow case $a=r$ with $\omega(\mathbb{R}^d)=1$ has been studied
   in detail by Rosenzweig, Slep\v{c}ev, and Wang \cite{rosenzweig2025mmd}. The present paper therefore
   focuses on the discordant regime $a\neq r$ and on arbitrary mass
   $\omega(\mathbb{R}^d)$, keeping the case $a=r$ as a secondary consideration.}.
}

\section{Setting, notation, and solution concepts}

For $s\in[0,1]$ set
\begin{equation}\label{eq:psiK}
        \psi_s(x):=|x|^{1+s},
        \qquad
        K_s(x):=\grad\psi_s(x)=(1+s)|x|^{s-1}x,
        \qquad x\ne0.
\end{equation}
When $d=1$ and $s=0$ we use the convention $K_0(x)=\sgn(x)$ and
$K_0(0)=0$.  The prescribed background is a nonnegative density $\omega$.
We define
\begin{equation}\label{eq:V}
        V(x):=(\psi_a*\omega)(x)
        =\int_{\R^d}|x-y|^{1+a}\omega(y)\dd y .
\end{equation}
The equation studied in this paper is
\begin{equation}\label{eq:pde}
        \partial_t\phi_t
        =\diver\Bigl(\phi_t\bigl(\grad V-K_r*\phi_t\bigr)\Bigr),
        \qquad t>0,
\end{equation}
or, equivalently,
\begin{equation}\label{eq:continuity}
        \partial_t\phi_t+\diver(\phi_t b_{\phi_t})=0,
        \qquad
        b_{\phi_t}(x):=-\grad V(x)+(K_r*\phi_t)(x).
\end{equation}
Unless explicitly stated otherwise, \(\phi_0\) is a probability density, and we
use the notation \(\phi_t\) for the time-dependent solution, \(\phi_T\) for the
numerical solution~(as computed approximately in the numerical experiments) at a terminal time \(T\), and \(\phi_*\) or \(\phi_\infty\)
for stationary limits.  The symbol \(\psi_s\) is reserved exclusively for the
interaction kernel and is not used to denote an evolving density.
\begin{assumption}[Standing hypotheses for the evolution]\label{ass:standing}
Throughout the well-posedness part we assume that there exists an exponent
$p>1$ such that
\begin{equation}\label{eq:standing1}
        \omega\ge0,
        \qquad
        \omega\in L^1(\R^d)\cap L^\infty(\R^d),
        \qquad
        \int_{\R^d}\langle x\rangle^p\omega(x)\dd x<\infty,
\end{equation}
and
\begin{equation}\label{eq:standing2}
        \phi_0\ge0,
        \qquad
        \phi_0\in L^1(\R^d)\cap L^\infty(\R^d),
        \qquad
        \int_{\R^d}\phi_0(x)\dd x=1,
        \qquad
        \int_{\R^d}\langle x\rangle^p\phi_0(x)\dd x<\infty,
\end{equation}
where $\langle x\rangle=(1+|x|^2)^{1/2}$.  No compact-support assumption is
made in the well-posedness theorem.  Compact support is imposed only in the
separate uniform-support theorem, where it is part of the conclusion to prove
that the support remains uniformly bounded for all time.
\end{assumption}

\begin{definition}[Distributional and Lagrangian solutions]\label{def:solutions}
Let $T>0$.  A nonnegative curve
\[
        \phi\in L^\infty(0,T;L^1(\R^d)\cap L^\infty(\R^d))
\]
is a distributional solution of \eqref{eq:pde} on $[0,T]$ if, for every
$\zeta\in C_c^\infty([0,T)\times\R^d)$,
\begin{equation}\label{eq:weaksol}
\int_0^T\!\int_{\R^d}
\Bigl(\partial_t\zeta(t,x)+b_{\phi_t}(x)\cdot\grad\zeta(t,x)\Bigr)
\phi_t(x)\dd x\dd t
+\int_{\R^d}\zeta(0,x)\phi_0(x)\dd x=0.
\end{equation}
It is a Lagrangian solution if there is a flow map $X_t(x)$ solving
\begin{equation}\label{eq:char}
        \dot X_t(x)=-\grad V(X_t(x))+(K_r*\phi_t)(X_t(x)),
        \qquad X_0(x)=x,
\end{equation}
for $\phi_0$-a.e. $x$, and
\begin{equation}\label{eq:pushforward}
        \phi_t=(X_t)_{\#}\phi_0.
\end{equation}
\end{definition}
In this work, apart from the well-posedness theory, we focus primarily on the case \(a,r\in[0,1)\). The quadratic
endpoint \(a=r=1\) is simple, since the characteristic flow can be computed explicitly by a direct calculation.

The associated interaction energy is
\begin{equation}\label{eq:energy}
        E(\phi)
        :=\int_{\R^d}(\psi_a*\omega)(x)\dd\phi(x)
        -\frac12\iint_{\R^d\times\R^d}\psi_r(x-y)\dd\phi(x)\dd\phi(y).
\end{equation}
Its first variation is
\begin{equation}\label{eq:firstvar}
        F_\phi(x):=\frac{\delta E}{\delta\phi}(x)
        =\psi_a*\omega(x)-\psi_r*\phi(x),
\end{equation}
and \eqref{eq:pde} is formally
\begin{equation}\label{eq:gf-form}
        \partial_t\phi_t=\diver\bigl(\phi_t\grad F_{\phi_t}\bigr).
\end{equation}

\begin{definition}[Zero-flux stationary state]\label{def:zeroflux}
A probability measure $\phi$ is a zero-flux stationary state if
\begin{equation}\label{eq:zeroflux}
        \phi\,\grad F_\phi=0
\end{equation}
in the sense of vector-valued distributions, equivalently
\begin{equation}\label{eq:zeroflux-test}
        \int_{\R^d}\xi(x)\cdot
        \bigl(K_a*\omega(x)-K_r*\phi(x)\bigr)\dd\phi(x)=0
\end{equation}
for every $\xi\in C_c^\infty(\R^d;\R^d)$ for which the integral is well-defined.
\end{definition}

\section{Regularized kernels and uniform estimates}\label{sec:3}

The regularization used throughout the approximation argument is the one
specified by the squared radius:
\begin{equation}\label{eq:specified-mollifier}
        {\qquad
        \psi_{s,\eps}(x):=(|x|^2+\eps)^{(1+s)/2},
        \qquad
        K_{s,\eps}(x):=\grad\psi_{s,\eps}(x),
        \qquad 0<\eps\le1.
        \qquad}
\end{equation}
This is not a convolution mollification.  Its advantage is that convexity and
one-sided sign estimates are preserved exactly.

\begin{lemma}[Formulas and convexity of the squared-radius regularization]\label{lem:eps-formulas}
Let $s\in[0,1]$ and $0<\eps\le1$.  Then
\begin{equation}\label{eq:grad-eps}
        K_{s,\eps}(x)
        =(1+s)(|x|^2+\eps)^{(s-1)/2}x,
\end{equation}
\begin{equation}\label{eq:hess-eps}
        D^2\psi_{s,\eps}(x)
        =(1+s)(|x|^2+\eps)^{(s-1)/2}I
        +(1+s)(s-1)(|x|^2+\eps)^{(s-3)/2}x\otimes x,
\end{equation}
and
\begin{equation}\label{eq:lap-eps}
        \Delta\psi_{s,\eps}(x)
        =(1+s)(|x|^2+\eps)^{(s-3)/2}
        \Bigl(d\eps+(d+s-1)|x|^2\Bigr).
\end{equation}
In particular,
\begin{equation}\label{eq:convex-lap}
        D^2\psi_{s,\eps}(x)\succcurlyeq0,
        \qquad
        \Delta\psi_{s,\eps}(x)\ge0,
        \qquad x\in\R^d.
\end{equation}
Moreover, $K_{s,\eps}\to K_s$ in $L^1_{\rm loc}(\R^d)$ as $\eps\downarrow0$.
\end{lemma}

\begin{proof}
The identities \eqref{eq:grad-eps}--\eqref{eq:lap-eps} follow by direct
differentiation.  To prove convexity, decompose any vector $\xi$ into its
radial and tangential parts relative to $x$.  The tangential eigenvalue of
$D^2\psi_{s,\eps}$ is $ (1+s)(|x|^2+\eps)^{(s-1)/2}\ge0,$
and the radial eigenvalue is $ (1+s)(|x|^2+\eps)^{(s-3)/2}(\eps+s|x|^2)\ge0.$
Thus $D^2\psi_{s,\eps}$ is positive semidefinite; \eqref{eq:lap-eps} then also
shows $\Delta\psi_{s,\eps}\ge0$.  Finally, for $x\ne0$,
$K_{s,\eps}(x)\to K_s(x)$ pointwise, while
$|K_{s,\eps}(x)|\le C_s(1+|x|^s)$ and the singularity at the origin is locally
integrable.  Dominated convergence on compact sets gives
$K_{s,\eps}\to K_s$ in $L^1_{\rm loc}$.
\end{proof}

\begin{lemma}[Uniform local integrability of regularized derivatives]\label{lem:regularized-derivatives}
Let $m\ge2$, $s\in[0,1]$, and assume either
\begin{equation}\label{eq:strict-derivative-condition}
        m<d+1+s,
\end{equation}
or the endpoint $(d,s,m)=(1,0,2)$.  Then there is a constant
$C=C(d,s,m)$, independent of $\eps\in(0,1]$, such that for every $\lambda>0$,
\begin{equation}\label{eq:local-Dm-eps}
        \int_{|x|<\lambda}|D^m\psi_{s,\eps}(x)|\dd x
        \le C\bigl(1+\lambda^{d+1+s-m}\bigr),
\end{equation}
and
\begin{equation}\label{eq:far-Dm-eps}
        \sup_{|x|\ge\lambda}|D^m\psi_{s,\eps}(x)|
        \le C\lambda^{1+s-m}.
\end{equation}
Consequently, for every $f\in L^1(\R^d)\cap L^\infty(\R^d)$,
\begin{equation}\label{eq:convolution-Dm-eps}
        \sup_{0<\eps\le1}\norm{D^m\psi_{s,\eps}*f}_{L^\infty}
        \le C_{d,s,m}
        \bigl(\norm{f}_{L^1}+\norm{f}_{L^\infty}\bigr).
\end{equation}
\end{lemma}

\begin{proof}
The point is to keep all estimates uniform in the squared-radius regularization.
For $m\ge2$, repeated differentiation of
$(|x|^2+\eps)^{(1+s)/2}$ gives
\begin{equation}\label{eq:Dm-raw-bound}
        |D^m\psi_{s,\eps}(x)|
        \le C_{d,s,m}(|x|^2+\eps)^{(1+s-m)/2}.
\end{equation}
Equivalently, with $\delta=\sqrt\eps$,
$ D^m\psi_{s,\eps}(x)=\delta^{1+s-m}
        (D^m\psi_{s,1})(x/\delta).$
For $|x|\le\delta$, the right-hand side in \eqref{eq:Dm-raw-bound} is bounded
by $C\delta^{1+s-m}$, and hence $ \int_{|x|\le\delta}|D^m\psi_{s,\eps}(x)|\dd x
        \le C\delta^{d+1+s-m}.$
For $\delta<|x|<\lambda$, \eqref{eq:Dm-raw-bound} gives
$|D^m\psi_{s,\eps}(x)|\le C|x|^{1+s-m}$.  Thus
\[
        \int_{\delta<|x|<\lambda}|D^m\psi_{s,\eps}(x)|\dd x
        \le C\int_0^\lambda \rho^{d+s-m}\dd\rho
        \le C\lambda^{d+1+s-m},
\]
provided $m<d+1+s$.  This proves \eqref{eq:local-Dm-eps} in the strict case.
In the endpoint $(d,s,m)=(1,0,2)$ the preceding power integral is logarithmic,
but the exact formula $ \frac{\dd^2}{\dd x^2}(x^2+\eps)^{1/2}
        =\eps(x^2+\eps)^{-3/2}$ 
has $L^1(\R)$ norm equal to $2$, uniformly in $\eps$, so the endpoint is also
covered.

For $|x|\ge\lambda$, the scaling formula and the asymptotic estimate
$|D^m\psi_{s,1}(x)|\le C|x|^{1+s-m}$ for $|x|\ge1$ yield
\eqref{eq:far-Dm-eps}; if $|x|/\delta<1$, then the bound is even easier because
$1+s-m<0$ and $\delta^{1+s-m}\le C\lambda^{1+s-m}$ when $|x|\ge\lambda$.
Finally, for any $x\in\R^d$ split the convolution as
\[
D^m\psi_{s,\eps}*f(x)=
\int_{|x-y|<1}D^m\psi_{s,\eps}(x-y)f(y)\dd y
+
\int_{|x-y|\ge1}D^m\psi_{s,\eps}(x-y)f(y)\dd y .
\]
The first term is controlled by \eqref{eq:local-Dm-eps} with $\lambda=1$ and
$\norm{f}_{L^\infty}$; the second is controlled by \eqref{eq:far-Dm-eps} with
$\lambda=1$ and $\norm{f}_{L^1}$.  This proves
\eqref{eq:convolution-Dm-eps}.
\end{proof}

\begin{lemma}[Averaged Lipschitz estimate for the singular force]\label{lem:averaged-lip}
Let $r\in[0,1]$, $\sigma\in L^1(\R^d)\cap L^\infty(\R^d)$, $\sigma\ge0$.
There exists
$C=C(d,r,\norm{\sigma}_{L^1},\norm{\sigma}_{L^\infty})$ such that, for every
$0<\eta\le1$,
\begin{equation}\label{eq:avg-lip}
        \sup_{x\in\R^d}
        \int_{\R^d}\sup_{|h|\le\eta}
        |K_r(x-y+h)-K_r(x-y)|\sigma(y)\dd y
        \le C\eta .
\end{equation}
The endpoint $d=1$, $r=0$ is included with $K_0=\sgn$.
\end{lemma}

\begin{proof}
Fix $x\in\R^d$ and split the $y$-integral into the three regions
$|x-y|\le2\eta$, $2\eta<|x-y|\le2$, and $|x-y|>2$.

On $|x-y|\le2\eta$, if $0<r\le1$, then $|K_r(x-y+h)-K_r(x-y)|\le
C\eta^r$ for $|h|\le\eta$.  Therefore this region contributes at most
$C\eta^{d+r}\norm{\sigma}_{L^\infty}\le C\eta$.  If $r=0$ and $d\ge2$,
$K_0$ is bounded and the same region contributes at most
$C\eta^d\norm{\sigma}_{L^\infty}\le C\eta$.

On $2\eta<|x-y|\le2$, the segment joining $x-y$ to $x-y+h$ stays away from the
origin.  Since $|DK_r(x-y+\theta h)|\le C|x-y|^{r-1}$ for $0\le\theta\le1$, the mean-value theorem gives
\[
        |K_r(x-y+h)-K_r(x-y)|\le C\eta |x-y|^{r-1}.
\]
Hence this annulus contributes at most $  C\eta\norm{\sigma}_{L^\infty}
        \int_{2\eta<|x-y|\le2}|x-y|^{r-1}\dd y
        \le C\eta,$
because the last integral is uniformly bounded when $d+r-1>0$.

On $|x-y|>2$, the derivative $DK_r$ is bounded along the same segment, and the
contribution is at most $C\eta\norm{\sigma}_{L^1}$.

It remains only to treat the one-dimensional sign kernel $d=1$, $r=0$.  In this
case $|\sgn(x-y+h)-\sgn(x-y)|$ can be nonzero only when $|x-y|\le\eta$.
Therefore
\[
        \int_\R\sup_{|h|\le\eta}
        |\sgn(x-y+h)-\sgn(x-y)|\sigma(y)\dd y
        \le 2\int_{|x-y|\le\eta}\sigma(y)\dd y
        \le4\eta\norm{\sigma}_{L^\infty}.
\]
Combining the estimates and then taking the supremum over $x$ proves
\eqref{eq:avg-lip}.
\end{proof}

\section{Well-posedness and Sobolev propagation}\label{sec:wellposedness}

\begin{theorem}[Global Lagrangian well-posedness via the squared-radius regularization]\label{thm:wellposed}
Let $a,r\in[0,1]$ and let Assumption \ref{ass:standing} hold.  Then for every
$T>0$ there exists a nonnegative Lagrangian distributional solution of
\eqref{eq:pde} on $[0,T]$.  It conserves mass,
\begin{equation}\label{eq:mass-cons}
        \int_{\R^d}\phi_t(x)\dd x=1,
        \qquad 0\le t\le T,
\end{equation}
satisfies the finite-time maximum bound
\begin{equation}\label{eq:Linfty-main}
        \norm{\phi_t}_{L^\infty}
        \le
        \exp\!\bigl(t C_{d,a}(\norm{\omega}_{L^1}+\norm{\omega}_{L^\infty})\bigr)
        \norm{\phi_0}_{L^\infty},
        \qquad 0\le t\le T,
\end{equation}
and satisfies the moment estimate
\begin{equation}\label{eq:p-moment-main}
        \sup_{0\le t\le T}\int_{\R^d}\langle x\rangle^p\phi_t(x)\dd x
        \le C_T\left(1+
        \int_{\R^d}\langle x\rangle^p\phi_0(x)\dd x
        +\int_{\R^d}\langle x\rangle^p\omega(x)\dd x\right).
\end{equation}
The solution is unique in the class of Lagrangian distributional solutions
satisfying the bounds \eqref{eq:Linfty-main}--\eqref{eq:p-moment-main} on each
finite time interval.
\end{theorem}
\begin{proof}[Proof sketch]
We regularize both interaction potentials by
$$
\psi_{\sigma,\varepsilon}(x)
:=\bigl(|x|^2+\varepsilon\bigr)^{(1+\sigma)/2},
\qquad
K_{\sigma,\varepsilon}:=\nabla\psi_{\sigma,\varepsilon},
\qquad \sigma\in\{a,r\},
$$and approximate \(\phi_0\) by smooth probability densities
\(\phi_{0,\varepsilon}\) with uniformly bounded \(L^\infty\)-norms and
\(p\)-moments.  The regularized equation admits a global smooth
Lagrangian solution
\(\phi^\varepsilon_t=(X^\varepsilon_t)_\#\phi_{0,\varepsilon}\).

The positivity of \(\Delta\psi_{r,\varepsilon}\), together with the
uniform convolution estimates from Section~\ref{sec:3}, yields an
\(\varepsilon\)-independent \(L^\infty\)-bound. {This is a classical maximum-principle argument for bounded-density solutions of aggregation equations, see \cite{raoul2012nonlocal}}. 
The linear-growth estimate for the regularized force then gives the
uniform \(p\)-moment bound by Gronwall's inequality.  These estimates
imply tightness and equicontinuity of
\((\phi^\varepsilon_t)_{t\in[0,T]}\) in the bounded-Lipschitz metric;
hence, after extraction,
$$
\phi^\varepsilon\to\phi
\quad\text{narrowly, uniformly on }[0,T].
$$

To identify the limit equation, one splits each singular convolution
into a near-field part and a far-field part.  The near-field
contribution is uniformly small by the \(L^\infty\)-bound, while the
far-field contribution converges by local uniform convergence of the
regularized kernels and the uniform moment bound.  This yields local
\(L^1_tL^\infty_x\)-convergence of the velocities and allows one to
pass to the limit in the weak formulation.

Finally, uniform local Lipschitz bounds for the regularized flows give,
by Arzela--Ascoli, a limiting flow \(X_t\), and the push-forward
identity passes to the limit: $\phi_t=(X_t)_\#\phi_0.$

Uniqueness follows by comparing two such flows on large balls.  
{This argument is similar to the Lagrangian stability strategy for bounded-density solutions of \cite{carrillo2014derivation, bertozzi2011lp}, whose framework covers repulsive kernels behaving like $-|x|^{b}$ near the origin with $b>2-d$, hence our $b=1+r$.} The
local part is controlled by the averaged Lipschitz estimate for the
force, whereas the contribution of the complement is controlled by
the moment tail; Gronwall's inequality and then the limit of the
radius to infinity conclude the argument. Since the detailed proof is lengthy, we defer it to Appendix~\ref{sec:pwd} in order to preserve the flow of the main text.
\end{proof}

\begin{proposition}[Propagation of $W^{n,\infty}$ regularity]\label{prop:Sobolev}
Assume that one of the following conditions holds:
\begin{enumerate}[label=(\roman*)]
    \item $n<d-1$ and $a,r\in[0,1]$;
    \item $n=d-1$ and $a,r\in(0,1]$.
\end{enumerate}
In addition, suppose that
\begin{equation}\label{eq:Wn-ass}
    \phi_0\in W^{n,\infty}(\R^d)\cap L^1(\R^d),
    \qquad
    \omega\in L^1(\R^d)\cap L^\infty(\R^d),
\end{equation}
then every smooth solution of the regularized equation \eqref{eq:regularized-pde}
satisfies, for $0\le t\le T$,
\begin{equation}\label{eq:Wninf-bound}
        \sup_{0<\eps\le1}\norm{\phi_t^\eps}_{W^{n,\infty}}
        \le C_T\Bigl(1+\norm{\phi_0}_{W^{n,\infty}}\Bigr),
\end{equation}
where $C_T$ depends on $d,n,a,r,T$, on
$\norm{\omega}_{L^1\cap L^\infty}$, and on
$\norm{\phi_0}_{L^1\cap L^\infty}$, but not on $\eps$.  Consequently, the
solution of Theorem \ref{thm:wellposed} belongs to
$L^\infty(0,T;W^{n,\infty})$ and satisfies the same estimate.
\end{proposition}

\begin{proof}
Write $ B_t^\eps:=\grad V_\eps-K_{r,\eps}*\phi_t^\eps,
        \partial_t\phi_t^\eps=B_t^\eps\cdot\grad\phi_t^\eps
        +\phi_t^\eps\diver B_t^\eps.$
The case $n=0$ is the maximum estimate \eqref{eq:Linfty-eps}.  Let
$1\le m\le n$, let $\alpha$ be a multi-index with $|\alpha|=m$, and apply
$D^\alpha$ to the equation.  Leibniz' rule gives
\begin{equation}\label{eq:Dalpha-eq}
\begin{aligned}
\partial_tD^\alpha\phi_t^\eps
&=B_t^\eps\cdot\grad D^\alpha\phi_t^\eps
  +D^\alpha\phi_t^\eps\,\diver B_t^\eps \\
&\quad+
\sum_{0<\beta\le\alpha}c_{\alpha\beta}
D^\beta B_t^\eps\cdot\grad D^{\alpha-\beta}\phi_t^\eps
+
\sum_{\beta<\alpha}d_{\alpha\beta}
D^{\alpha-\beta}\phi_t^\eps\,D^\beta\diver B_t^\eps.
\end{aligned}
\end{equation}
At a point where $|D^\alpha\phi_t^\eps|$ attains its maximum, the transport
term has no positive contribution in the maximum-principle argument.  Hence
\begin{equation}\label{eq:Mm-ineq}
\frac{\dd}{\dd t}\norm{D^\alpha\phi_t^\eps}_{L^\infty}
\le
C\norm{D^\alpha\phi_t^\eps}_{L^\infty}
\norm{\diver B_t^\eps}_{L^\infty}
+C\sum_{j=0}^{m-1}M_j(t)\,A_{m,j}(t),
\end{equation}
where $  M_j(t):=\max_{|\gamma|\le j}\norm{D^\gamma\phi_t^\eps}_{L^\infty},$
and $A_{m,j}(t)$ is a finite sum of norms of derivatives of $B_t^\eps$ of
order at most $m+1-j$ and derivatives of $\diver B_t^\eps$ of order at most
$m-j$.
Thus the largest derivative of $V_\eps$ or $\psi_{r,\eps}*\phi_t^\eps$ that
appears is of order $m+2\le n+2$.  For $\ell\le n+2$, Lemma
\ref{lem:regularized-derivatives} gives
\begin{equation}\label{eq:kernel-ell-bound}
        \norm{D^\ell V_\eps}_{L^\infty}
        +\norm{D^\ell(\psi_{r,\eps}*\phi_t^\eps)}_{L^\infty}
        \le C_T
\end{equation}
provided $\ell<d+1+a$ and $\ell<d+1+r$.  The two alternatives in the statement
are exactly what is needed for this condition for all $\ell\le n+2$; when
$d=n+1$, the strict inequality requires $a,r>0$, while when $d>n+1$ it also
allows $a=r=0$.  Combining \eqref{eq:Mm-ineq} and \eqref{eq:kernel-ell-bound}
and summing over $m\le n$ yields
$ \frac{\dd}{\dd t}\norm{\phi_t^\eps}_{W^{n,\infty}}
        \le C_T\bigl(1+\norm{\phi_t^\eps}_{W^{n,\infty}}\bigr).$
Gronwall's lemma proves \eqref{eq:Wninf-bound}.  Passing to the limit
$\eps\downarrow0$ follows by weak-* compactness in
$L^\infty(0,T;W^{n,\infty})$.
\end{proof}

\section{Uniform compactness of the support} \label{sec:confinement}

The well-posedness theorem does not require compact support.  Compact support is
considered only in this section, where the aim is different: we prove that, in the
attractive regimes, compactly supported initial data generate solutions whose
supports remain contained in a fixed ball for all times.  This uniform support
bound will be used in Section~\ref{sec:longtime} to study the large-time
convergence.  The proof has two main ingredients.  First, an energy sublevel
estimate ensures that the support always contains at least one point in a fixed
ball.  Second, once the diameter of the support is sufficiently large, it cannot
increase further.
\begin{revision}
For the pure interaction equation, confinement estimates were obtained
in \cite{carrillo2012confinement, balague2014confinement} by directly
estimating \(R(t)=\max_i |x_i(t)|\). In the present equation, the background
\(\omega\) breaks translation invariance, and the center of mass is no
longer conserved. Moreover, applying a direct estimate of
\(\max_i|x_i(t)|\) to our equation does not give the sharp mass
condition when \(a=r\in(0,1)\): at a particle with maximal radius
\(R(t)\), the radial repulsive force is bounded above using
\(|x_i-x_j|\le 2R(t)\), giving a term of order \(2^rR(t)^r\), whereas
the attractive force has leading order \(M R(t)^r\), where
\(M=\omega(\mathbb R^d)\). This estimate therefore requires
\(M>2^r\). We obtain the sharper condition \(M>1\) by a different
argument: an energy-sublevel estimate first provides an anchor point
that remains in a fixed ball, and we then control the support diameter,
\(\max_{i,j}|x_i(t)-x_j(t)|\). We use the
energy \(E\) defined in \eqref{eq:energy1}. For \(a,r\in[0,1)\), under
the moment condition in Proposition~\ref{prop:energy-diss}, that
proposition gives \(E(\phi_t)\le E(\phi_0)\).
\end{revision}

\begin{proposition}[Moment anchoring from the energy sublevel]\label{prop:moment-anchor}\begin{revision}
Let $a,r\in[0,1)$.\end{revision}
Assume that $\phi_t$ is a solution with $\calE(\phi_t)\le\calE(\phi_0)$ for all
$t\ge0$ and that $\omega$ has finite $(1+a)$-moment.  If either
\begin{enumerate}[label=(\roman*)]
\item $a>r$ and $\omega(\R^d)>0$, or
\item $a=r$ and $\omega(\R^d)>1$,
\end{enumerate}
then there are constants $C_0,R_0<\infty$, depending only on the energy
sublevel, $a,r$, and $\omega,\phi_0$, such that
\begin{equation}\label{eq:uniform-moment-anchor}
        \sup_{t\ge0}\int_{\R^d}|x|^{1+a}\dd\phi_t(x)\le C_0,
        \qquad
        \supp\phi_t\cap B_{R_0}(0)\ne\emptyset
        \quad\text{for every }t\ge0 .
\end{equation}
\end{proposition}
\begin{remark}\label{rmk:eql}
In the equal-exponent case \(a=r\in[0,1)\), if \(\omega\in\mathcal P_2(\R^d)\),
Fornasier and Hütter~\cite{fornasier2016consistency} proved that, for every
\(p\in\bigl(0,(1+a)/2\bigr)\), there exists a constant \(C_0<\infty\) such that
$\sup_{t\ge0}\int_{\R^d}|x|^p\,\dd\phi_t(x)\le C_0 .$
\end{remark}
\begin{proof}
Write $q_a=1+a$ and $q_r=1+r$.  We use repeatedly the elementary inequalities
valid for $q\ge1$:
\begin{equation}\label{eq:q-triangle-lower}
        |x-y|^q\ge 2^{1-q}|x|^q-|y|^q,
        \qquad
        |x-y|^q\le 2^{q-1}(|x|^q+|y|^q).
\end{equation}
First assume $a>r$.  Since $\omega(\R^d)>0$, the attractive part satisfies
\begin{align}
\int (\psi_a*\omega)(x)\dd\mu(x)
&=\iint |x-y|^{q_a}\omega(y)\dd y\dd\mu(x) \notag\\
&\ge 2^{1-q_a}\omega(\R^d)\int |x|^{q_a}\dd\mu(x)
       -\int |y|^{q_a}\omega(y)\dd y .
\label{eq:anchor-attraction-lower}
\end{align}
The repulsive part satisfies
\begin{equation}\label{eq:anchor-repulsion-upper}
        \frac12\iint |x-z|^{q_r}\dd\mu(x)\dd\mu(z)
        \le C\int(1+|x|^{q_r})\dd\mu(x).
\end{equation}
Combining \eqref{eq:anchor-attraction-lower}--\eqref{eq:anchor-repulsion-upper}
gives
\begin{equation}\label{eq:anchor-a-r-preabsorb}
        E(\mu)
        \ge c\int |x|^{q_a}\dd\mu(x)
        -C\int |x|^{q_r}\dd\mu(x)-C .
\end{equation}
Because $q_r<q_a$, Young's inequality gives
$|x|^{q_r}\le \eta |x|^{q_a}+C_\eta$.  Choosing $\eta$ small enough in
\eqref{eq:anchor-a-r-preabsorb} yields
\begin{equation}\label{eq:anchor-a-r-final}
        E(\mu)\ge c\int |x|^{q_a}\dd\mu(x)-C .
\end{equation}
Therefore every energy sublevel has a uniformly bounded $q_a$-moment.

Now assume $a=r$ and put $q=1+a$.  Let $M=\omega(\R^d)>1$ and
$\bar\omega=M^{-1}\omega$.  Decompose
\begin{align}
E(\mu)
&=(M-1)\iint |x-y|^q\dd\bar\omega(y)\dd\mu(x) \notag\\
&\quad +\left[\iint |x-y|^q\dd\bar\omega(y)\dd\mu(x)
        -\frac12\iint |x-z|^q\dd\mu(x)\dd\mu(z)\right].
\label{eq:equal-power-decomp-detailed}
\end{align}
The bracket is bounded from below.  Indeed, adding and subtracting
$\frac12\iint |y-y'|^q\dd\bar\omega(y)\dd\bar\omega(y')$ gives
\begin{align}
&\iint |x-y|^q\dd\bar\omega(y)\dd\mu(x)
        -\frac12\iint |x-z|^q\dd\mu(x)\dd\mu(z) \notag\\
&\qquad=\frac12\iint |y-y'|^q\dd\bar\omega(y)\dd\bar\omega(y')
        -\frac12\iint |u-v|^q\dd(\mu-\bar\omega)(u)\dd(\mu-\bar\omega)(v).
\label{eq:equal-power-cnd}
\end{align}
For $1\le q\leq2$, the kernel $|x|^q$ is conditionally negative definite, see Berg--Christensen--Ressel~\cite[Chapter 3, Section 2, Corollary 2.10]{BergChristensenRessel1984}, hence
the second term on the right-hand side of \eqref{eq:equal-power-cnd} is
nonnegative. Thus the bracket in \eqref{eq:equal-power-decomp-detailed} is
bounded below by a constant depending only on $\bar\omega$.  The first term in
\eqref{eq:equal-power-decomp-detailed} controls the $q$-moment by
\eqref{eq:q-triangle-lower}:
\begin{equation}\label{eq:equal-power-moment-control}
        \iint |x-y|^q\dd\bar\omega(y)\dd\mu(x)
        \ge 2^{1-q}\int |x|^q\dd\mu(x)
          -\int |y|^q\dd\bar\omega(y).
\end{equation}
Since $M-1>0$, \eqref{eq:equal-power-decomp-detailed} and
\eqref{eq:equal-power-moment-control} give the uniform $q$-moment bound.

Applying the preceding estimates to $\mu=\phi_t$ and using
$E(\phi_t)\le E(\phi_0)$ gives the first part of
\eqref{eq:uniform-moment-anchor}.  The anchoring is then immediate: if
$\supp\phi_t\cap B_R(0)=\emptyset$, then
$\int |x|^{1+a}\dd\phi_t(x)\ge R^{1+a}$.  Choosing
$R_0=(2C_0)^{1/(1+a)}$ forces at least one point of $\supp\phi_t$ to lie in
$B_{R_0}(0)$ for every $t\ge0$.
\end{proof}

\begin{lemma}[Two elementary diameter estimates]\label{lem:diameter-geometry}
Let \(0<s\le1\), and set $ K_s(x)=\nabla |x|^{1+s}
        =
        (1+s)|x|^{s-1}x .$
For \(X_1,X_2,y\in\R^d\), define
\begin{equation}\label{eq:Gs-def}
        G_s(X_1,X_2;y)
        :=
        \ip{K_s(X_1-y)-K_s(X_2-y)}{X_1-X_2}.
\end{equation}
Let $ D:=|X_1-X_2|,$ then:

\begin{enumerate}[label=(\alph*)]
\item There exists \(C_s<\infty\) such that
\begin{equation}\label{eq:G-upper-global}
        0\le G_s(X_1,X_2;y)\le C_sD^{1+s}
\end{equation}
for all \(X_1,X_2,y\in\R^d\).

\item Suppose that there exists \(Z\in\R^d\) such that
$|X_1-Z|\le D, |X_2-Z|\le D, |Z|\le R_0 .$
Then, for every \(A<\infty\), there exist constants
\(D_0=D_0(A,R_0,s)\) and \(c_s=c_s(A,R_0,s)>0\) such that, if
\(D\ge D_0\), then
\begin{equation}\label{eq:G-lower-background-ball}
        G_s(X_1,X_2;y)\ge c_sD^{1+s},
        \qquad |y|\le A .
\end{equation}
\end{enumerate}
\end{lemma}

\begin{proof}
The nonnegativity in \eqref{eq:G-upper-global} follows from convexity of
\(|x|^{1+s}\): $ \ip{K_s(u)-K_s(v)}{u-v}\ge0.$

For the upper bound, if \(0<s<1\), then \(K_s\) is globally \(s\)-Hölder:
\begin{equation}\label{eq:Ks-holder}
        |K_s(u)-K_s(v)|\le C_s|u-v|^s .
\end{equation}
Thus $ G_s(X_1,X_2;y)
        \le
        |K_s(X_1-y)-K_s(X_2-y)|\,D
        \le C_sD^{1+s}.$
If \(s=1\), then \(K_1(x)=2x\), and hence $ G_1(X_1,X_2;y)=2|X_1-X_2|^2=2D^2.$
This proves \eqref{eq:G-upper-global}.

Now prove the lower bound.  Let $ e:=\frac{X_1-X_2}{D},
        m:=\frac{X_1+X_2}{2}.$
Then $ X_1=m+\frac D2e,
        X_2=m-\frac D2e.$
The assumptions \(|X_i-Z|\le D\) imply, by the parallelogram identity,
\[
        |m-Z|^2+\frac{D^2}{4}
        =
        \frac{|X_1-Z|^2+|X_2-Z|^2}{2}
        \le D^2.
\]
Therefore $ |m-Z|\le \frac{\sqrt3}{2}D.$
If \(|y|\le A\), then $ \left|\frac{m-y}{D}\right|
        \le
        \frac{|m-Z|}{D}+\frac{|Z|+|y|}{D}
        \le
        \frac{\sqrt3}{2}+\frac{R_0+A}{D}.$ 
Choose $\rho\in\left(\frac{\sqrt3}{2},1\right),$
and then \(D_0\) so large that $\frac{\sqrt3}{2}+\frac{R_0+A}{D_0}\le \rho.$
Set $q:=\frac{m-y}{D}$, by homogeneity,
\begin{equation}\label{eq:G-scaled}
        D^{-(1+s)}G_s(X_1,X_2;y)
        =
        \ip{K_s(q+e/2)-K_s(q-e/2)}{e}.
\end{equation}
Define $   H_s(q,e):=
        \ip{K_s(q+e/2)-K_s(q-e/2)}{e}.$
On the compact set $\mathcal K_\rho:=\{(q,e): |q|\le\rho,\ |e|=1\},$
we claim that \(H_s(q,e)>0\).  Indeed, for fixed \(q,e\), let
$ f(\tau):=|q+\tau e|^{1+s},
 -\frac12\le\tau\le\frac12.$
Since \(1+s>1\), the function \(z\mapsto |z|^{1+s}\) is strictly convex.
Because \(|q|<1\), the segment $ \{q+\tau e:-1/2\le\tau\le1/2\}$
is not contained in a line segment on which \(|\cdot|^{1+s}\) is affine.  Hence
\(f\) is strictly convex, and therefore $  f'(1/2)-f'(-1/2)>0.$
But $ f'(1/2)-f'(-1/2)
        =
        \ip{K_s(q+e/2)-K_s(q-e/2)}{e}
        =
        H_s(q,e).$
Thus \(H_s>0\) on \(\mathcal K_\rho\).  By compactness,$c_s:=\min_{\mathcal K_\rho}H_s>0.$
Combining this with \eqref{eq:G-scaled} gives $  G_s(X_1,X_2;y)\ge c_sD^{1+s}$
for all \(D\ge D_0\) and all \(|y|\le A\).
\end{proof}

\begin{remark}
The restriction \(s>0\) in Lemma~\ref{lem:diameter-geometry} is essential.
For \(s=0\), the lower bound \eqref{eq:G-lower-background-ball} is false.  For
example, in one dimension, take
$ X_1=D, X_2=0, y=-1.$
Then $    K_0(X_1-y)-K_0(X_2-y)
        =
        \sgn(D+1)-\sgn(1)
        =
        0,$
and hence $ G_0(X_1,X_2;y)=0.$
Thus no positive lower bound of order \(D\) can hold for \(s=0\).  The endpoint
\(a=r=0\) is treated separately in Theorem~\ref{thm:support}.
\end{remark}

\begin{lemma}[Near constancy of \(G_s\) on fixed balls]\label{lem:G-near-constant}
Let \(0<s\le1\), and assume the hypotheses of
Lemma~\ref{lem:diameter-geometry}.  For every \(A<\infty\),
\begin{equation}\label{eq:G-near-constant}
        \sup_{|y|\le A}
        \left|
        \frac{G_s(X_1,X_2;y)-G_s(X_1,X_2;0)}{D^{1+s}}
        \right|
        \longrightarrow 0
        \qquad\text{as }D\to\infty .
\end{equation}
\end{lemma}

\begin{proof}
If \(0<s<1\), then by \eqref{eq:Ks-holder},
\[
\begin{aligned}
|G_s(X_1,X_2;y)-G_s(X_1,X_2;0)|
&\le
D|K_s(X_1-y)-K_s(X_1)|  \\
&\quad+
D|K_s(X_2-y)-K_s(X_2)|  \\
&\le C_sD|y|^s .
\end{aligned}
\]
Thus, for \(|y|\le A\), $
        |G_s(X_1,X_2;y)-G_s(X_1,X_2;0)|
        \le C_{A,s}D.$
Dividing by \(D^{1+s}\) gives \(C_{A,s}D^{-s}\to0\).

If \(s=1\), then \(K_1(x)=2x\), so \(G_1(X_1,X_2;y)\) is independent of \(y\).
Hence the left-hand side in \eqref{eq:G-near-constant} is identically zero.
\end{proof}

\begin{theorem}[Uniform compactness of the support]\label{thm:support}
\begin{revision}
Let \(a,r\in[0,1)\), and suppose that Assumption~\ref{ass:standing} holds with $
p\ge 1+\max\{a,r\},$ thus the conclusions of Theorem~\ref{thm:wellposed} and Proposition~\ref{prop:energy-diss} hold. Assume further that \(\phi_0\) is compactly supported. If either $
a>r\text{ and }
\omega(\R^d)>0,
$ or $
a=r\text{ and }
\omega(\R^d)>1,
$ then there exists \(R_*<\infty\) such that
\begin{equation}\label{eq:uniform-support}
\supp \phi_t\subset B_{R_*}(0)
\qquad\text{for all } t\ge0.
\end{equation}
\end{revision}\end{theorem}
\begin{remark}
For \(a\ge r>1\), this result is no longer valid.  Indeed, Example~\ref{ex:particle-escape-mass-larger-than-one} in Appendix~\ref{sec:apx} shows that
particles may escape to infinity in this regime.  The reason is that the global
estimate \eqref{eq:G-upper-global}, which is essential in the proof above,
fails when \(s>1\).
\end{remark}

\begin{proof}
We prove the claim in three cases.

\smallskip
\noindent
\emph{Case 1: \(a=r=0\).}
Let $ M:=\omega(\R^d)>1.$
Choose \(A<\infty\) and \(0<\eta<1\) such that $ m_A:=\omega(B_A(0))$
satisfies
\begin{equation}\label{eq:fail0}
    (1-\eta)m_A-(M-m_A)>1.  
\end{equation}This is possible because \(m_A\uparrow M>1\).  Let
 $ R_A:=\frac{2A}{\eta}.$
If \(|x|\ge R_A\) and \(e=x/|x|\), then for every \(y\in B_A(0)\),
\[
        e\cdot K_0(x-y)
        =
        \frac{|x|-e\cdot y}{|x-y|}
        \ge
        \frac{|x|-A}{|x|+A}
        \ge 1-\eta.
\]
For \(y\notin B_A(0)\), we use \(e\cdot K_0(x-y)\ge -1\).  Hence
 $ e\cdot(K_0*\omega)(x)
        \ge
        (1-\eta)m_A-(M-m_A)
        >1.$
On the other hand, since \(\phi_t\) is a probability measure,
 $ e\cdot(K_0*\phi_t)(x)\le1.$
Therefore, for \(|x|\ge R_A\),
$e\cdot\bigl[-K_0*\omega(x)+K_0*\phi_t(x)\bigr]<0.$
Thus the radial velocity points strictly inward outside \(B_{R_A}(0)\). 
{Let $ R(t):=\max_{x\in\operatorname{supp}\phi_t}|x|.$
The characteristic velocity is $ \dot X_t(x)=-K_0*\omega(X_t(x))+K_0*\phi_t(X_t(x)).$
From the preceding estimate, if \(|x|\ge R_A\) and \(e=x/|x|\), then $e\cdot\bigl(-K_0*\omega(x)+K_0*\phi_t(x)\bigr)<0.$
Thus every characteristic located outside \(B_{R_A}(0)\) has inward radial
velocity.

Since $R(t)=\max_{x\in\operatorname{supp}\phi_0}|X_t(x)|,$
the standard maximum rule for upper Dini derivatives gives
\[
        D^+R(t)
        \le
        \max_{\substack{x\in\operatorname{supp}\phi_0\\ |X_t(x)|=R(t)}}
        \frac{X_t(x)}{|X_t(x)|}\cdot \dot X_t(x),
\]
see Remark \ref{rmk:dini} in the Appendix.
If \(R(t)\ge R_A\), then every active point \(X_t(x)\) in the maximum satisfies
\(|X_t(x)|\ge R_A\). Hence the radial velocity estimate implies
 $ D^+R(t)\le 0,\text {whenever }R(t)\ge R_A.$
Consequently, $ R(t)\le \max\{R(0),R_A\}, t\ge0.$
}
This proves \eqref{eq:uniform-support} in the endpoint \(a=r=0\).

\smallskip
\noindent
\emph{Case 2: \(a>r\).}
By Proposition~\ref{prop:moment-anchor}, there exists \(R_0<\infty\) such that $ \supp\phi_t\cap B_{R_0}(0)\ne\emptyset \text{ for every }t\ge0.$
{We next derive the differential inequality for the diameter.  Let $  D(t):=\operatorname{diam}(\operatorname{supp}\phi_t).$
Since \(\phi_t=(X_t)_{\#}\phi_0\), we may write
\[
        D(t)^2
        =
        \max_{x,z\in\operatorname{supp}\phi_0}
        |X_t(x)-X_t(z)|^2.
\]
The standard maximum rule for upper Dini derivatives yields
\[
        D^+D(t)^2
        \le
        \max_{\substack{x,z\in\operatorname{supp}\phi_0\\
        |X_t(x)-X_t(z)|=D(t)}}
        2\left\langle
        \dot X_t(x)-\dot X_t(z),
        X_t(x)-X_t(z)
        \right\rangle,
\]
see Remark \ref{rmk:dini} in the Appendix.
Choose a diameter pair and denote $ X_1:=X_t(x), X_2:=X_t(z), |X_1-X_2|=D(t).$
Along characteristics, $\dot X_i=-K_a*\omega(X_i)+K_r*\phi_t(X_i),
         i=1,2.$
Therefore
\[
\begin{aligned}
\dot X_1-\dot X_2
&=
-\bigl(K_a*\omega(X_1)-K_a*\omega(X_2)\bigr) 
+\bigl(K_r*\phi_t(X_1)-K_r*\phi_t(X_2)\bigr).
\end{aligned}
\]
Expanding the convolutions, we get
\[
K_a*\omega(X_1)-K_a*\omega(X_2)
=
\int_{\mathbb R^d}
\bigl(K_a(X_1-y)-K_a(X_2-y)\bigr)\omega(y)\,dy,
\]
and
\[
K_r*\phi_t(X_1)-K_r*\phi_t(X_2)
=
\int_{\mathbb R^d}
\bigl(K_r(X_1-y)-K_r(X_2-y)\bigr)\,d\phi_t(y).
\]
By the definition
$  G_s(X_1,X_2;y)
        :=
        \left\langle
        K_s(X_1-y)-K_s(X_2-y),
        X_1-X_2
        \right\rangle ,$
we obtain
\begin{equation}\label{eq:diam-dini-detailed}
\begin{aligned}
D^+D(t)^2
&\le
-2\int_{\mathbb R^d}G_a(X_1,X_2;y)\omega(y)\,dy +2\int_{\mathbb R^d}G_r(X_1,X_2;y)\,d\phi_t(y).
\end{aligned}
\end{equation}
}

Choose \(Z_t\in\supp\phi_t\cap B_{R_0}(0)\).  Since \(X_1,X_2\) are a diameter
pair, $|X_i-Z_t|\le D(t), i=1,2.$
Choose \(A<\infty\) such that $  m_A:=\omega(B_A(0))>0.$
Since \(a>0\), Lemma~\ref{lem:diameter-geometry} gives, for all sufficiently large
\(D(t)\), $
        G_a(X_1,X_2;y)\ge c_aD(t)^{1+a}, |y|\le A.$
Thus $ \int_{\R^d}G_a(X_1,X_2;y)\omega(y)\dd y
        \ge c_am_AD(t)^{1+a}.$
For the repulsive term, if \(r>0\), use \eqref{eq:G-upper-global}; if \(r=0\),
use \(|K_0(u)-K_0(v)|\le2\).  In both cases,
$ \int_{\R^d}G_r(X_1,X_2;y)\dd\phi_t(y)
        \le C_rD(t)^{1+r}.$
Therefore
\begin{equation}\label{eq:diam-a-r-final}
        D^+D(t)^2
        \le
        -2c_am_AD(t)^{1+a}
        +
        2C_rD(t)^{1+r}.
\end{equation}
Since \(a>r\), the right-hand side is negative for all sufficiently large
\(D(t)\).  Hence \(D(t)\) cannot cross a fixed large value upward.  Combining
this with the anchoring point \(Z_t\in B_{R_0}(0)\) gives
$ \supp\phi_t\subset B_{R_*}(0)$
for some \(R_*<\infty\).

\smallskip
\noindent
\emph{Case 3: \(a=r=s\in(0,1)\).}
Let $   M:=\omega(\R^d)>1.$
By Proposition~\ref{prop:moment-anchor}, there are \(C_0,R_0<\infty\) such that
$ \sup_{t\ge0}\int_{\R^d}|x|^{1+s}\dd\phi_t(x)\le C_0,
        \supp\phi_t\cap B_{R_0}(0)\ne\emptyset.$
Choose \(\eta>0\) small and then \(A<\infty\) large such that
\begin{equation}\label{eq:equal-tail-choices}
        \omega(B_A(0))\ge M-\eta,
        \qquad
        \sup_{t\ge0}\phi_t(B_A(0)^c)\le\eta.
\end{equation}
The second inequality follows from the uniform moment bound and Markov's
inequality.

Let \(X_1,X_2\) be a diameter pair, \(D=|X_1-X_2|\) and \(G_s(y):=G_s(X_1,X_2;y)\).  By
Lemma~\ref{lem:diameter-geometry}, applied at \(y=0\), there exists
\(c_s>0\) such that, for all sufficiently large \(D\),
$ G_s(0)\ge c_sD^{1+s}.$
Moreover, by Lemma~\ref{lem:G-near-constant}, for every \(\varepsilon>0\),
after increasing the threshold on \(D\) if necessary,
$|G_s(y)-G_s(0)|\le \varepsilon D^{1+s}, |y|\le A.$
Using also \(0\le G_s(y)\le C_sD^{1+s}\), we get
\[
\begin{aligned}
&\int G_s(y)\,d\omega-\int G_s(y)\,d\phi_t\\
&\qquad\qquad\ge
\int_{B_A}G_s(y)\,d\omega
-
\int_{B_A}G_s(y)\,d\phi_t
-
\int_{B_A^c}G_s(y)\,d\phi_t  \\
&\qquad\qquad=
G_s(0)\bigl(\omega(B_A)-\phi_t(B_A)\bigr)
+
\int_{B_A}(G_s(y)-G_s(0))\,d(\omega-\phi_t)
-
\int_{B_A^c}G_s(y)\,d\phi_t  \\
&\qquad\qquad\ge
c_sD^{1+s}\bigl(M-1-\eta\bigr)
-
\varepsilon D^{1+s}\bigl(M+1\bigr)
-
C_s\eta D^{1+s}.
\end{aligned}
\]
Choose first \(\eta>0\) and then \(\varepsilon>0\) small enough so that
\begin{equation}\label{eq:failo}
        c_s(M-1-\eta)-\varepsilon(M+1)-C_s\eta>0.
\end{equation}
This gives $ \int G_s(y)\,d\omega-\int G_s(y)\,d\phi_t
        \ge cD^{1+s}$ 
for all sufficiently large \(D\) and some small $c>0$.  Inserting this into
\eqref{eq:diam-dini-detailed} gives $ D^+D(t)^2\le -2cD(t)^{1+s}<0,$
whenever \(D(t)\) is large.  Hence the diameter is uniformly bounded, and the
anchoring point in \(B_{R_0}(0)\) gives \eqref{eq:uniform-support}.
\end{proof}

{\begin{remark}\label{rem:support-borderline}
   The proof of Theorem~\ref{thm:support} does not cover the borderline case
\(a=r\) and \(\omega(\mathbb R^d)=1\).  In this case, the conditions used
in \eqref{eq:fail0} and \eqref{eq:failo} fail, so the argument no longer
yields a uniform compact-support bound.  This is expected: if \(\omega\) is not
compactly supported, then one should not expect \(\phi_t\) to remain uniformly
compactly supported, since the long-time behavior is expected to satisfy
\(\phi_t\to\omega\).
\end{remark}
}

\section{Characterization of stationary distributions} \label{sec:stationary}

This section is formulated for the nonquadratic attractive-dominant range $ 0\le r\le a<1,$
where the homogeneous kernels may be inverted in the sense of tempered
homogeneous distributions.  The basic facts on homogeneous distributions and Fourier transforms of homogeneous kernels are
standard; see Gelfand--Shilov~\cite[Chapters II--III]{GelfandShilov1964},
Grafakos~\cite[Section 2.4.3]{Grafakos2014}, and
Stein--Weiss~\cite[Chapter V]{SteinWeiss1971}.  

We use the Fourier transform \begin{equation}
     \widehat f(\xi)=\int_{\R^d}e^{-ix\cdot\xi}f(x)\dd x.
\end{equation}
All Fourier identities involving homogeneous kernels are understood in
\(\calS'(\R^d)\).

For \(s\in[0,1)\), the homogeneous distribution $  \psi_s(x)=|x|^{1+s}$
has Fourier transform
\begin{equation}\label{eq:FT-psi}
        \widehat{\psi_s}
        =
        \gamma_{d,s}\operatorname{fp}|\xi|^{-d-1-s}
        \qquad\text{in }\calS'(\R^d),
\end{equation}
where \(\gamma_{d,s}\ne0\) depends only on \(d,s\), and on the Fourier
normalization,  $
    \left\langle
        \operatorname{fp}|\xi|^{-d-1-s},\phi
\right\rangle
:=
\lim_{\varepsilon\downarrow0}
\left[
        \int_{|\xi|>\varepsilon}
        |\xi|^{-d-1-s}\phi(\xi)\dd\xi
        -
        \frac{|\mathbb S^{d-1}|}{1+s}
        \varepsilon^{-1-s}\phi(0)
\right],
$
see \cite[Theorem 2.4.6]{Grafakos2014}.

For \(\alpha>0\), the fractional Laplacian is defined by the homogeneous
Fourier multiplier
\begin{equation}
    \widehat{(-\Delta)^\alpha f}(\xi)
        =
        |\xi|^{2\alpha}\widehat f(\xi),
\end{equation}
whenever the right-hand side defines a tempered distribution.  In particular,
for $ \alpha=\frac{d+1+s}{2},$
one has $ \widehat{(-\Delta)^{(d+1+s)/2}f}(\xi)
        =
        |\xi|^{d+1+s}\widehat f(\xi).$

We now define the inverse convolution operator.  It is not introduced as an
unrestricted Fourier multiplier on all of \(\calS'(\R^d)\).  Instead, it is defined on the range of convolution by \(\psi_s\).  For \(s\in[0,1)\), set
\begin{equation}
     \mathfrak P_s
        :=
        \left\{
        f\in\calS'(\R^d):
        \text{ there exists a finite measure }\mu
        \text{ such that } f=\psi_s*\mu
        \text{ in }\calS'(\R^d)
        \right\}.
\end{equation}
This representing measure is unique.  Indeed, suppose
$  \psi_s*\mu_1=\psi_s*\mu_2\text{ in }\calS'(\R^d),$
and set \(\nu:=\mu_1-\mu_2\), then $ \psi_s*\nu=0.$
Taking Fourier transforms gives $ \widehat{\psi_s}\,\widehat\nu=0.$
Using \eqref{eq:FT-psi}, this becomes $  \gamma_{d,s}\operatorname{fp}|\xi|^{-d-1-s}\,\widehat\nu=0
        \text{ in }\calS'(\R^d).$
On \(\R^d\setminus\{0\}\), the finite-part distribution
\(\operatorname{fp}|\xi|^{-d-1-s}\) coincides with the ordinary smooth
function \(|\xi|^{-d-1-s}\), which is strictly positive.  Hence
$ \widehat\nu(\xi)=0,
        \text{ for every }\xi\ne0.$
Since \(\nu\) is a finite signed measure, $ \widehat\nu(\xi)=\int_{\R^d}e^{-ix\cdot\xi}\,d\nu(x)$
is continuous on \(\R^d\).  Therefore
$ \widehat\nu(0)
        =
        \lim_{\xi\to0,\ \xi\ne0}\widehat\nu(\xi)
        =
        0,$
thus \(\widehat\nu\equiv0\) on \(\R^d\). Finally, the Fourier transform is injective on finite signed measures.  Indeed,
for every \(\phi\in\calS(\R^d)\), Fourier inversion and Fubini give
\[
\begin{aligned}
\int_{\R^d}\phi(x)\,d\nu(x)
&=
(2\pi)^{-d}
\int_{\R^d}
\widehat\phi(\xi)
\left(\int_{\R^d}e^{ix\cdot\xi}\,d\nu(x)\right)
d\xi  \\
&=
(2\pi)^{-d}
\int_{\R^d}
\widehat\phi(\xi)\widehat\nu(-\xi)\,d\xi
=
0.
\end{aligned}
\]
Thus \(\nu\) annihilates \(\calS(\R^d)\), hence also
\(C_c^\infty(\R^d)\).  Therefore \(\nu=0\) as a finite signed measure, and the
representing measure is unique.

For \(f=\psi_s*\mu\in\mathfrak P_s\), define $ \calL_s f:=\mu,$
thus $ \calL_s(\psi_s*\mu)=\mu$ and $
        \calL_s\psi_s=\delta_0\text{ in }\calS'(\R^d).$
{On this admissible class, \(\calL_s\) is defined through the Fourier multiplier
\begin{equation}\label{eq:L-symbol}
        \widehat{\calL_s f}(\xi)
        :=
        \gamma_{d,s}^{-1}|\xi|^{d+1+s}\widehat f(\xi).
\end{equation}
}
Equivalently, 
\begin{equation}\label{eq:L-fraclap}
        \calL_s=c_{d,s}(-\Delta)^{(d+1+s)/2},
\end{equation}
with \(c_{d,s}=\gamma_{d,s}^{-1}\neq 0\).

We shall apply \(\calL_r\) to potentials generated by \(\psi_a\), where \(0\le r\le a<1\).
Let $\beta:=a-r\in[0,1)$, if \(\mu\) is a finite measure and \(\psi_a*\mu\) is a tempered distribution,
we have $  \widehat{\calL_r(\psi_a*\mu)}(\xi)=
        \frac{\gamma_{d,a}}{\gamma_{d,r}}
        |\xi|^{r-a}\widehat\mu(\xi)
        =
        \frac{\gamma_{d,a}}{\gamma_{d,r}}
        |\xi|^{-\beta}\widehat\mu(\xi).$
This is well-defined as a tempered distribution.  Indeed, since
\(\beta\in[0,1)\), the multiplier \(|\xi|^{-\beta}\) is locally integrable at
the origin, and it has no growth at infinity.  Hence
$|\xi|^{-\beta}\widehat\mu(\xi)$
defines a tempered distribution because \(\widehat\mu\) is bounded and
continuous.

When \(a=r\), one has \(\beta=0\), and the definition reduces to
$  \calL_a(\psi_a*\mu)=\mu.$
When \(a>r\), the multiplier \(|\xi|^{-(a-r)}\) is the Fourier multiplier of the
Riesz potential of order \(a-r\).  Equivalently,
$ \calL_r(\psi_a*\mu)
        =
        C_{d,a,r}\,\mu*|x|^{-(d-a+r)}$
in \(\calS'(\R^d)\), with the convolution understood through the Fourier
definition.

In the following, we will set
\begin{equation}\label{eq:F-def-stationary}
        F:=\psi_a*\omega-\psi_r*\phi .
\end{equation}
\begin{proposition}[Fourier inversion of the stationary equation]
\label{prop:fourier-inversion}
Let \(0\le r\le a<1\).  Let \(\omega\) be a finite nonnegative measure and let
\(\phi\) be a finite measure and $\psi_r*\phi\in\mathfrak P_r$.
Then, in \(\calS'(\R^d)\),
\begin{equation}\label{eq:phi-inverse}
        \phi=Q_{a,r}-\calL_rF,
        \qquad
        Q_{a,r}:=\calL_r(\psi_a*\omega).
\end{equation}
Furthermore,
\begin{equation}\label{eq:Q-symbol}
        \widehat{Q_{a,r}}(\xi)
        =
        \frac{\gamma_{d,a}}{\gamma_{d,r}}
        |\xi|^{r-a}\widehat\omega(\xi).
\end{equation}
When \(a>r\), this is equivalently the Riesz-potential distribution
\begin{equation}\label{eq:Q-riesz}
        Q_{a,r}
        =
        C_{d,a,r}\,\omega*|x|^{-(d-a+r)} .
\end{equation}
When \(a=r\), one has $Q_{a,a}=\omega.$
\end{proposition}

\begin{proof}
Since \(\psi_r*\phi\in\mathfrak P_r\), the definition of \(\calL_r\) gives $ \calL_r(\psi_r*\phi)=\phi.$
Therefore, $  \calL_rF
        =
        \calL_r(\psi_a*\omega)-\calL_r(\psi_r*\phi)
        =
        Q_{a,r}-\phi.$
Rearranging gives \eqref{eq:phi-inverse}. From the paragraph after \eqref{eq:L-fraclap}, we have
$ \widehat{Q_{a,r}}(\xi)
        =
        \frac{\gamma_{d,a}}{\gamma_{d,r}}
        |\xi|^{r-a}\widehat\omega(\xi),$
which proves \eqref{eq:Q-symbol}.  If \(a>r\), then
$ |\xi|^{r-a}=|\xi|^{-(a-r)}$
is the Fourier multiplier of the Riesz potential of order \(a-r\).  Its
physical-space kernel is $  |x|^{(a-r)-d}=|x|^{-(d-a+r)}.$
This gives \eqref{eq:Q-riesz}.  If \(a=r\), then
$ |\xi|^{r-a}=1,
        \gamma_{d,a}/\gamma_{d,r}=1,$
so \(Q_{a,a}=\omega\).
\end{proof}

\begin{theorem}[Free-boundary characterization of zero-flux stationary states]
\label{thm:stationary-characterization}
Let \(0\le r\le a<1\), \(\omega\) be a finite nonnegative measure. Assume $\phi$ is a zero-flux stationary state, $\psi_r*\phi\in\mathfrak P_r$ and \(F\) given by \eqref{eq:F-def-stationary} has gradient that is
\(\phi\)-integrable, then we have \begin{equation}\label{eq:gradF-on-support}
        \grad F=0
        \qquad \phi\text{-a.e. on }S,
\end{equation}
here  $ S:=\supp\phi,$
and $\phi$ can be represented by
\begin{equation}\label{eq:stationary-inverse}
        \phi=Q_{a,r}-\calL_rF,
        \qquad
        Q_{a,r}=\calL_r(\psi_a*\omega).
\end{equation}
Conversely, if the free-boundary pair \((\Omega,{F})\) with $\psi_a*\omega-F\in \mathfrak P_r$ satisfies
\begin{alignat}{2}
    \calL_r F-\calL_r(\psi_a*\omega)
    &=0
    &&\qquad \text{in }\Omega,
    \label{eq:free1}\\
       \calL_r(\psi_a*\omega)-\calL_r F
    &\geq 0
    &&\qquad \text{as a finite measure with support $\Omega^c$},
    \label{eq:free3}\\
    \bigl(\calL_r(\psi_a*\omega)-\calL_r F\bigr)(\R^d)
    &=1
    &&
    \label{eq:free4}\\
     \grad F
    &=0
    &&\qquad (\calL_r(\psi_a*\omega)-\calL_r F)\text{-a.e. on }\Omega^c.
    \label{eq:free2}
\end{alignat}
then $\phi=\calL_r(\psi_a*\omega)-\calL_r F$ is a zero-flux stationary state. 

\end{theorem}
{\begin{remark}\label{rem:nonuniqueness}
   Zero-flux stationary states need not be unique.  For example, let
\(d=1\), \(a=r=0\), and $\omega=\one_{[-1,1]} .$
Then both $  \phi=\delta_0
 \text{ and }
        \phi=\one_{[-1/2,1/2]}\,\dd x$
are zero-flux stationary probability measures.  Indeed, with
\(K_0=\operatorname{sgn}\), we have
$ K_0*\omega(0)=0=K_0*\delta_0(0),$
so \(\delta_0\) is stationary.  Moreover, for
\(x\in[-1/2,1/2]\), we have $ K_0*\omega(x)=2x=K_0*\phi(x)$
 hence $ \phi\,(K_0*\omega-K_0*\phi)=0.$
Thus the zero-flux stationary state is not unique.
\end{remark}

}
\begin{proof}
First, we have $ \grad F
        =
        K_a*\omega-K_r*\phi$
in the sense of distributions, which is \(\phi\)-integrable.  The zero-flux stationary condition is $ \phi\bigl(K_a*\omega-K_r*\phi\bigr)=0$
in the sense of vector-valued distributions.  Using the identity above, this is
the same as $\phi\,\grad F=0.$
Since \(\grad F\) is \(\phi\)-integrable, the condition
\(\phi\,\grad F=0\) is equivalent to $ \grad F=0,  \phi\text{-a.e.}$  Thus the last
condition may be written as $  \grad F=0, \phi\text{-a.e. on }S,$
which proves \eqref{eq:gradF-on-support}. The form of $\phi$ in \eqref{eq:stationary-inverse} is derived by \eqref{eq:phi-inverse}.

Conversely, suppose that the pair \((\Omega,F)\) satisfies
\eqref{eq:free1}--\eqref{eq:free4}. Then
\(\phi:=\calL_r(\psi_a*\omega)-\calL_rF\) is a finite probability
measure supported in \(\Omega^c\). Since
\(\psi_a*\omega-F\in\mathfrak P_r\), there exists a finite measure
\(\mu\) such that
\(\psi_a*\omega-F=\psi_r*\mu\). By Proposition~\ref{prop:fourier-inversion},
the representing measure is
\(\mu=\calL_r(\psi_a*\omega-F)
=\calL_r(\psi_a*\omega)-\calL_rF=\phi\).
Consequently, \(F=\psi_a*\omega-\psi_r*\phi\). By
\eqref{eq:free2}, we therefore have
\(\phi(K_a*\omega-K_r*\phi)=\phi\grad F=0\) in the sense of
vector-valued distributions.
\end{proof}

\begin{proposition}[Green-function representation]
\label{prop:green}
Let \(0\le r\le a<1\), and let \((\Omega,F)\) be a free-boundary pair
satisfying the conditions of
Theorem~\ref{thm:stationary-characterization}, with associated measure
\(\phi:=Q_{a,r}-\calL_rF\). Assume in addition that
\(\Omega^c=\supp\phi=\overline D\), where \(D\) is a connected Lipschitz
domain, and that \(\nabla F\) is continuous in a neighborhood of
\(\Omega^c\).

Set \(\alpha_r:=(d+1+r)/2\), and choose an integer \(p\ge0\) and
\(s\in(0,1]\) such that
\begin{equation}\label{eq:alpha-ps}
\alpha_r=p+s.
\end{equation}
If \(\alpha_r\in\N\), we adopt the convention \(s=1\) and
\(p=\alpha_r-1\). The condition \eqref{eq:free2} and the above assumptions imply
that \(F=C\) on \(\Omega^c\) for some constant \(C\). Consequently,
\((-\Delta)^pF=0\) in \(D=\operatorname{int}(\Omega^c)\) if \(p\ge1\),
whereas \((-\Delta)^pF=F=C\) on \(\Omega^c\) if \(p=0\). Set \(\bar H:=(-\Delta)^p(F-C)\), and assume that \(\bar H\) belongs to a function class in which the factorization
\(\calL_rF=c_{d,r}(-\Delta)^s\bar H\) is valid. Then \(\bar H=0\) in \(D\);
since \(\partial D\) has zero Lebesgue measure, it follows that
\(\bar H=0\) almost everywhere in \(\Omega^c\).

Assume that, in the same function class, the fractional Dirichlet problem
\begin{equation}\label{eq:fractional-dirichlet}
(-\Delta)^sH=K_\Omega:=c_{d,r}^{-1}Q_{a,r}
\quad\text{in }\Omega,
\qquad
H=0\quad\text{a.e. in }\Omega^c
\end{equation}
has a unique solution and that this solution admits the Green representation
\(H(x)=\int_\Omega G_\Omega^s(x,y)\,\dd K_\Omega(y)\) for \(x\in\Omega\),
where \(G_\Omega^s\) denotes the Green kernel associated with
\eqref{eq:fractional-dirichlet}. If \(K_\Omega\) has a density, the integral
is understood as
\(\int_\Omega G_\Omega^s(x,y)K_\Omega(y)\,\dd y\). Then \(\phi\) admits the
representation
\begin{equation}\label{eq:phi-green}
\phi
=
Q_{a,r}-c_{d,r}(-\Delta)^sH.
\end{equation}
\end{proposition}

{\begin{remark}
In the above proposition, we assumed that a free-boundary pair
\((\Omega,F)\) satisfies \eqref{eq:free1}--\eqref{eq:free2}.
However, the construction in \eqref{eq:phi-green} requires only the
domain \(\Omega\): once \(\Omega\) is prescribed, the corresponding
candidate measure \(\phi\) can be reconstructed. This observation is
particularly useful when the geometry of the support is known in
advance. For example, when \(d=1\), one may assume
\(\Omega^c=[A,B]\); when \(\omega\) is radially symmetric, one may
instead assume \(\Omega^c=B_R(0)\). The remaining free-boundary
conditions can then be used to determine \(A\), \(B\), or \(R\). In
some cases, these parameters can be obtained analytically, as
illustrated by the examples in Section~\ref{sec:examples}.We also emphasize that, for a general background measure \(\omega\),
determining the unknown domain \(\Omega\) is in general a highly
nontrivial problem.
\end{remark}}
{
\begin{remark}
Some of the assumptions in Proposition~\ref{prop:green}, particularly the uniqueness assumption for the fractional Dirichlet problem, may be difficult to verify, as they rely on sophisticated theory for the fractional Laplacian. Nevertheless, in many applications, we use \eqref{eq:phi-green} directly to construct a candidate measure \(\phi\) and then verify that it indeed defines a zero-flux stationary state. This is the approach adopted in Section~\ref{sec:examples}.
\end{remark}

}

\begin{proof}
By the zero-flux condition in
Theorem~\ref{thm:stationary-characterization}, we have
\(\nabla F=0\), \(\phi\)-almost everywhere on \(\Omega^c\). Since
\(\nabla F\) is continuous and \(\Omega^c=\supp\phi\), it follows that
\(\nabla F=0\) everywhere on \(\Omega^c\). Indeed, otherwise
\(\nabla F\) would be nonzero in a neighborhood of some point of
\(\supp\phi\), and this neighborhood would have positive \(\phi\)-measure.
Since \(D\) is connected, \(F\) is constant on \(D\), and continuity then
gives \(F=C\) on \(\Omega^c=\overline D\).

Consequently, \(F-C=0\) in \(D\), and hence
\(\bar H=(-\Delta)^p(F-C)=0\) in \(D\).
Since \(\bar H\) is represented by a function and \(\partial D\) has zero
Lebesgue measure, we obtain \(\bar H=0\) almost everywhere in
\(\Omega^c\).

In \(\Omega\), the free-boundary equation gives
\(\calL_rF=Q_{a,r}\). Using the assumed factorization and the fact that
fractional Laplacians annihilate constants, we obtain
\((-\Delta)^s\bar H=c_{d,r}^{-1}Q_{a,r}=K_\Omega\) in \(\Omega\).
Thus \(\bar H\) is a solution of
\eqref{eq:fractional-dirichlet} in the prescribed function class.
By uniqueness, \(\bar H=H\), where \(H\) is given by the stated Green
representation. Finally,
\(\phi=Q_{a,r}-\calL_rF
=Q_{a,r}-c_{d,r}(-\Delta)^s\bar H
=Q_{a,r}-c_{d,r}(-\Delta)^sH\), which proves
\eqref{eq:phi-green}.
\end{proof}

{\begin{remark}
    When \(s=1\), the problem is local, since \(\Delta\) is a local operator.  When
\(s\in(0,1)\), the problem is nonlocal, since \((-\Delta)^s\) is a nonlocal
operator.

The decomposition in \eqref{eq:alpha-ps} is not essential.  Using
polyharmonic variants of Boggio's formula, one could instead take
\(p=0\) and \(s=\alpha_r\).  However, the decomposition
\eqref{eq:alpha-ps} is expected to lead to a simpler computation.

\end{remark}}
\begin{example}[Boggio kernel for balls and exterior balls]
\label{ex:boggio}
For \(0<s<1\) and \(\Omega=B_R(0)\), Boggio's formula gives the Green kernel
\begin{equation}\label{eq:boggio-ball}
        G_{B_R}^s(x,y)
        =\kappa_{d,s}|x-y|^{-(d-2s)}
        \int_0^{\eta_R(x,y)}\frac{t^{s-1}}{(1+t)^{d/2}}\dd t,
\end{equation}
where
\[
        \eta_R(x,y)=\frac{(R^2-|x|^2)(R^2-|y|^2)}{R^2|x-y|^2}.
\]
Boggio's formula and its fractional and polyharmonic variants are classical;
see Boggio~\cite{Boggio1905}, Bogdan--Byczkowski~\cite{BogdanByczkowski1999},
Dyda~\cite{Dyda2012}, and the fractional polyharmonic~(that is $s\geq 1$) treatment of
Dipierro--Grunau~\cite{dipierro2017boggio}.  For \(\Omega=\overline{B_R(0)}^c\), the
Kelvin transform gives the analogous exterior-ball formula with
\[
        \eta_R(x,y)=\frac{(|x|^2-R^2)(|y|^2-R^2)}{R^2|x-y|^2},
        \qquad x,y\in \Omega.
\]
\end{example}

\begin{table}[htbp]
\centering
\begin{tabular}{c|c}
\hline
Constant & Explicit value \\
\hline
\(\gamma_{d,s}\)
&
\(\displaystyle
\frac{2^{d+1+s}\pi^{d/2}
\Gamma\left(\frac{d+1+s}{2}\right)}
{\Gamma\left(-\frac{1+s}{2}\right)}
\)
\\[3mm]
\(c_{d,s}=\gamma_{d,s}^{-1}\)
&
\(\displaystyle
\frac{\Gamma\left(-\frac{1+s}{2}\right)}
{2^{d+1+s}\pi^{d/2}
\Gamma\left(\frac{d+1+s}{2}\right)}
\)
\\[3mm]
\(\gamma_{d,a}/\gamma_{d,r}\)
&
\(\displaystyle
2^{a-r}
\frac{
\Gamma\left(\frac{d+1+a}{2}\right)
\Gamma\left(-\frac{1+r}{2}\right)}
{
\Gamma\left(-\frac{1+a}{2}\right)
\Gamma\left(\frac{d+1+r}{2}\right)}
\)
\\[3mm]
\(C_{d,a,r}\)
&
\(\displaystyle
\frac{
\Gamma\left(\frac{d-a+r}{2}\right)
\Gamma\left(\frac{d+1+a}{2}\right)
\Gamma\left(-\frac{1+r}{2}\right)}
{
\pi^{d/2}
\Gamma\left(\frac{a-r}{2}\right)
\Gamma\left(-\frac{1+a}{2}\right)
\Gamma\left(\frac{d+1+r}{2}\right)}
\)
\\[3mm]
\(\kappa_{d,s}\)
&
\(\displaystyle
\frac{\Gamma(d/2)}
{2^{2s}\pi^{d/2}\Gamma(s)^2}
\)
\\
\hline
\end{tabular}
\caption{Constants used in Section~\ref{sec:stationary}.}
\label{tab:section6-constants}
\end{table}

\section{Explicit stationary states computed from the stationary characterization}\label{sec:examples}

We arrange the examples by dimension.  The common principle is the same in each case: compute the stationary measure from the stationary-characterization formula~\eqref{eq:phi-green} first, and only afterwards verify the zero-flux condition or compare with the long-time dynamics.  In particular, the numerical curves labelled as theoretical stationary states are computed from the explicit formulas in Proposition \ref{prop:1d-ar} for the one-dimensional case and Proposition \ref{prop:2d-general-disk} for the two-dimensional case, rather than by comparing with the background density. The MATLAB/GNU OCTAVE code, developed with the assistance of AI, and the accompanying implementation details and further refinement study are available at \url{https://github.com/lukangsun/attraction-repulsion}.

\subsection{One-dimensional interval formula and numerical verification}

For one-dimensional zero-flux stationary states, we obtain the following representation.

\begin{proposition}[Direct one-dimensional interval representation for $0\leq r\leq a<1$]
\label{prop:1d-ar}
Let \(d=1\), let \(0\le r\le a<1\), and fix a compact interval \(I=[A,B]\).
Set
\begin{equation}\label{eq:Ga-def}
        G_a:=D^2V=
        \begin{cases}
        2\omega, & a=0,\\[1mm]
        a(1+a)|\cdot|^{a-1}*\omega, & 0<a<1.
        \end{cases}
\end{equation}
Let \(\phi_I\) be a zero-flux stationary probability supported on \(I\).
If \(r=0\), assume moreover the no-endpoint-atom
ansatz $\phi_I(\{A\})=\phi_I(\{B\})=0.$
Then \(\phi_I\) is necessarily the measure \(\phi_I\) given by
\begin{equation}\label{eq:phi-I-direct-riesz}
        {
        \phi_I=
        \begin{cases}
        \displaystyle \frac12G_a\,\one_I\,\dd x, & r=0,\\[3mm]
        \displaystyle \text{the solution of }
        r(1+r)\int_A^B |x-y|^{r-1}\dd\phi_I(y)=G_a(x),\quad x\in(A,B),
        &0<r<1.
        \end{cases}}
\end{equation}
When \(0<r<1\), set \(\Omega:= I^c\),
\(\sigma:=r/2\), and \(Q_{a,r}:=c_{1,r}(-\Delta)^{1+\sigma}V\).
A candidate measure for \eqref{eq:phi-I-direct-riesz} is obtained
through the Green-kernel reconstruction~(see also \eqref{eq:phi-green})
\begin{equation}\label{eq:rho-green-explicit}
    \phi_I
    =
    \left(Q_{a,r}-c_{1,r}(-\Delta)^\sigma H\right)\llcorner I,
\end{equation}
where for \(x\in\Omega\),
\(H(x)=c_{1,r}^{-1}\int_\Omega
G_\Omega^\sigma(x,y)Q_{a,r}(y)\,\dd y\), $
    G_\Omega^\sigma(x,y)
    =
    \kappa_{1,\sigma}|x-y|^{r-1}
    \int_0^{\eta_I(x,y)}
    \frac{t^{r/2-1}}{\sqrt{1+t}}\,\dd t, x,y\in\Omega,x\neq y,
$
and $
    \eta_I(x,y)
    =
    \frac{4(x-A)(x-B)(y-A)(y-B)}
    {(B-A)^2|x-y|^2},
$ 
and \(H=0\) on \(I\).


Conversely, if the candidate \(\phi_I\) in \eqref{eq:phi-I-direct-riesz} is a
nonnegative finite measure, satisfies \(\phi_I(I)=1\), and satisfies the
one-point force condition
\begin{equation}\label{eq:interval-centering-r-positive}
        V'(x_*)=(1+r)\int_A^B\sgn(x_*-y)|x_*-y|^r\dd\phi_I(y)
        \qquad\text{for }x_*\in (A,B),
\end{equation}
then \(\phi_I\) is a zero-flux stationary probability supported on \(I\).
For \(r=0\), the formula reduces to
\begin{equation}\label{eq:phi-1d-r0-direct-new}
        \phi_I=\frac12V''\,\one_I\,\dd x,
\end{equation}
and the mass and force conditions are equivalent to
$  V'(A)=-1, V'(B)=1.$
In particular, for \(a=r=0\),
\begin{equation}\label{eq:phi-a0-r0-direct-new}
        \phi_I=\omega\one_I\,\dd x,
        \qquad
        \int_{-\infty}^{A}\omega(z)\dd z
        =\int_B^\infty\omega(z)\dd z
        =\frac{\omega(\R)-1}{2}.
\end{equation}
\end{proposition}

{
\begin{remark}
    The formula for \(\phi_I\) in the case \(d=1\) and \(r=0\) was also obtained
in \cite{difrancesco2015asymptotic}. 
The assumption that \(\phi(\{A\})=\phi(\{B\})=0\) is essential for
the absolutely continuous interval representation above. Indeed, suppose
that \(\omega\) is symmetric about the origin and let
\(V=\psi_a*\omega\). Then \(V'\) is odd. If there exists \(L>0\) such
that \(V'(L)=1/2\), the probability measure
\(\phi=\frac12(\delta_{-L}+\delta_L)\) is a zero-flux stationary state.
In fact, using the convention \(K_0(0)=0\), one has
\(K_0*\phi(L)=1/2\) and \(K_0*\phi(-L)=-1/2\). By the oddness of
\(V'\), the condition \(V'(L)=1/2\) also gives \(V'(-L)=-1/2\).
Consequently, \(V'=K_0*\phi\) holds at both points of
\(\operatorname{supp}\phi=\{-L,L\}\), and hence
\(\phi(V'-K_0*\phi)=0\).

\end{remark}
}

\begin{proof}
Let $F:=V-\psi_r*\phi.$ Since \(F\) is constant on $(A,B)$, we have $ D^2F=0 \text{ in }\calS'((A,B)).$
The distributional second derivative of \(V\) is exactly \eqref{eq:Ga-def},
because $ D^2|x|=2\delta_0,
         D^2|x|^{1+a}=a(1+a)|x|^{a-1}, (0<a<1).$

If \(r=0\), then $ D^2(|\cdot|*\phi)=2\phi$
as a distribution on \(\R\).  Restricting to \((A,B)\) gives
$ 2\phi\llcorner(A,B)=G_a\,\one_{(A,B)}\,\dd x.$
Under the no-endpoint-atom ansatz, this gives
$   \phi=\frac12G_a\,\one_I\,\dd x,$
which is the first line of \eqref{eq:phi-I-direct-riesz}. If \(0<r<1\), the Green-kernel representation \eqref{eq:rho-green-explicit} follows directly from Proposition \ref{prop:green}.

Conversely, suppose the candidate \(\phi_I\) is nonnegative, has mass one, and
satisfies \eqref{eq:interval-centering-r-positive}.  The equation
$   D^2(V-\psi_r*\phi_I)=0
        \text{ in }(A,B)$
implies that $ V'-K_r*\phi_I$
is constant on \((A,B)\).  The one-point condition
\eqref{eq:interval-centering-r-positive} makes this constant zero.  Therefore $ K_a*\omega(x)-K_r*\phi_I(x)=0, x\in(A,B).$
Since \(\phi_I\) is supported on \(I\), this gives
$\phi_I(K_a*\omega-K_r*\phi_I)=0,$
thus \(\phi_I\) is a zero-flux stationary state.

For \(r=0\), the formula above gives
\(\phi_I=\frac12V''\mathbf{1}_I\,dx\). Since \(\phi_I\) has no atoms, for
every \(x\in(A,B)\) we have
\begin{equation}\label{eq:K0-phi-interval}
\begin{aligned}
K_0*\phi_I(x)
&=\int_I\operatorname{sgn}(x-y)\,\dd\phi_I(y)\\
&=\phi_I([A,x])-\phi_I((x,B])\\
&=2\phi_I([A,x])-\phi_I(I).
\end{aligned}
\end{equation}
Moreover,
\(\phi_I([A,x])=\frac12\bigl(V'(x)-V'(A)\bigr)\) and
\(\phi_I(I)=\frac12\bigl(V'(B)-V'(A)\bigr)\). Hence
\begin{equation}\label{eq:K0-phi-endpoint-derivatives}
K_0*\phi_I(x)
=
V'(x)-\frac12\bigl(V'(A)+V'(B)\bigr).
\end{equation}

The mass condition \(\phi_I(I)=1\) is equivalent to
\(\frac12\bigl(V'(B)-V'(A)\bigr)=1\). Under this condition,
\eqref{eq:K0-phi-interval} becomes
\(K_0*\phi_I(x)=V'(x)-V'(A)-1\). Therefore, the force-balance
condition \(V'(x_*)=K_0*\phi_I(x_*)\) at one point
\(x_*\in(A,B)\) is equivalent to \(V'(A)=-1\). Combining this
identity with the mass condition yields \(V'(B)=1\). Thus the mass
condition and force balance at one interior point are together
equivalent to
$V'(A)=-1,
V'(B)=1.$
Notice also from \eqref{eq:K0-phi-endpoint-derivatives} that
\(V'(x)-K_0*\phi_I(x)\) is constant on \((A,B)\). Consequently,
force balance at one interior point implies force balance throughout
the interval.

When \(a=0\), we have \(V=\psi_0*\omega\) and hence
$V'(x)
=
K_0*\omega(x)
=
2\int_{-\infty}^{x}\omega(z)\,\dd z-\omega(\R).$
The condition \(V'(A)=-1\) therefore gives
\(\int_{-\infty}^{A}\omega(z)\,\dd z
=(\omega(\R)-1)/2\), while \(V'(B)=1\) gives
\(\int_{B}^{\infty}\omega(z)\,\dd z
=(\omega(\R)-1)/2\), hence
 \eqref{eq:phi-a0-r0-direct-new} holds.
\end{proof}

{Because the formula in \eqref{eq:rho-green-explicit} is highly complicated, it is difficult to
determine \(A\) and \(B\) analytically. Consequently, it is also
difficult to verify the nonnegativity, unit-mass, and zero-flux
conditions required for the candidate in
\eqref{eq:rho-green-explicit} to be a stationary state.} Instead, we include a one-dimensional experiment in the attractive-dominant regime
$a>r$, in order to verify the direct interval characterization of Proposition
\ref{prop:1d-ar}. The experiment~(see Figure \ref{fig:1d-ageqr-experiment}) uses the symmetric background
$     \omega_m(x)=\frac m2\one_{[-1,1]}(x), m=1.8,$
and the exponent pair $ d=1, a=0.8, r=0.2.$ The implementation details can be obtained through \url{https://github.com/lukangsun/attraction-repulsion}.

\begin{figure}[t]
\centering
\includegraphics[width=1.0\textwidth]{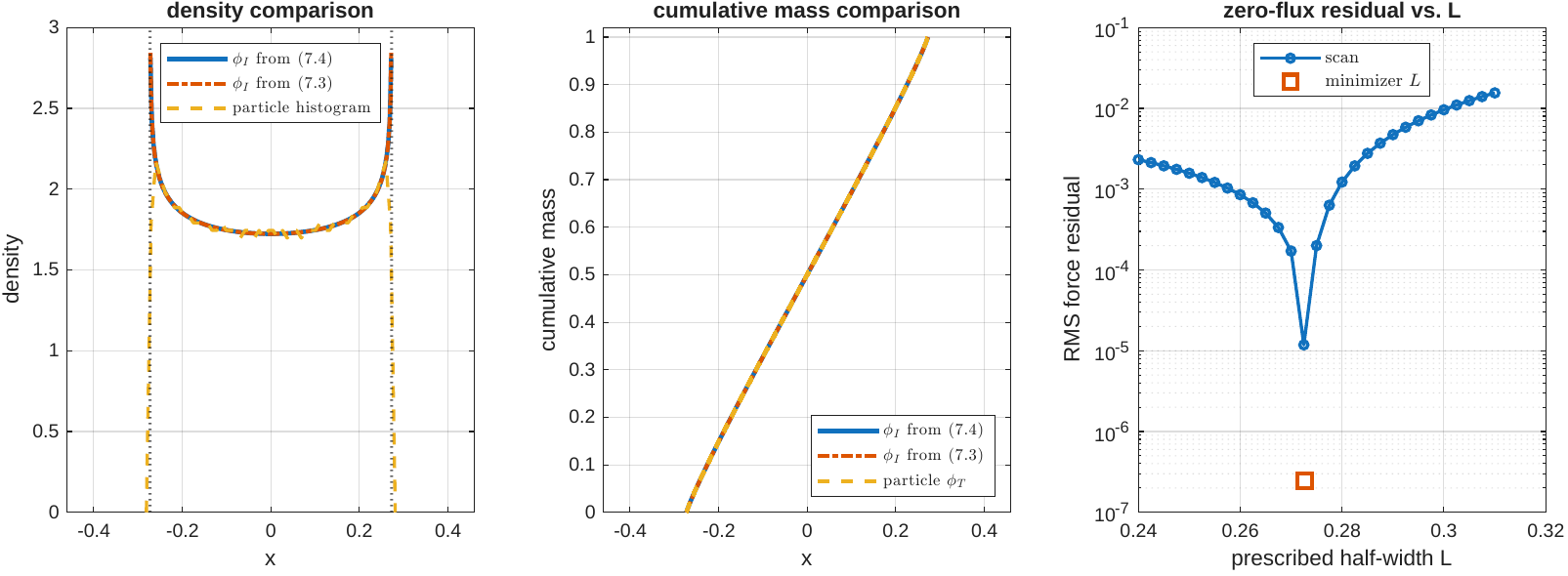}
\caption{One-dimensional verification in the attractive-dominant regime
$a>r$, with $a=0.8$, $r=0.2$, and $\omega_m=(m/2)\one_{[-1,1]}$, $m=1.8$.
Here, we let $I=[-L,L]$.  The left panel compares the density $\phi_I$
computed from the zero-flux condition \eqref{eq:interval-centering-r-positive}
with the density $\phi_I$ obtained from the direct Green-kernel
reconstruction \eqref{eq:rho-green-explicit}, together with a histogram of
the long-time particle distribution.  The middle panel compares the three
corresponding cumulative masses: $\phi_I$ from
\eqref{eq:interval-centering-r-positive}, $\phi_I$ from
\eqref{eq:rho-green-explicit}, and $\phi_T$ at terminal time $T=50$.  The
right panel shows the root-mean-square force residual of the discretized
zero-flux system \eqref{eq:interval-centering-r-positive} as a function of
the prescribed half-width $L$; its minimizer, marked by the square, selects
$L\approx0.2727$.}
\label{fig:1d-ageqr-experiment}
\end{figure}

\subsection{Two-dimensional disk formula and numerical verification}
In dimension \(d=2\), recovering \(\phi\) from the potential
\(\psi_r*\phi\) requires a fractional nonlocal inversion for
\(0\le r<1\); it cannot be done by applying a purely local differential
operator.  Hence the stationary measure obtained on a disk is not
computed by restricting the background density.  Instead one first solves the
fractional exterior Dirichlet problem for the auxiliary quantity
$H=(-\Delta)F$, and only then recovers $\phi$ from the inverse-kernel formula.

According to Proposition \ref{prop:green}, let $d=2$ and
$0\le r\le a<1,
        \alpha_r:=\frac{3+r}{2},
        s:=\frac{1+r}{2}\in\left[\frac12,1\right).
$
Then $\calL_r=c_{2,r}(-\Delta)^{1+s}$ and the source term in \eqref{eq:fractional-dirichlet} is
\begin{equation}\label{eq:2d-Qar}
        Q_{a,r}=\calL_r(\psi_a*\omega).
\end{equation}
With the convention that $Q_{a,a}=\omega$.  When
$a>r$, $Q_{a,r}$ is the Riesz-type potential
\begin{equation}\label{eq:2d-Q-riesz}
        Q_{a,r}=C_{2,a,r}\,\omega*|x|^{-(2-a+r)}
\end{equation}
see \eqref{eq:Q-riesz}.  Since $a-r\in(0,1)$ in this case, the kernel in
\eqref{eq:2d-Q-riesz} is locally integrable.

\begin{proposition}[Two-dimensional disk formula for $0\le r\le a<1$]
\label{prop:2d-general-disk}
Let \(d=2\), \(0\le r\le a<1\), and let \(Q_{a,r}\) be defined by
\eqref{eq:2d-Qar}, with the explicit representation given in
\eqref{eq:2d-Q-riesz}. Fix $R>0$ and suppose that the stationary support is the
closed disk $\overline{B_R(0)},
        \Omega= \overline{B_R(0)}^c.$
A candidate stationary measure is then given by
\begin{equation}\label{eq:2d-phi-general}
    \phi_R
    =
    \Bigl(Q_{a,r}-c_{2,r}(-\Delta)^sH_R\Bigr)
    \llcorner B_R(0),
\end{equation}
where \(H_R\) is defined by the Green-kernel representation
\begin{equation}\label{eq:2d-H-green-general}
    H_R(x)
    =
    c_{2,r}^{-1}
    \int_{B_R(0)^c}
    G_R^s(x,y)Q_{a,r}(y)\,\dd y,
    \qquad x\in B_R(0)^c,
\end{equation}
and \(H_R=0\) on \(B_R(0)\), here by Boggio's formula $
        G_{R}^{s}(x,y)
        =\kappa_{2,s}|x-y|^{-(2-2s)}
        \int_0^{\eta_R(x,y)}\frac{t^{s-1}}{1+t}\dd t, x,y\in B_R^c, $
and
$\eta_R(x,y)
        :=\frac{(|x|^2-R^2)(|y|^2-R^2)}{R^2|x-y|^2}.
$Thus, in two dimensions, the correction
to \(Q_{a,r}\) is a nonlocal bulk term generated by the exterior
solution \(H_R\).

If the candidate \(\phi_R\) defined in \eqref{eq:2d-phi-general} is a
nonnegative measure of total mass one and satisfies
\begin{equation}\label{eq:2d-force-verification-general}
    \int_{\R^2}K_a(x-y)\omega(y)\,\dd y
    =
    \int_{B_R(0)}K_r(x-y)\,\dd\phi_R(y),
    \qquad x\in B_R(0),
\end{equation}
then \(\phi_R\) is a zero-flux stationary probability measure.
\end{proposition}

\begin{proof} This follows from Proposition~\ref{prop:green}. Unlike in the one-dimensional case, when \(d=2\) and \(0\le r\le a<1\), the resulting fractional Dirichlet problem is always nonlocal.
\end{proof}

{Because the formula in \eqref{eq:2d-phi-general}  is highly complicated, it is difficult to
determine $R$ analytically. Consequently, it is also
difficult to verify the nonnegativity, unit-mass, and zero-flux
conditions required for the candidate in
\eqref{eq:2d-phi-general} to be a stationary state.} 
Instead, we include a two-dimensional experiment in the attractive-dominant regime
$a>r$, in order to verify the direct disk characterization of Proposition
\ref{prop:2d-general-disk}.
For the radial numerical example~(see Figure \ref{fig:2d-ar-phiT-compare}) below we take $
        \omega_m(x)=\frac{m}{2\pi}e^{-|x|^2/2}, m=1.8,
$
and use the nontrivial exponent pair$
        a=0.6, r=0.2.$ The implementation details can be obtained through \url{https://github.com/lukangsun/attraction-repulsion}.

\begin{figure}[t]
\centering
\includegraphics[width=1.0\textwidth]{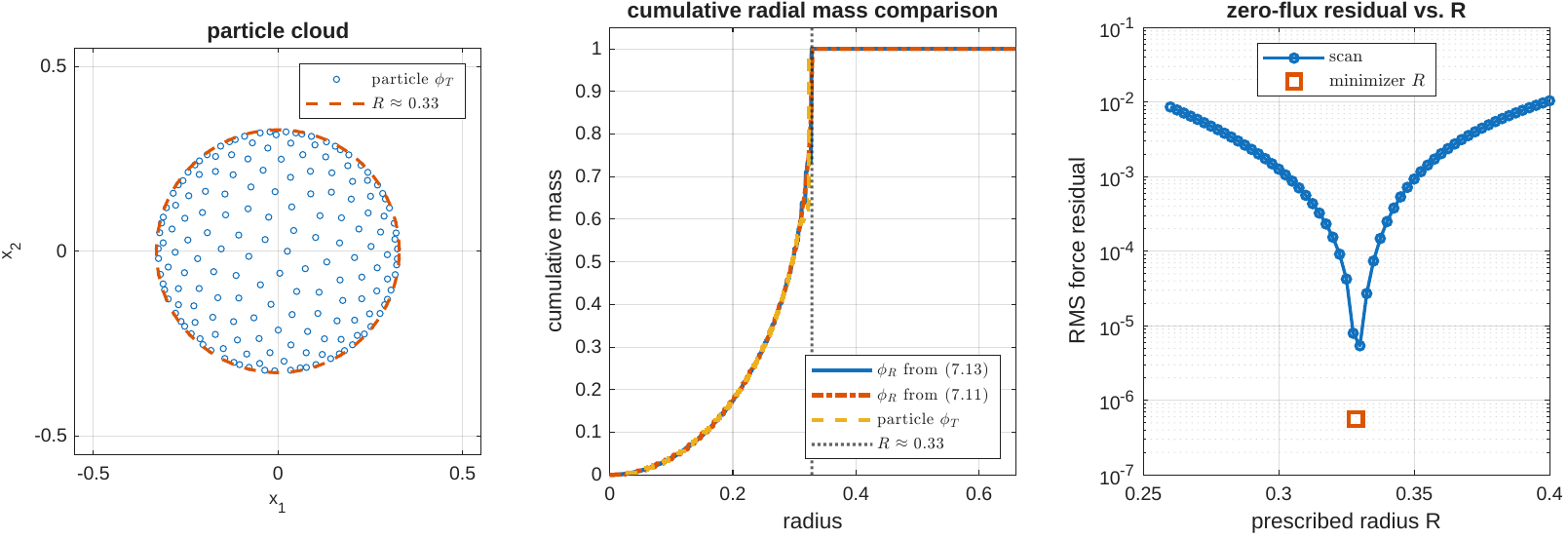}
\caption{Two-dimensional numerical verification for $a=0.6$, $r=0.2$, $\omega_m(x)=({m}/{2\pi})e^{-|x|^2/2}, m=1.8$.  The left panel shows a representative particle cloud for the
long-time radial particle solution $\phi_T$.  The middle panel compares
three cumulative radial masses: the stationary state $\phi_R$ selected by
the zero-flux equation \eqref{eq:2d-force-verification-general}, the
independently computed direct reconstruction $\phi_R$ from
\eqref{eq:2d-phi-general}, and the empirical cumulative mass of $\phi_T$ at
terminal time $T=50$.  The right panel shows the root-mean-square residual
of the discretized zero-flux equations
\eqref{eq:2d-force-verification-general} as a function of the prescribed
support radius $R$; its minimizer, marked by the square, selects
$R\approx0.33$.}
\label{fig:2d-ar-phiT-compare}
\end{figure}

\subsection{Three-dimensional local examples}

In the local case $d=3$ and $a=r=0$, the inverse-kernel characterization takes
an especially concrete form.  Since
$ \Delta^2 |x|=-8\pi\delta_0,$
we have
$ \calL_0=-\frac1{8\pi}\Delta^2,
        Q_{0,0}=\omega.$
{The same structural fact, namely that $-|x|$ is, up to a constant, the fundamental solution of $\Delta^2$ in $\mathbb{R}^3$, underlies the equilibrium measures computed by Chafa\"i, Saff, and Womersley \cite{chafai2023threshold} for Riesz kernels with power-law external fields in dimension $d=s+4$, which for $s=-1$ is exactly $d=3$ with kernel $-|x|$.}        
Thus, if
$  F=|\cdot|*\omega-|\cdot|*\phi,
         U:=\frac{1}{8\pi}H=-\frac1{8\pi}\Delta F,$
then the inverse formula $\phi=Q_{0,0}-\calL_0F$ becomes
\begin{equation}\label{eq:phi-from-U-3d}
        \phi=\omega-\Delta U
        \qquad\text{in }\calS'(\R^3).
\end{equation}
Let $S:=\supp\phi$ and $\Omega:= S^c$.  Since $\phi=0$ in
$\Omega$, \eqref{eq:phi-from-U-3d} gives
\begin{equation}\label{eq:exterior-poisson-U}
        \Delta U=\omega\quad\text{in }\Omega.
\end{equation}
Since the zero-flux condition is \(\phi\nabla F=0\), the
characterization requires \(\nabla F=0\) on \(S\), consequently, \(F\) is constant on each connected component of
\(\operatorname{int}S\), and hence
\(U=-(8\pi)^{-1}\Delta F=0\) there. If \(\omega\) is radially symmetric, it is natural to impose the ansatz
\(S=\overline{B_R(0)}\), which leads to the exterior Poisson problem
\begin{equation}\label{eq:exterior-poisson-U-ball}
    \Delta U=\omega
    \quad\text{in }S^c,
    \qquad
    U=0
    \quad\text{in }B_R(0),
    \qquad
    U(x)\to0
    \quad\text{as }|x|\to\infty.
\end{equation}
The decay condition at infinity, together with the homogeneous
Dirichlet condition, determines the exterior solution uniquely. After
extending \(U\) by zero to \(B_R(0)\), the stationary candidate is
recovered globally as
\begin{equation}\label{eq:phi-by-characterization}
    \phi=\omega-\Delta U
    \qquad\text{in }\mathcal S'(\R^3).
\end{equation}
The singular part of \(\phi\) supported on
\(\partial B_R(0)\) arises from taking the distributional Laplacian of
this zero extension across the boundary.

We shall use two elementary radial computations, which will be used in the calculation of the examples. First, let
\(G_R(x)=g_R(|x|)\) be the solution of
\begin{equation*}
    \Delta G_R=f(|x|)
    \quad\text{in }S^c,
    \qquad
    G_R=0
    \quad\text{in }B_R(0),
    \qquad
    G_R(x)\to0
    \quad\text{as }|x|\to\infty.
\end{equation*}
If \(\int_R^\infty s|f(s)|\,\dd s<\infty\), then
\begin{equation}\label{eq:normal-derivative-formula}
    g_R'(R+)
    =
    -\frac1R\int_R^\infty sf(s)\,\dd s,
\end{equation}
and the zero extension of \(G_R\) satisfies
\begin{equation}\label{eq:distributional-laplacian-GR}
    \Delta G_R
    =
    f(|x|)\mathbf 1_{\{|x|>R\}}\,\dd x
    -
    \frac1R
    \left(\int_R^\infty sf(s)\,\dd s\right)
    \mathcal H^2\llcorner\partial B_R(0)
    \qquad\text{in }\mathcal S'(\R^3).
\end{equation}
Consequently, the candidate \(\phi_R=\omega-\Delta G_R\), with
\(\omega=f(|x|)\,\dd x\), is given by
\begin{equation}\label{eq:3d-radial-candidate}
    \phi_R
    =
    f(|x|)\mathbf 1_{B_R(0)}\,\dd x
    +
    \frac1R
    \left(\int_R^\infty sf(s)\,\dd s\right)
    \mathcal H^2\llcorner\partial B_R(0).
\end{equation}
The detailed derivation for this formula is given in Appendix \ref{rmk:104}.

Second, for $x\in\overline{B_R}$,
\begin{align}
\int_{|y|>R} f(|y|)\frac{x-y}{|x-y|}\dd y
        &=\frac{8\pi}{3}\left(\int_R^\infty s f(s)\dd s\right)x,
        \label{eq:radial-volume-force}\\
\int_{\partial B_R}\frac{x-y}{|x-y|}\dd\Hh^2(y)
        &=\frac{8\pi}{3}R x.\label{eq:radial-surface-force}
\end{align}
To prove these identities, first compute, for $|x|<s$, $ \int_{\partial B_s}|x-y|\dd\Hh^2(y)
        =4\pi s^3+\frac{4\pi}{3}s|x|^2.$
Taking the $x$-gradient gives the surface identity with $s$ in place of $R$;
integrating it over $s\in(R,\infty)$ with density $f(s)$ gives
\eqref{eq:radial-volume-force}.  The formula extends to
$x\in\partial B_R$ by continuity of the angular integral.

\begin{proposition}[Uniform ball background in three dimensions]\label{prop:3d-uniform}
Let $d=3$, $a=r=0$, and
$ \omega(x)=c_m\one_{B_2(0)}(x),
        c_m:=\frac{3m}{32\pi},
        m\ge1.$
Let $R\in(0,2]$ be the unique solution of
\begin{equation}\label{eq:R-cubic}
        R^3-12R+\frac{16}{m}=0.
\end{equation}
Then the
measure
\begin{equation}\label{eq:phi-3d-uniform}
        \phi
        =c_m\one_{B_R(0)}\dd x
        +c_m\left(\frac{2}{R}-\frac{R}{2}\right)
        \Hh^2\lfloor\partial B_R(0).
\end{equation} is a zero-flux stationary probability measure.
\end{proposition}

\begin{proof}
\emph{Step 1: compute \(\phi\) from the characterization.}
Since \(\omega=c_m\mathbf{1}_{B_2(0)}\,\dd x\), its radial density is
\(f(r)=c_m\mathbf{1}_{[0,2]}(r)\). For \(0<R<2\), the coefficient of
the surface measure in \eqref{eq:3d-radial-candidate} is
$  \frac1R\int_R^\infty sf(s)\,\dd s
    =
    \frac{c_m}{R}\int_R^2s\,\dd s
    =
    \frac{c_m}{R}\left(2-\frac{R^2}{2}\right)
    =
    c_m\left(\frac2R-\frac R2\right).$
Therefore, \eqref{eq:3d-radial-candidate} gives \eqref{eq:phi-3d-uniform}.

\emph{Step 2: impose the probability constraint.}
The mass of \eqref{eq:phi-3d-uniform} is
\begin{align*}
        \phi(\R^3)
        &=c_m\frac{4\pi}{3}R^3
          +c_m\left(\frac2R-\frac R2\right)4\pi R^2=\frac{3m}{4}R-\frac{m}{16}R^3.
\end{align*}
Thus $\phi(\R^3)=1$ is equivalent to \eqref{eq:R-cubic}.  Since
$h(R):=R^3-12R+16/m$ satisfies $h(0)>0$, $h(2)\le0$, and
$h'(R)<0$ on $(0,2)$, there is a unique root in $(0,2]$.

\emph{Step 3: verify the stationary condition.}
For $x\in\overline{B_R}$, the volume part of $\phi$ cancels the part of
$\omega$ inside $B_R$.  Hence
\begin{align*}
K_0*\omega(x)-K_0*\phi(x)
&=c_m\int_{B_2\setminus B_R}\frac{x-y}{|x-y|}\dd y-c_m\left(\frac2R-\frac R2\right)
\int_{\partial B_R}\frac{x-y}{|x-y|}\dd\Hh^2(y).
\end{align*}
Using \eqref{eq:radial-volume-force} with $f=\one_{[R,2]}$ and using
\eqref{eq:radial-surface-force}, the first term is
$c_m\frac{4\pi}{3}(4-R^2)x$, while the second term is
$ c_m\left(\frac2R-\frac R2\right)\frac{8\pi}{3}R x
        =c_m\frac{4\pi}{3}(4-R^2)x.$
The force therefore vanishes on $\supp\phi$, and
$\phi(K_0*\omega-K_0*\phi)=0$.  This proves that the probability measure
computed above is zero-flux stationary.
\end{proof}

\begin{proposition}[Gaussian background in three dimensions]\label{prop:3d-gaussian}
Let $d=3$, $a=r=0$, and $\omega(x)=c_m e^{-|x|^2/2},
        c_m:=\frac{m}{(2\pi)^{3/2}}.$
If $m>1$, let $R>0$ be the unique solution of
\begin{equation}\label{eq:erf-condition}
        \operatorname{erf}\left(\frac{R}{\sqrt2}\right)=\frac1m,
        \qquad
        \operatorname{erf}(z):=\frac2{\sqrt\pi}\int_0^z e^{-u^2}\dd u.
\end{equation}
Then the radial free boundary $S=B_R(0)$ gives, through the characterization
\eqref{eq:exterior-poisson-U-ball}--\eqref{eq:phi-by-characterization}, the
measure
\begin{equation}\label{eq:phi-3d-gaussian}
        \phi
        =c_m e^{-|x|^2/2}\one_{B_R(0)}\dd x
        +\frac{c_m e^{-R^2/2}}{R}
        \Hh^2\lfloor\partial B_R(0).
\end{equation}
This measure is a zero-flux stationary probability measure.  If $m=1$, the
limiting stationary probability is $\phi=\omega$.
\end{proposition}

\begin{proof}
\emph{Step 1: compute \(\phi\) from the characterization.}
Since
\(\omega=c_m e^{-|x|^2/2}\,\dd x\), its radial density is
\(f(r)=c_m e^{-r^2/2}\). The coefficient of the surface measure in
\eqref{eq:3d-radial-candidate} is
\(\frac1R\int_R^\infty sf(s)\,\dd s
=\frac{c_m}{R}\int_R^\infty se^{-s^2/2}\,\dd s
=\frac{c_m e^{-R^2/2}}{R}\).
Therefore, \eqref{eq:3d-radial-candidate} gives
\eqref{eq:phi-3d-gaussian}.

\emph{Step 2: impose the probability constraint.}
Using \(c_m=m/(2\pi)^{3/2}\), the total mass of
\eqref{eq:phi-3d-gaussian} is
\begin{equation}\label{eq:gaussian-candidate-mass}
\begin{aligned}
    \phi(\R^3)
    &=
    m\int_{B_R(0)}
    \frac{e^{-|x|^2/2}}{(2\pi)^{3/2}}\,\dd x
    +
    4\pi R^2
    \frac{m e^{-R^2/2}}{(2\pi)^{3/2}R}\\
    &=
    m\left[
        \operatorname{erf}\left(\frac{R}{\sqrt2}\right)
        -
        \sqrt{\frac2\pi}R e^{-R^2/2}
    \right]
    +
    m\sqrt{\frac2\pi}R e^{-R^2/2}\\
    &=
    m\operatorname{erf}\left(\frac{R}{\sqrt2}\right).
\end{aligned}
\end{equation}
Thus the condition \(\phi(\R^3)=1\) is equivalent to
\eqref{eq:erf-condition}. If \(m>1\), then \(1/m\in(0,1)\), and the
strict monotonicity of the error function implies that
\eqref{eq:erf-condition} has a unique finite solution \(R>0\).

\emph{Step 3: verify the stationary condition.}
For \(x\in\overline{B_R(0)}\), the absolutely continuous part of
\(\phi\) cancels the restriction of \(\omega\) to \(B_R(0)\).
Consequently,
\begin{equation}\label{eq:gaussian-force-difference}
\begin{aligned}
    K_0*\omega(x)-K_0*\phi(x)
    &=
    c_m\int_{B_R(0)^c}
    e^{-|y|^2/2}\frac{x-y}{|x-y|}\,\dd y\\
    &\quad
    -
    \frac{c_m e^{-R^2/2}}{R}
    \int_{\partial B_R(0)}
    \frac{x-y}{|x-y|}\,\dd\mathcal H^2(y).
\end{aligned}
\end{equation}
By \eqref{eq:radial-volume-force}, applied with
\(f(s)=e^{-s^2/2}\mathbf 1_{[R,\infty)}(s)\), the first term on the
right-hand side of \eqref{eq:gaussian-force-difference} is
\(c_m\frac{8\pi}{3}e^{-R^2/2}x\). By
\eqref{eq:radial-surface-force}, the second term is
\(\frac{c_m e^{-R^2/2}}{R}\frac{8\pi}{3}Rx
=c_m\frac{8\pi}{3}e^{-R^2/2}x\).
The two terms therefore cancel on \(\operatorname{supp}\phi\), and
hence \(\phi(K_0*\omega-K_0*\phi)=0\). Thus the probability measure
obtained above is zero-flux stationary.

When \(m=1\), the background \(\omega\) is already a probability
measure and satisfies \(K_0*\omega-K_0*\omega=0\) identically.
Therefore, \(\phi=\omega\) is stationary. It is also the limit of
\eqref{eq:phi-3d-gaussian} as \(R\to\infty\).
\end{proof}

Since the measure obtained above has already been rigorously verified to be a zero-flux stationary measure, we do not believe that an additional numerical experiment is necessary. In the more general cases \(0\leq r\leq a<1\), stationary candidates can be constructed using Proposition~\ref{prop:green}.
\begin{revision}
    \begin{remark}
In the examples of this section, although \(\omega\)~(with \(m>1\))
has a regular bounded density, the corresponding zero-flux stationary
state contains a singular surface component and therefore does not
belong to \(L^\infty(\R^d)\). 
{Equilibrium measures combining an absolutely continuous part with a singular part on a sphere were also obtained in \cite{chafai2023threshold} for Riesz energies with power-law external fields.}
This singular component arises naturally
from the free-boundary characterization~\eqref{eq:exterior-poisson-U-ball}: extending the auxiliary
exterior solution by zero across the free boundary may produce a
boundary-supported measure when the relevant local operator $-\Delta$~(see \eqref{eq:phi-from-U-3d}) is applied. This does not contradict the finite-time
well-posedness result in Theorem~\ref{thm:wellposed}, which does not
provide a uniform-in-time \(L^\infty\) bound or exclude concentration
as \(t\to\infty\). Thus, a singular stationary measure may in
principle occur as a long-time limit.
    \end{remark}
\end{revision}

\section{Long-time convergence to stationary states}\label{sec:longtime}

\begin{proposition}[Energy-dissipation identity]
\label{prop:energy-diss}
Let $a,r\in[0,1)$, and let $(\phi_t)_{t\geq0}$ be the
Lagrangian distributional solution of \eqref{eq:pde} constructed in
Theorem~\ref{thm:wellposed}. Assume that
Assumption~\ref{ass:standing} holds with an exponent
$ p\geq 1+\max\{a,r\}.$
Then $E(\phi_t)$ is finite for every $t\geq0$,
$t\mapsto E(\phi_t)$ belongs to
$AC_{\mathrm{loc}}([0,\infty))$, and, for every
$0\leq s\leq t<\infty$,
\begin{equation}
\label{eq:energy-dissipation-integrated}
    E(\phi_t)
    +
    \int_s^t\int_{\mathbb R^d}
    |K_a*\omega(x)-K_r*\phi_\tau(x)|^2
    \,d\phi_\tau(x)\,d\tau
    =
    E(\phi_s).
\end{equation}
Consequently,
\begin{equation}
\label{eq:energy-dissipation-differential}
    \frac{d}{dt}E(\phi_t)
    =
    -\int_{\mathbb R^d}
    |K_a*\omega(x)-K_r*\phi_t(x)|^2
    \,d\phi_t(x)
\end{equation}
for almost every $t>0$.

{If, moreover, $\inf E>-\infty$, then}
\begin{equation}
\label{eq:averaged-dissipation}
    \frac1T\int_0^T
    \left(
        \int_{\mathbb R^d}
        |K_a*\omega-K_r*\phi_t|\,d\phi_t
    \right)^2dt
    \leq
    \frac{E(\phi_0)-\inf E}{T}.
\end{equation}
In particular, there exists a sequence $t_k\to\infty$ such that
\begin{equation}
\label{eq:vanishing-dissipation}
    \int_{\mathbb R^d}
    |K_a*\omega-K_r*\phi_{t_k}|\,d\phi_{t_k}
    \longrightarrow0.
\end{equation}
\end{proposition}

\begin{proof}[Proof sketch]
Let \(m:=\max\{a,r\}\), and use the squared-radius regularization from
the proof of Theorem~\ref{thm:wellposed}.  For the corresponding smooth solution
\(\phi^\varepsilon_t\), {define the regularized energy $E_\varepsilon(\mu):=\int(\psi_{a,\varepsilon}*\omega)\,d\mu-\frac12\int(\psi_{r,\varepsilon}*\mu)\,d\mu$.} Its first variation has gradient $B^\varepsilon_t
=
K_{a,\varepsilon}*\omega
-
K_{r,\varepsilon}*\phi^\varepsilon_t.$ After introducing a spatial cutoff and letting its radius tend to
infinity, the moment assumption
\(p\ge 1+m\) controls the boundary term.  Differentiating the
regularized energy along the smooth flow therefore gives $
 E_\varepsilon(\phi^\varepsilon_t)
+\int_s^t\int_{\mathbb R^d}
|B^\varepsilon_\tau|^2\,d\phi^\varepsilon_\tau\,d\tau
= E_\varepsilon(\phi^\varepsilon_s).
$ The same moment assumption provides uniform integrability of the
moments required by both interaction terms.  Consequently, the
regularized energies converge to \(E(\phi_t)\): on compact
sets this follows from narrow convergence and local uniform
convergence of the regularized potentials, while the tails are
uniformly negligible by the moment bounds.  The dissipation terms are
handled in the same way, using the local convergence of the forces
established in Theorem~\ref{thm:wellposed} and a tail estimate based on
\(2m<1+m\le p\).  Passing to the limit in the regularized identity
yields the results.

Absolute continuity of \(t\mapsto E(\phi_t)\) and the
a.e.\ differential identity follow immediately.  If
\(\inf E>-\infty\), the energy identity and Cauchy--Schwarz
give the time-averaged dissipation estimate; hence one can select a
sequence \(t_k\to\infty\) along which the dissipation defect tends to
zero.  Since the detailed proof is lengthy, we defer it to Appendix~\ref{sec:pwd} in order to preserve the flow of the main text.
\end{proof}

\begin{remark}[On the moment assumption in Proposition~\ref{prop:energy-diss}]
\label{rem:energy-moment}
The strengthened moment condition $ p\geq 1+\max\{a,r\}$
in Proposition~\ref{prop:energy-diss} is natural in view of
the growth of the energy. {Indeed, the two terms in the energy~\eqref{eq:energy1}}
involve moments of orders $1+a$ and $1+r$, respectively. Thus the
condition $p\geq 1+\max\{a,r\}$ ensures that both interaction terms are
finite and that the energy is well defined as a finite real-valued
quantity along the flow. 

This condition is specific to the energy analysis and should be
distinguished from the weaker requirement $p>1$ in
Assumption~\ref{ass:standing}, which is sufficient for the
well-posedness theory. In particular, no such higher-order moment is
needed for the existence and uniqueness result of
Theorem~\ref{thm:wellposed}.
\end{remark}

\begin{assumption}[Moment compactness of the trajectory]\label{ass:compactness}
There exists an exponent $ q>\max\{a,r\}$
such that
\begin{equation}\label{eq:moment-assumption}
        \sup_{t\ge0}\int_{\R^d}|x|^q\dd\phi_t(x)<\infty,
        \qquad
        \int_{\R^d}|x|^q\omega(x)\dd x<\infty .
\end{equation}
\end{assumption}

\begin{remark}[When the moment compactness assumption is available]
\label{rem:moment-known}
Assumption \ref{ass:compactness} is automatic if a separate uniform-support
theorem applies.  It is also available from the energy-sublevel moment estimate
in the attractive regimes treated earlier.  For instance, Proposition~\ref{prop:moment-anchor} shows that if either
\(a>r, a,r\in [0,1)\) and \(\omega(\R^d)>0\), or \(a=r\) and \(\omega(\R^d)>1\), then the
energy sublevel estimate yields a uniform \((1+a)\)-moment bound, provided
\(\omega\) has finite \((1+a)\)-moment.  In these regimes one may take $q=1+a,$
which satisfies \(q>\max\{a,r\}\).

On the other hand, by Remark~\ref{rmk:eql}, if \(a=r\in[0,1)\) and
\(\omega\in\mathcal P_2\), then one has a uniform \(p\)-moment bound for every $  p\in\left(0,\frac{1+a}{2}\right).$
Since \(a<(1+a)/2\) when \(a<1\), one may choose
$ q\in\left(a,\frac{1+a}{2}\right),$
which again gives \(q>\max\{a,r\}\).

Outside these regimes, Assumption~\ref{ass:compactness} should be retained as a
separate compactness hypothesis.
\end{remark}
\begin{theorem}[Subsequential convergence to zero-flux stationary states]
\label{thm:subseq-conv}
\begin{revision}
Let \(0\leq r\leq a<1\), let \(\phi_t\) be a global Lagrangian solution of
\eqref{eq:pde}, and assume that the conditions of Proposition
\ref{prop:energy-diss} hold. Set \(M:=\int_{\mathbb R^d}\omega(x)\,dx\), and assume, in addition, that either \(a>r\) and \(M>0\), or \(a=r\) and \(M\geq1\). Then \(\inf_{\mu\in\mathcal P(\mathbb R^d)}E(\mu)>-\infty\). \end{revision}Additionally under Assumption \ref{ass:compactness}, there exist a sequence \(t_k\to\infty\) and a probability measure \(\phi_\infty\) such that
\begin{equation}\label{eq:narrow-conv}
        \phi_{t_k}\weak\phi_\infty
        \qquad\text{narrowly,}
\end{equation}
{and \(\phi_\infty\) is a zero-flux stationary state, $\phi_\infty(K_a*\omega-K_r*\phi_\infty)=0$ in the sense of vector-valued distributions.} Moreover, for every \(0<\delta<q\), the convergence in \eqref{eq:narrow-conv} holds against every continuous function \(f\) satisfying \(|f(x)|\leq C_f(1+|x|^{q-\delta})\).
\end{theorem}
\begin{proof}[Proof sketch]
By Proposition~\ref{prop:energy-diss}, there exists a sequence \(t_k\to\infty\) such
that $
\int_{\mathbb R^d}
\bigl|K_a*\omega-K_r*\phi_{t_k}\bigr|\,d\phi_{t_k}
\rightarrow0.
$ Assumption~\ref{ass:compactness} gives tightness of \((\phi_{t_k})_k\).  Thus, after
passing to a subsequence,$
\phi_{t_k}\rightharpoonup\phi_\infty$ narrowly.  The same moment bound gives uniform integrability of every
power strictly below \(q\), and hence improves this convergence to
test functions with growth \(O(1+|x|^{q-\delta})\).

It remains to pass to the limit in the approximate zero-flux
relation.  Let \(\zeta\in C_c^\infty(\mathbb R^d;\mathbb R^d)\).
The attraction term passes to the limit because
\(K_a*\omega\) is continuous on \(\operatorname{supp}\zeta\) and has
admissible growth.  For the repulsion term, if \(r>0\), the kernel$
(x,y)\longmapsto \zeta(x)\cdot K_r(x-y)$ is continuous and has growth controlled by the \(q\)-moment; therefore
the convergence of the product measures
\(\phi_{t_k}\otimes\phi_{t_k}\) applies.  When \(r=0\), the diagonal
{discontinuity is removed using the oddness of \(K_0\). Set $H_\zeta(x,y):=(\zeta(x)-\zeta(y))\cdot K_0(x-y)$ off the diagonal and $H_\zeta(x,x):=0$. Then $\int\zeta\cdot(K_0*\mu)\,d\mu=\frac12\iint H_\zeta\,d\mu^{\otimes2}$. The symmetrized integrand is bounded and continuous, so it also passes to the limit.}

We therefore obtain$
\int_{\mathbb R^d}
\zeta(x)\cdot
\bigl(K_a*\omega(x)-K_r*\phi_\infty(x)\bigr)\,
d\phi_\infty(x)
=0.
$ Since \(\zeta\) is arbitrary,$
\phi_\infty\bigl(K_a*\omega-K_r*\phi_\infty\bigr)=0
$ in the sense of vector-valued distributions.  Thus \(\phi_\infty\) is
a zero-flux stationary state.  The detailed proof is given in
Appendix~\ref{sec:pwd}.
\end{proof}

The next lemma is well known; for the reader's convenience, we recall its statement here.
\begin{lemma}
\label{lem:barbalat}
Let \(f:[0,\infty)\to[0,\infty)\) be uniformly continuous.  If
$f\in L^p(0,\infty)$
for some \(1\le p<\infty\), then
$ f(t)\to0
        \text{ as }t\to\infty .$
\end{lemma}

\begin{sectioneightchanges}
\begin{lemma}[Uniform H\"older continuity of the dissipation defect]
\label{lem:Dfrak-uniform-continuity}
Let $a,r\in[0,1)$, assume Assumption~\ref{ass:standing} with
$p\ge1+\max\{a,r\}$, and let $\phi_t=(X_t)_\#\phi_0$ be the global
Lagrangian solution of Theorem~\ref{thm:wellposed}. Define
$\mathfrak D(t):=\int_{\R^d}|K_a*\omega-K_r*\phi_t|\dd\phi_t$.
The following conclusions hold, with constants independent of $s,t\ge0$.
\begin{enumerate}[label=\textup{(\roman*)},nosep]
\item If $0<r<1$ and $\inf E>-\infty$, then
$|\mathfrak D(t)-\mathfrak D(s)|\le C|t-s|^r$.
The constant depends only on $d,r,a$, $\|\omega\|_1,\|\omega\|_\infty$ and
$E(\phi_0)-\inf E$.
\item If $r=0$, 
for $a=0$, $\mathfrak D$ is uniformly $\alpha$-H\"older continuous for
every $0<\alpha<1$. For $0<a<1$, the same conclusion holds for every
$0<\alpha<(1+a)/(1+2a)$.
If, in addition, $\phi_0$ is compactly supported, then every
$0<\alpha<1$ is admissible also when $0<a<1$.
\end{enumerate}
\end{lemma}

\begin{proof}
\noindent\emph{Step 1: a common estimate and the case $0<r<1$.}
Put $A:=K_a*\omega$, $F_t:=K_r*\phi_t$, and $b_t:=F_t-A$.
The characteristic equation is $\dot X_t(z)=b_t(X_t(z))$, and the
push-forward identity gives
$\mathfrak D(t)=\int|b_t(X_t(z))|\dd\phi_0(z)$.
Lemma~\ref{lem:averaged-lip}, applied to $\omega$, gives
$|A(x+u)-A(x)|\le L|u|$ for $|u|\le1$, with a constant $L$ depending
only on $d,a,\|\omega\|_1$ and $\|\omega\|_\infty$.
For an arbitrary displacement $u$, divide the segment into $N$ pieces
with $|u|/N\le1$ and sum these inequalities. Thus the same bound holds
for every $u$, and $\Lip(A)\le L$ globally.

For $s,t\ge0$, set $h:=|t-s|$ and
$\Delta_{t,s}(z):=X_t(z)-X_s(z)$.
The reverse triangle inequality gives
$|\mathfrak D(t)-\mathfrak D(s)|
\le\int|b_t(X_t(z))-b_s(X_s(z))|\dd\phi_0(z)$.
For the repulsive field, the push-forward identity yields
$F_t(X_t(z))=\int K_r(X_t(z)-X_t(w))\dd\phi_0(w)$.
Subtracting the corresponding formula at time $s$ and integrating in
$z$, we obtain
\begin{equation}\label{eq:Dfrak-split}
\begin{aligned}
|\mathfrak D(t)-\mathfrak D(s)|
&\le L\|\Delta_{t,s}\|_{L^1(\phi_0)}+\mathcal R_r(t,s),\\
\mathcal R_r(t,s)
&:=\iint|K_r(X_t(z)-X_t(w))-K_r(X_s(z)-X_s(w))|
\dd\phi_0(z)\dd\phi_0(w).
\end{aligned}
\end{equation}
Assume first $0<r<1$ and use the global $r$-H\"older estimate
\eqref{eq:Ks-holder}.

We first obtain a weaker modulus from the dissipated energy, then use
it to bound the mean speed and improve the modulus. Assume $s\le t$
and put $J_{s,t}:=\int_s^t\mathfrak D(\tau)\dd\tau$.
Integrating the characteristic equation and using the push-forward
identity yields
\begin{equation}\label{eq:Dfrak-mean-displacement}
\begin{aligned}
\|\Delta_{t,s}\|_{L^1(\phi_0)}
&\le\int_s^t\int|b_\tau(X_\tau(z))|\dd\phi_0(z)\dd\tau\\
&=\int_s^t\int|b_\tau(x)|\dd\phi_\tau(x)\dd\tau=J_{s,t}.
\end{aligned}
\end{equation}
Since $\phi_0^{\otimes2}$ is a probability measure and $u\mapsto u^r$
is concave, Jensen's inequality gives
$\mathcal R_r(t,s)\le C_r\iint
|\Delta_{t,s}(z)-\Delta_{t,s}(w)|^r\dd\phi_0^{\otimes2}
\le C_r(2\|\Delta_{t,s}\|_{L^1(\phi_0)})^r$.
Combining this with \eqref{eq:Dfrak-split} and
\eqref{eq:Dfrak-mean-displacement} gives
\begin{equation}\label{eq:Dfrak-increment-mean-speed}
|\mathfrak D(t)-\mathfrak D(s)|
\le LJ_{s,t}+C_r(2J_{s,t})^r.
\end{equation}
The energy identity \eqref{eq:energy-dissipation-integrated} and
Cauchy--Schwarz, first in space and then in time, imply
$\int_0^\infty\mathfrak D(\tau)^2\dd\tau
\le\int_0^\infty\int|b_\tau|^2\dd\phi_\tau\dd\tau\le E(\phi_0)-\inf E$
and $J_{s,t}\le (E(\phi_0)-\inf E)^{1/2}h^{1/2}$.
Substitution into \eqref{eq:Dfrak-increment-mean-speed} gives
\begin{equation}\label{eq:Dfrak-energy-bootstrap}
\begin{aligned}
|\mathfrak D(t)-\mathfrak D(s)|
&\le L(E(\phi_0)-\inf E)^{1/2}h^{1/2}\\
&\quad+2^rC_r(E(\phi_0)-\inf E)^{r/2}h^{r/2}.
\end{aligned}
\end{equation}
This proves continuity of $\mathfrak D$ and a uniform local
$r/2$-H\"older bound in terms of the initial energy gap.

For each $t\ge0$, the bound
$\int_t^{t+1}\mathfrak D(\tau)^2\dd\tau\le E(\phi_0)-\inf E$ and continuity
imply that there exists at least one $u_t\in[t,t+1]$ such that
$\mathfrak D(u_t)\le (E(\phi_0)-\inf E)^{1/2}$.
Applying \eqref{eq:Dfrak-energy-bootstrap} to $t,u_t$, whose distance
is at most one, yields
$\mathfrak D(t)\le (1+L)(E(\phi_0)-\inf E)^{1/2}+2^rC_r(E(\phi_0)-\inf E)^{r/2}=:B$.
Since $t$ was arbitrary, $\mathfrak D(\tau)\le B$ for every $\tau\ge0$.
Thus, by the definition of $J_{s,t}$ and $h=t-s$,
$J_{s,t}=\int_s^t\mathfrak D(\tau)\dd\tau
\le\int_s^t B\dd\tau=B(t-s)=Bh$.
Returning to
\eqref{eq:Dfrak-increment-mean-speed}, we obtain
\begin{equation}\label{eq:Dfrak-positive-r-holder}
|\mathfrak D(t)-\mathfrak D(s)|
\le LBh+2^rC_rB^rh^r
\le (LB+2^rC_rB^r)h^r\qquad(0\le h\le1).
\end{equation}
For $h>1$, nonnegativity and $\mathfrak D\le B$ give
$|\mathfrak D(t)-\mathfrak D(s)|\le B\le Bh^r$.
Thus \eqref{eq:Dfrak-positive-r-holder}, with
$C:=\max\{B,LB+2^rC_rB^r\}$, proves~\textup{(i)} for all $s,t$.

\smallskip
\noindent\emph{Step 2: differentiation of inverse-distance moments when $r=0$.}
In the rest of the proof $K_0(x)=x/|x|$ for $x\ne0$, with $K_0(0)=0$.
Fix $0<\beta<1$, and define
$I_\beta(\mu):=\iint|x-y|^{-\beta}\dd\mu(x)\dd\mu(y)$ and
$I_\beta(t):=I_\beta(\phi_t)$.
Splitting the initial integral at distance one and using polar
coordinates gives
\begin{equation}\label{eq:r0-initial-inverse-moment}
\begin{aligned}
I_\beta(0)
&\le\int\left(\|\phi_0\|_\infty
\int_{|u|<1}|u|^{-\beta}\dd u+1\right)\dd\phi_0(x)\\
&=1+\|\phi_0\|_\infty|\mathbb S^{d-1}|
\int_0^1\rho^{d-1-\beta}\dd\rho
=1+\frac{|\mathbb S^{d-1}|}{d-\beta}\|\phi_0\|_\infty<\infty.
\end{aligned}
\end{equation}
Here $\beta<d$ in every dimension under consideration.

Fix $T>0$. By Theorem~\ref{thm:wellposed}, the densities $\phi_\tau$
are bounded uniformly for $0\le\tau\le T$.
Lemma~\ref{lem:averaged-lip} then gives a constant $\Lambda_T<\infty$
such that $\Lip(b_\tau)\le\Lambda_T$ throughout this interval.
For distinct labels $z,w$, put
$Z_\tau:=X_\tau(z)-X_\tau(w)$ and $\rho_\tau:=|Z_\tau|$.
The Lipschitz bound gives
$|b_\tau(X_\tau(z))-b_\tau(X_\tau(w))|\le\Lambda_T\rho_\tau$.
The chain rule expresses $\dot\rho_\tau$ as the scalar product of
$Z_\tau/\rho_\tau$ with this velocity difference, so
$|\dot\rho_\tau|\le\Lambda_T\rho_\tau$.
Gronwall's inequality, in both directions, and the chain rule yield
\begin{equation}\label{eq:r0-pair-flow-bounds}
\begin{aligned}
e^{-\Lambda_T\tau}|z-w|&\le\rho_\tau
\le e^{\Lambda_T\tau}|z-w|,\qquad 0\le\tau\le T,\\
\left|\frac{d}{d\tau}\rho_\tau^{-\beta}\right|
&\le\beta\Lambda_T\rho_\tau^{-\beta}
\le\beta\Lambda_Te^{\beta\Lambda_TT}|z-w|^{-\beta}.
\end{aligned}
\end{equation}
 The right hand side of \eqref{eq:r0-pair-flow-bounds} is
integrable by \eqref{eq:r0-initial-inverse-moment}, it justifies
$I_\beta\in AC([0,T])$.
Changing variables with the flow, the time derivative is
\begin{equation}\label{eq:r0-inverse-derivative}
\begin{aligned}
I_\beta'(t)
&=-\beta\iint\frac{(x-y)\cdot(b_t(x)-b_t(y))}
{|x-y|^{\beta+2}}\dd\phi_t(x)\dd\phi_t(y)\\
&=\beta\iint\frac{(x-y)\cdot(A(x)-A(y))}{|x-y|^{\beta+2}}
\dd\phi_t(x)\dd\phi_t(y)-\beta R_\beta(\phi_t)\\
&\le\beta L I_\beta(t)-\beta R_\beta(\phi_t).
\end{aligned}
\end{equation}
In this formula, set
$f_\beta(x,y,z):=(x-y)\cdot(K_0(x-z)-K_0(y-z))/|x-y|^{\beta+2}$
and $R_\beta(\mu):=\iiint f_\beta\dd\mu^{\otimes3}$.
The identity for the repulsive term follows by inserting
$F_t(x)-F_t(y)=\int(K_0(x-z)-K_0(y-z))\dd\phi_t(z)$.
 Moreover,
$|F_t(x)-F_t(y)|\le\Lambda_T'|x-y|$ on $[0,T]$ for a finite
$\Lambda_T'$, which gives
$0\le R_\beta(\phi_t)\le\Lambda_T'I_\beta(t)<\infty$.
Thus all terms in \eqref{eq:r0-inverse-derivative} are well-defined.

\smallskip
\noindent\emph{Step 3: the three-point lower bound.}
For pairwise distinct $x,y,z$, let
$\ell:=\max\{|x-y|,|y-z|,|z-x|\}$ and
$T_\beta:=(f_\beta(x,y,z)+f_\beta(y,z,x)+f_\beta(z,x,y))/3$.
Each summand is nonnegative. To estimate the summand associated with
the largest side, suppose $|x-y|=\ell$, put $u=x-z$, $v=y-z$, and
write $b:=|u|$, $c:=|v|$.
If $\theta$ is the angle between $u$ and $v$, the cosine rule and
$\ell\ge\max\{b,c\}$ give
$\cos\theta=(b^2+c^2-\ell^2)/(2bc)
\le\min\{b/c,c/b\}/2\le1/2$.
Furthermore, direct multiplication gives
$ (u-v)\cdot(u/b-v/c)
=b+c-(u\cdot v)(1/b+1/c)
=(b+c)(1-\cos\theta)$.
Since $b+c\ge\ell$, this quantity is at least $\ell/2$.
Consequently
$f_\beta(x,y,z)\ge\frac12\ell^{-\beta-1}$ and
$T_\beta\ge\frac16\ell^{-\beta-1}$.

Define $J_\beta(\mu):=\iiint\ell^{-\beta-1}\dd\mu^{\otimes3}$.
The product measure is invariant under cyclic permutation of the
variables, so integrating each of the three summands in $T_\beta$
gives the same value $R_\beta(\mu)$. Hence
\begin{equation}\label{eq:r0-three-point-bound}
R_\beta(\mu)=\iiint T_\beta\dd\mu^{\otimes3}
\ge\frac16\iiint\ell^{-\beta-1}\dd\mu^{\otimes3}
=\frac16J_\beta(\mu).
\end{equation}
We apply \eqref{eq:r0-three-point-bound} to $\mu=\phi_t$, which is
non-atomic at every finite time. Thus the diagonals have zero product
measure and $0<J_\beta(\phi_t)<\infty$ by Step~2.

\smallskip
\noindent\emph{Step 4: interpolation, explicit comparison, and uniformity in time.}
For $\mu=\phi_t$, define
$F(\rho):=\int\mu(B_\rho(x))\dd\mu(x)$.
Because $\mu$ is a probability measure, $0\le F\le1$ and Jensen's
inequality gives
$F(\rho)^2\le\int\mu(B_\rho(x))^2\dd\mu(x)$.
For each $x$, the ball mass is
$\mu(B_\rho(x))=\int\one_{\{|x-y|<\rho\}}\dd\mu(y)$.
Expanding its square with separate variables $y,z$ and then integrating
in $x$, then
\begin{equation}\label{eq:r0-ball-square-expansion}
\begin{aligned}
\mu(B_\rho(x))^2
&=\bigl(\int\one_{\{|x-y|<\rho\}}\dd\mu(y)\bigr)
  \bigl(\int\one_{\{|x-z|<\rho\}}\dd\mu(z)\bigr)\\
&=\iint\one_{\{|x-y|<\rho\}}\one_{\{|x-z|<\rho\}}
  \dd\mu(y)\dd\mu(z),\\
\int\mu(B_\rho(x))^2\dd\mu(x)
&=\iiint\one_{\{|x-y|<\rho\}}\one_{\{|x-z|<\rho\}}
  \dd\mu^{\otimes3},
\end{aligned}
\end{equation}
where $\dd\mu^{\otimes3}=\dd\mu(x)\dd\mu(y)\dd\mu(z)$.
When both indicators in \eqref{eq:r0-ball-square-expansion} equal one,
$|x-y|<\rho$ and $|x-z|<\rho$, the triangle inequality gives
$|y-z|\le|y-x|+|x-z|<2\rho$, so all three distances are less than
$2\rho$ and hence $\ell<2\rho$.
Since $\beta+1>0$, this implies
$\ell^{-\beta-1}\ge(2\rho)^{-\beta-1}$ and therefore
$1\le(2\rho)^{\beta+1}\ell^{-\beta-1}$.
When at least one indicator equals zero, their product is zero and
is also bounded by the nonnegative quantity
$(2\rho)^{\beta+1}\ell^{-\beta-1}$.
Thus the product has this upper bound in both cases.
Combining it with \eqref{eq:r0-ball-square-expansion} and the definition of $J_\beta(\mu)$ yields
\begin{equation}\label{eq:r0-ball-mass-bound}
\begin{aligned}
F(\rho)^2
&\le\int\mu(B_\rho(x))^2\dd\mu(x)
 =\iiint\one_{\{|x-y|<\rho\}}\one_{\{|x-z|<\rho\}}\dd\mu^{\otimes3}\\
&\le(2\rho)^{\beta+1}\iiint\ell^{-\beta-1}\dd\mu^{\otimes3}
 =(2\rho)^{\beta+1}J_\beta(\mu).
\end{aligned}
\end{equation}
Both sides of \eqref{eq:r0-ball-mass-bound} are nonnegative, so taking
square roots gives
$F(\rho)\le((2\rho)^{\beta+1}J_\beta(\mu))^{1/2}
=2^{(\beta+1)/2}J_\beta(\mu)^{1/2}\rho^{(\beta+1)/2}$.

Using $\beta\int_0^\infty \rho^{-\beta-1}
\mathbf1_{\{|x-y|<\rho\}}\,d\rho
=\beta\int_{|x-y|}^\infty\rho^{-\beta-1}\,d\rho
=\bigl[-\rho^{-\beta}\bigr]_{|x-y|}^{\infty}
=|x-y|^{-\beta}$, we obtain
\begin{equation}\label{eq:r0-energy-interpolation}
\begin{aligned}
I_\beta(\mu)
&=\beta\int_0^\infty\rho^{-\beta-1}
\bigl(\iint\one_{\{|x-y|<\rho\}}\dd\mu(x)\dd\mu(y)\bigr)\dd\rho\\
&=\beta\int_0^\infty\rho^{-\beta-1}F(\rho)\dd\rho\\
&\le\beta2^{(\beta+1)/2}J_\beta(\mu)^{1/2}
\int_0^\delta\rho^{-(\beta+1)/2}\dd\rho
+\beta\int_\delta^\infty\rho^{-\beta-1}\dd\rho\\
&=\frac{2\beta\,2^{(\beta+1)/2}}{1-\beta}
J_\beta(\mu)^{1/2}\delta^{(1-\beta)/2}+\delta^{-\beta}.
\end{aligned}
\end{equation}
This is the precise point where $\beta<1$ is needed.
Set $A_\beta:=1+2\beta\,2^{(\beta+1)/2}/(1-\beta)$ and choose
$\delta=J_\beta(\mu)^{-1/(\beta+1)}$ in
\eqref{eq:r0-energy-interpolation}.
Both terms then have the same power of $J_\beta$, because
$1/2-(1-\beta)/(2(\beta+1))=\beta/(\beta+1)$.
Thus $I_\beta(\mu)\le A_\beta J_\beta(\mu)^{\beta/(\beta+1)}$.
Combining this with \eqref{eq:r0-three-point-bound} yields
\begin{equation}\label{eq:r0-coercivity}
R_\beta(\mu)\ge c_\beta I_\beta(\mu)^{1+1/\beta},
\qquad c_\beta:=\frac16 A_\beta^{-(\beta+1)/\beta}>0.
\end{equation}

By \eqref{eq:r0-inverse-derivative} and \eqref{eq:r0-coercivity},
$I_\beta'\le\beta L I_\beta-\beta c_\beta I_\beta^{1+1/\beta}$.
The function $I_\beta$ is positive, so put
$Y(t):=I_\beta(t)^{-1/\beta}$.
The chain rule gives
$Y'=-(1/\beta)I_\beta^{-1/\beta-1}I_\beta'
\ge-LY+c_\beta$.
For $L>0$, multiply by $e^{Lt}$ and integrate from $0$ to $t$;
for $L=0$, integrate directly. This gives
\begin{equation}\label{eq:r0-bernoulli-comparison}
Y(t)\ge
\begin{cases}
e^{-Lt}I_\beta(0)^{-1/\beta}
+\dfrac{c_\beta}{L}(1-e^{-Lt}),&L>0,\\[3pt]
I_\beta(0)^{-1/\beta}+c_\beta t,&L=0.
\end{cases}
\end{equation}
In the first case the right-hand side is at least
$\min\{I_\beta(0)^{-1/\beta},c_\beta/L\}$; in the second it is at
least $I_\beta(0)^{-1/\beta}$. Taking the negative power $-\beta$ in
\eqref{eq:r0-bernoulli-comparison} proves
\begin{equation}\label{eq:r0-inverse-moment-bound}
\sup_{t\ge0}I_\beta(t)
\le\max\{I_\beta(0),(L/c_\beta)^\beta\}=:M_\beta<\infty.
\end{equation}
By \eqref{eq:r0-initial-inverse-moment}, $M_\beta$ depends only on the
initial density, the background, $d,a$ and $\beta$. The finite-time
constants $\Lambda_T,\Lambda_T'$ used to justify differentiation do not
enter this bound. 

\smallskip
\noindent\emph{Step 5: an increment estimate for the unit-direction kernel.}
For $\xi\ne0$, differentiation gives
$DK_0(\xi)=(I-\widehat\xi\otimes\widehat\xi)/|\xi|$ and
$\|DK_0(\xi)\|\le1/|\xi|$, where $\widehat\xi:=\xi/|\xi|$.
This bound is also valid in $d=1$, where the derivative vanishes away
from the origin. Fix $0<\alpha<1$, $u\ne0$, and $q:=|u-v|$.
If $q\le|u|/2$, the connecting segment stays at distance at least
$|u|/2$ from zero, so
$|K_0(u)-K_0(v)|\le2q/|u|\le2(q/|u|)^\alpha$.
If $q>|u|/2$, boundedness of $K_0$, including $K_0(0)=0$, gives
$|K_0(u)-K_0(v)|\le2\le2^{1+\alpha}(q/|u|)^\alpha$.
Hence in both cases
\begin{equation}\label{eq:r0-kernel-increment}
|K_0(u)-K_0(v)|
\le C_\alpha|u-v|^\alpha|u|^{-\alpha},
\qquad C_\alpha:=2^{1+\alpha}.
\end{equation}

\smallskip
\noindent\emph{Step 6: bounded speed, including $a=r=0$.}
Suppose $|b_t|\le B$ along the characteristics, with $B$ independent
of time. Integrating the characteristic equation gives
$|\Delta_{t,s}(z)|\le Bh$.
Since $\phi_0$ has a density, $z\ne w$ for
$\phi_0^{\otimes2}$-almost every pair $(z,w)$.
The lower bound in \eqref{eq:r0-pair-flow-bounds} implies
$X_s(z)\ne X_s(w)$ for these pairs, so we may apply \eqref{eq:r0-kernel-increment} with
$u=X_s(z)-X_s(w)$ and $v=X_t(z)-X_t(w)$.
The preceding displacement bound yields
$|u-v|=|\Delta_{t,s}(z)-\Delta_{t,s}(w)|\le2Bh$, and therefore
$|K_0(X_t(z)-X_t(w))-K_0(X_s(z)-X_s(w))|
\le C_\alpha(2Bh)^\alpha|X_s(z)-X_s(w)|^{-\alpha}$.

Integrating this inequality with respect to
$\dd\phi_0(z)\dd\phi_0(w)$ gives
$\mathcal R_0(t,s)\le C_\alpha(2Bh)^\alpha
\iint|X_s(z)-X_s(w)|^{-\alpha}\dd\phi_0(z)\dd\phi_0(w)$.
By $\phi_s=(X_s)_\#\phi_0$, the double integral equals
$\iint|x-y|^{-\alpha}\dd\phi_s(x)\dd\phi_s(y)=I_\alpha(s)$.
Using \eqref{eq:r0-inverse-moment-bound}, we conclude that
$\mathcal R_0(t,s)\le C_\alpha(2Bh)^\alpha I_\alpha(s)
\le C_\alpha(2B)^\alpha M_\alpha h^\alpha$.
Using this estimate in \eqref{eq:Dfrak-split}, we obtain
\begin{equation}\label{eq:r0-bounded-speed-holder}
|\mathfrak D(t)-\mathfrak D(s)|
\le LBh+C_\alpha(2B)^\alpha M_\alpha h^\alpha.
\end{equation}
For $0<h\le1$, the right-hand side is at most
$(LB+C_\alpha(2B)^\alpha M_\alpha)h^\alpha$.
For $h>1$, $\mathfrak D\le B$ gives the bound $2Bh^\alpha$.
This proves uniform $\alpha$-H\"older continuity for every
$0<\alpha<1$ whenever the speed is bounded.
In particular, for $a=0$ one has
$|K_0*\phi_t|\le1$ and $|K_0*\omega|\le M:=\omega(\R^d)$, so
$B=1+M$ works everywhere, without any support assumption.

\smallskip
\noindent\emph{Step 7: $0<a<1$.}
First suppose $M>0$. The energy identity and
Proposition~\ref{prop:moment-anchor}, with $r=0<a$, give
$C_{\rm mom}:=\sup_{t\ge0}\int|x|^{1+a}\dd\phi_t(x)<\infty$.
Let $m_a:=\int|y|^a\omega(y)\dd y$.
Because $|x-y|^a\le|x|^a+|y|^a$, one has
$|A(x)|\le(1+a)(M|x|^a+m_a)$.
Together with $|K_0*\phi_t|\le1$, this yields
$|b_t(x)|\le C_A(1+|x|^a)$, where
$C_A:=1+(1+a)(M+m_a)$.
Take $k:=(1+a)/a>2$. Since $ak=1+a$, direct integration gives
$\int|b_t|^k\dd\phi_t
\le2^{k-1}C_A^k(1+\int|x|^{ak}\dd\phi_t)
\le2^{k-1}C_A^k(1+C_{\rm mom})$.
Consequently $B_k:=\sup_{t\ge0}\|b_t\|_{L^k(\phi_t)}<\infty$.
For $s<t$, integrating the characteristic equation and applying
Minkowski's inequality gives
\begin{equation}\label{eq:r0-speed-moment}
\begin{aligned}
\|\Delta_{t,s}\|_{L^k(\phi_0)}
&=\left\|\int_s^t b_\tau(X_\tau(\cdot))\dd\tau\right\|_{L^k(\phi_0)}\\
&\le\int_s^t\|b_\tau(X_\tau(\cdot))\|_{L^k(\phi_0)}\dd\tau
=\int_s^t\|b_\tau\|_{L^k(\phi_\tau)}\dd\tau
\le B_kh.
\end{aligned}
\end{equation}
The same estimate holds for $t<s$ by reversing the endpoints.
If $M=0$, nonnegativity implies $\omega=0$, so $|b_t|\le1$ and
\eqref{eq:r0-speed-moment} holds with $B_k\le1$ directly.

Fix $0<\alpha<k/(k+1)$ and put
$\beta:=\alpha k/(k-\alpha)$, $p_1:=k/\alpha$, and
$p_2:=k/(k-\alpha)$.
These are conjugate exponents since $1/p_1+1/p_2=1$, and
$\beta<1$ is equivalent to $\alpha(k+1)<k$.
Put $Q_{t,s}:=Z_t-Z_s=\Delta_{t,s}(z)-\Delta_{t,s}(w)$,
using the notation $Z_s$ from Step~2.
The triangle inequality in $L^k(\phi_0^{\otimes2})$ and
\eqref{eq:r0-speed-moment} give
$\|Q_{t,s}\|_{L^k(\phi_0^{\otimes2})}
\le\|\Delta_{t,s}(z)\|_{L^k(\phi_0^{\otimes2})}
+\|\Delta_{t,s}(w)\|_{L^k(\phi_0^{\otimes2})}
=2\|\Delta_{t,s}\|_{L^k(\phi_0)}\le2B_kh$.
Using \eqref{eq:r0-kernel-increment}, then H\"older's inequality with
$p_1,p_2$, and finally \eqref{eq:r0-inverse-moment-bound}, gives
\begin{equation}\label{eq:r0-repulsive-holder}
\begin{aligned}
\mathcal R_0(t,s)
&\le C_\alpha\iint|Q_{t,s}|^\alpha|Z_s|^{-\alpha}\dd\phi_0^{\otimes2}\\
&\le C_\alpha
\left(\iint|Q_{t,s}|^{\alpha p_1}\dd\phi_0^{\otimes2}\right)^{1/p_1}
\left(\iint|Z_s|^{-\alpha p_2}\dd\phi_0^{\otimes2}\right)^{1/p_2}\\
&=C_\alpha\|Q_{t,s}\|_{L^k(\phi_0^{\otimes2})}^\alpha
I_\beta(s)^{1-\alpha/k}
\le C_\alpha(2B_k)^{\alpha}M_\beta^{1-\alpha/k}h^\alpha.
\end{aligned}
\end{equation}
Here $\alpha p_1=k$, $\alpha p_2=\beta$, and
$1/p_2=1-\alpha/k$, which explains every exponent in the estimate.
The attraction term in \eqref{eq:Dfrak-split} is at most $LB_kh$,
because $\phi_0$ is a probability measure and
$\|\Delta_{t,s}\|_{L^1(\phi_0)}\le\|\Delta_{t,s}\|_{L^k(\phi_0)}$.
Combining this with \eqref{eq:r0-repulsive-holder}, and using
$h\le h^\alpha$ for $h\le1$, proves the desired estimate for small
time differences. For $h>1$, the bound
$\mathfrak D(t)\le\|b_t\|_{L^k(\phi_t)}\le B_k$ gives
$|\mathfrak D(t)-\mathfrak D(s)|\le2B_kh^\alpha$.
Finally, $k/(k+1)=(1+a)/(1+2a)>1/2$, so the exponent $1/2$ is
always admissible.

If $\phi_0$ is compactly supported and $M>0$,
Theorem~\ref{thm:support} gives a uniform radius $R_*$.
The estimate $|b_t(x)|\le C_A(1+R_*^a)$ on the supports then permits
the bounded-speed argument \eqref{eq:r0-bounded-speed-holder}, for
every $0<\alpha<1$. If $M=0$, use $B=1$ in the same argument.
This completes (ii).
\end{proof}

\begin{definition}[Omega-limit set and omega-limit points]
\label{def:omega-limit}
For a trajectory $\{\phi_t:t\ge0\}\subset\calP(\R^d)$, its narrow
omega-limit set is
$\Omega_{\rm lim}(\phi_0):=\{\mu\in\calP(\R^d):
\phi_{t_k}\weak\mu\text{ for some }t_k\to\infty\}$.
Every element of this set is called an omega-limit point.
\end{definition}

\begin{corollary}[Asymptotic stationarity and full convergence under uniqueness]
\label{cor:full-conv}
Assume the hypotheses of Theorem~\ref{thm:subseq-conv}.
Then $\mathfrak D(t)\to0$, the omega-limit set is nonempty and narrowly
compact, and every omega-limit point is a zero-flux stationary state.
If there is only one zero-flux stationary state in the omega-limit set,
say $\bar\phi$, then
\begin{equation}\label{eq:full-conv}
\phi_t\weak\bar\phi\qquad\text{as }t\to\infty.
\end{equation}
\end{corollary}

\begin{remark}\label{rem:longtime-data}
Under Assumption~\ref{ass:standing} with $p\ge1+a$, all hypotheses of
Corollary~\ref{cor:full-conv}, apart from its optional uniqueness condition,
follow from conditions on the data in the following regimes:
$0\le r<a<1$, $\omega(\R^d)>0$; or $0\le a=r<1$,
$\omega(\R^d)>1$; or $0\le a=r<1$, $\omega(\R^d)=1$ and
$\omega\in\calP_2(\R^d)$.
The first two cases use Proposition~\ref{prop:moment-anchor}, and the
third uses Remark~\ref{rmk:eql}, to obtain Assumption~\ref{ass:compactness}.
\end{remark}

\begin{proof}
By Proposition \ref{prop:energy-diss},
$ \int_0^T\mathfrak D(t)^2\,\dd t
        \le
        E(\phi_0)-E(\phi_T)
        \le
        E(\phi_0)-\inf E .$
Letting \(T\to\infty\), we obtain
$  \mathfrak D\in L^2(0,\infty).$
By Lemma \ref{lem:Dfrak-uniform-continuity}, \(\mathfrak D\) is uniformly
continuous on \([0,\infty)\).  Applying Lemma \ref{lem:barbalat} with \(p=2\),
we get
\begin{equation}\label{eq:Dfrak-goes-zero}
        \mathfrak D(t)\to0
        \qquad\text{as }t\to\infty .
\end{equation}

Let \(t_k\to\infty\) be arbitrary.  By the compactness assumption in
Theorem \ref{thm:subseq-conv}, there exists a subsequence \(t_{k_j}\to\infty\) and a probability measure
\(\phi_\infty\) such that $  \phi_{t_{k_j}}\weak\phi_\infty$
narrowly.  Since \(\mathfrak D(t)\to0\), we have
$ \mathfrak D(t_{k_j})\to0.$
Therefore the proof of Theorem \ref{thm:subseq-conv} applies to this subsequence
and shows that \(\phi_\infty\) is a zero-flux stationary state.

Thus every omega-limit point is zero-flux stationary.  If the omega-limit set
contains only one such zero-flux stationary state, denoted by \(\bar\phi\), then
every convergent subsequence of \(\phi_t\) has the same limit \(\bar\phi\).
By precompactness, this is equivalent to full convergence:
$  \phi_t\weak\bar\phi
    \text{ as }t\to\infty .$
\end{proof}

\end{sectioneightchanges}
\section{Wasserstein gradient flow of MMD: The case
\texorpdfstring{\(a=r\) with \(\omega(\mathbb{R}^d)=1\)}
{a=r with omega(Rd)=1}}
\label{sec:a=r}
In this section, we specifically comment on the special case $a = r$ and $\omega \in
\mathcal{P}(\mathbb{R}^d)$, i.e., the case in which the attraction and
repulsion exponents coincide and the background measure is itself a
probability measure. This case is of particular interest because, {as
mentioned in Section~\ref{sec:intro},  the energy \eqref{eq:energy1}
reduces to the Minimum Mean Discrepancy in the following sense:}
\begin{equation}\label{eq:energy_MMD}
    E(\phi) = \operatorname{MMD}^2(\phi, \omega) + C_\omega,
\end{equation}
where $C_\omega = \frac{1}{2}\iint \psi_r(x-y)\,\mathrm{d}\omega(x)\,
\mathrm{d}\omega(y)$ depends only on $\omega$, so that minimizing $E$
over probability measures is equivalent to minimizing the MMD between
$\phi$ and $\omega$. The gradient flow \eqref{eq:main} is therefore
precisely the Wasserstein gradient flow of $\operatorname{MMD}^2(\cdot,
\omega)$ with the negative-distance kernel $-|\cdot|^{1+r}$, and
$\phi = \omega$ is the unique global minimizer of $E$, with $E(\omega)
= C_\omega$.
\\

The results of the present paper apply to this case as follows.
\\

The global Lagrangian well-posedness of Theorem~\ref{thm:wellposed},
including the {finite-time} $L^\infty$ and moment bounds, the propagation of
$W^{n,\infty}$ regularity of Proposition~\ref{prop:Sobolev}, and the
uniqueness in the Lagrangian class, hold without any restriction on the
mass of $\omega$. In particular they apply to the case $a = r$ and
$\omega(\mathbb{R}^d) = 1$.
\\

Theorem~\ref{thm:support} on uniform confinement of the support
does \emph{not} apply in general when $a = r$ and $\omega(\mathbb{R}^d) = 1$. The intuitive reason can be explained as follows: {if  $\omega$ is not compactly supported} and $\phi_t$ converges then the support of $\phi_t$ ought to grow. {Yet, for $d=1$, $a=r=0$ and compactly supported $\omega$, the support of $\phi_t$ remains uniformly bounded, as observed in \cite[Remark~4.5]{difrancesco2015asymptotic}; see also Remark~\ref{rem:support-borderline}.} 
\\

The free-boundary characterization of zero-flux stationary states
developed in Section~\ref{sec:stationary} applies to the case $a = r$
without restriction on the mass of $\omega$. When $a = r$, the source
term in Proposition~\ref{prop:green} simplifies to $Q_{a,a} = \omega$,
and the characterization of Theorem~\ref{thm:stationary-characterization} reduces to
\begin{equation}
    \phi = \omega - \mathcal{L}_r F, \qquad
    F = \psi_r * \omega - \psi_r * \phi,
\end{equation}
with the zero-flux condition $\nabla F = 0$ holding $\phi$-a.e.\ on
the support of $\phi$. In particular, $\phi = \omega$ is always a
zero-flux stationary state when $a = r$ and $\omega(\mathbb{R}^d) = 1$. This is consistent with
the MMD interpretation: $\phi = \omega$ is the unique minimizer of
$E$, and any minimizer of a smooth energy is a stationary point of its
gradient flow. {In general $\omega$ is not the \emph{only} zero-flux stationary state: as in Remark~\ref{rem:nonuniqueness}, for $\omega=\one_{[-1/2,1/2]}$ both $\delta_0$ and $\omega$ are zero-flux stationary states. Non-uniqueness, however, requires singular measures: by \cite[Theorem~1.4]{rosenzweig2025mmd}, every absolutely continuous zero-flux stationary state with finite $(1+a)$-moment coincides with $\omega$, in every dimension.}
\\

\begin{sectioneightchanges}
For $a=r\in[0,1)$ and $\omega(\R^d)=1$,
Proposition~\ref{prop:energy-diss} gives the energy-dissipation identity
under the moment assumption $p\ge1+a$.
The MMD representation \eqref{eq:energy_MMD} gives
$E(\phi_t)\ge E(\omega)$.
Lemma~\ref{lem:Dfrak-uniform-continuity} yields uniform
$r$-H\"older continuity of $\mathfrak D$ when $r>0$, and uniform
$\alpha$-H\"older continuity for every $0<\alpha<1$ when $r=0$.
Together with the energy-dissipation identity and
Lemma~\ref{lem:barbalat}, this implies $\mathfrak D(t)\to0$.

If $\omega\in\calP_2$, Remark~\ref{rmk:eql} provides a uniform
$q$-moment bound for every $0<q<(1+a)/2$.
Choosing $a<q<(1+a)/2$ verifies Assumption~\ref{ass:compactness}.
Corollary~\ref{cor:full-conv} then shows that the omega-limit set is
nonempty and narrowly compact, and that every omega-limit point is a
zero-flux stationary state. If this set contains only one stationary
state, the whole trajectory converges narrowly to that state.

Convergence to the prescribed target, $\phi_t\weak\omega$, is established
under uniform-in-time bounds on the $(1+a)$-moment and on a subcritical
$L^p$ norm in \cite[Corollary~5.25]{rosenzweig2025mmd}.
Without these additional bounds, target convergence remains open in
general. Moreover, no initial-data-independent decay rate can hold
\cite[Proposition~1.11]{rosenzweig2025mmd}.
\end{sectioneightchanges}
\\
In Table~\ref{tab:mmd-case} we summarize the validity of the results in the MMD case $a=r$ and $\omega(\mathbb R^d)=1${, including those of \cite{rosenzweig2025mmd}}.
\\

For different and more specific results on the analysis of the MMD gradient flow for $a=r$, we refer to the work by Rosenzweig, Slep\v{c}ev, and Wang~\cite{rosenzweig2025mmd}.

{
\begin{table}[htbp]
\centering
\small
\renewcommand{\arraystretch}{1.25}
\begin{tabularx}{\textwidth}{>{\raggedright\arraybackslash}X >{\raggedright\arraybackslash}p{3.1cm} c >{\raggedright\arraybackslash}X}
\toprule
\textbf{Result} & \textbf{Kind} & \textbf{Holds?} & \textbf{Remarks} \\
\midrule
Thm.~\ref{thm:wellposed}, Prop.~\ref{prop:Sobolev}: well-posedness, $L^\infty$/moment bounds, $W^{n,\infty}$ regularity, uniqueness in the Lagrangian class
  & Global well-posedness & \checkmark
  & Well-posedness also in \cite[Thm.~1.2]{rosenzweig2025mmd}, in the larger class $\calP_r\cap L^p$, $r\ge1$, $p>d/(d+a-1)$ ($p=\infty$ if $d=1$, $a=0$) \\ \addlinespace
Thm.~\ref{thm:support}
  & Uniform support confinement & $\times$
  & See Remark~\ref{rem:support-borderline}; for $d=1$, $a=0$ and compactly supported $\omega$, cf.\ \cite[Remark~4.5]{difrancesco2015asymptotic} \\
\midrule
Prop.~\ref{prop:green}, Thm.~\ref{thm:stationary-characterization}
  & Zero-flux stationary characterization & \checkmark
  & Among absolutely continuous measures with finite $(1+a)$-moment it selects only $\phi=\omega$ \cite[Thm.~1.4]{rosenzweig2025mmd} \\ \addlinespace
$\phi=\omega$ is a zero-flux state (Def.~\ref{def:zeroflux})
  & MMD minimizer is stationary & \checkmark
  & -- \\ \addlinespace
$\phi=\omega$ is the only zero-flux state among all probability measures
  & Uniqueness of stationary state & $\times$
  & Singular counterexamples, e.g.\ $\delta_0$ (Remark~\ref{rem:nonuniqueness}) \\ \addlinespace
$\phi=\omega$ is the only zero-flux state among absolutely continuous measures with finite $(1+a)$-moment
  & Rigidity & \checkmark
  & \cite[Thm.~1.4(i)]{rosenzweig2025mmd}, in every dimension $d\ge1$ \\
\midrule
Prop.~\ref{prop:energy-diss}
  & Energy dissipation & \checkmark
  & cf.\ \cite[Prop.~3.8]{rosenzweig2025mmd} \\ \addlinespace
Assumption~\ref{ass:compactness}, Thm.~\ref{thm:subseq-conv}
  & Subsequential convergence & \checkmark
  & Moment bound from Remark~\ref{rmk:eql} \\ \addlinespace
\sectioneightedit{Every $\omega$-limit point of $\phi_t$ is a zero-flux state}
  & Asymptotic criticality & \checkmark
  & {Cor.~\ref{cor:full-conv} for every $0\le a=r<1$, using Remark~\ref{rmk:eql}} \\ \addlinespace
Cor.~\ref{cor:full-conv}
  & {Full convergence under uniqueness} & \sectioneightedit{$(\checkmark)$}
  &{Unique omega-limit state; moment compactness from Remark~\ref{rmk:eql}} \\ \addlinespace
$\phi_t\rightharpoonup\omega$ and $\operatorname{MMD}(\phi_t,\omega)\to0$
  & Convergence to the target & $(\checkmark)$
  & \cite[Cor.~5.25]{rosenzweig2025mmd}, if $\sup_{t\ge0}\big(\int|x|^{1+a}\,\mathrm{d}\phi_t+\|\phi_t\|_{L^p}\big)<\infty$ for a subcritical $p$; open in general \\ \addlinespace
Initial-data-independent MMD decay rate; global Polyak--{\L}ojasiewicz inequality
  & Quantitative convergence & $\times$
  & \cite[Prop.~1.11 and Prop.~6.2]{rosenzweig2025mmd} \\
\bottomrule
\end{tabularx}
\caption{Results in the case $a = r$, $\omega(\mathbb{R}^d) = 1$
(MMD gradient flow), from the present paper and from \cite{rosenzweig2025mmd}.
\checkmark\ = holds; $\times$ = does not hold in general; $(\checkmark)$ = holds under additional hypotheses.}
\label{tab:mmd-case}
\end{table}
}

\section*{Declaration of generative AI and AI-assisted technologies in the manuscript preparation process.}
During the preparation of this work the authors used ChatGPT in order to 1. search for references, 2. check for spelling and 3. generate the software to perform the numerical experiments in Section \ref{sec:examples}. After using this tool/service, the authors reviewed and edited the content as needed and take full responsibility for the content of the published article.

\begingroup
\hypersetup{linkcolor=black,citecolor=black,urlcolor=black}
\bibliography{bib.bib}
\endgroup
\bibliographystyle{manuscript-siam}

\bigskip
\begin{center}
  \FundingLogos
  
  \vspace{0.5em}
  \begin{tcolorbox}\centering\small
    Funded by the European Union. Views and opinions expressed are however those of the author(s) only and do not necessarily reflect those of the European Union or the European Research Council Executive Agency. Neither the European Union nor the granting authority can be held responsible for them. This project has received funding from the European Research Council (ERC) under the European Union’s Horizon Europe research and innovation programme (grant agreement No. 101198055, project acronym NEITALG).
   \end{tcolorbox}
\end{center}

\section{Appendix}\label{sec:apx}

\subsection{Proofs with Details}\label{sec:pwd}

\begin{proof}[Proof of Theorem \ref{thm:wellposed}]
The proof maintains the squared-radius regularization throughout and develops the compactness argument needed to pass to the limit.\\
\smallskip
\noindent\emph{Step 1: the regularized equation and consistent notation.}
For $0<\eps\le1$, define
\begin{equation}\label{eq:regularized-potentials}
        V_\eps:=\psi_{a,\eps}*\omega,
        \qquad
        K_{r,\eps}:=\grad\psi_{r,\eps},
        \qquad
        B_t^\eps:=\grad V_\eps-K_{r,\eps}*\phi_t^\eps.
\end{equation}
The velocity in the continuity equation is $-B_t^\eps$.  Choose smooth
nonnegative initial data $\phi_{0,\eps}$ with mass one such that
\begin{equation}\label{eq:initial-approx}
        \phi_{0,\eps}\to\phi_0\quad\text{in }L^1(\R^d),
        \qquad
        \norm{\phi_{0,\eps}}_{L^\infty}\le \norm{\phi_0}_{L^\infty}+1,
\end{equation}
and
\begin{equation}\label{eq:initial-moment-approx}
        \sup_{0<\eps\le1}\int_{\R^d}\langle x\rangle^p\phi_{0,\eps}(x)\dd x<\infty,
        \qquad
        \int\langle x\rangle^p\phi_{0,\eps}\to
        \int\langle x\rangle^p\phi_0 .
\end{equation}
For instance one may truncate $\phi_0$, mollify, and renormalize.  The
regularized equation is
\begin{equation}\label{eq:regularized-pde}
        \partial_t\phi_t^\eps=\diver(\phi_t^\eps B_t^\eps),
        \qquad
        \phi^\eps_{t=0}=\phi_{0,\eps},
\end{equation}
or equivalently
\begin{equation}\label{eq:regularized-flow}
        \dot X_t^\eps(x)=-B_t^\eps(X_t^\eps(x)),
        \qquad
        \phi_t^\eps=(X_t^\eps)_{\#}\phi_{0,\eps}.
\end{equation}
For fixed $\eps>0$, the vector field is smooth and has at most linear growth,
because $|K_{s,\eps}(x)|\le C_s(1+|x|)$.  Classical Picard iteration therefore
gives a global smooth Lagrangian solution.  Positivity and mass conservation
follow from the push-forward formula.

\smallskip
\noindent\emph{Step 2: the uniform maximum estimate.}
Let $M_\eps(t):=\norm{\phi_t^\eps}_{L^\infty}$.  At a spatial maximum point of
$\phi_t^\eps$,
\[
\partial_t\phi_t^\eps
=\grad\phi_t^\eps\cdot B_t^\eps
+\phi_t^\eps\Bigl(\Delta V_\eps-(\Delta\psi_{r,\eps})*\phi_t^\eps\Bigr).
\]
The gradient term vanishes, and Lemma \ref{lem:eps-formulas} gives
$\Delta\psi_{r,\eps}\ge0$.  Thus
\begin{equation}\label{eq:max-est-eps}
        \frac{\dd}{\dd t}M_\eps(t)
        \le \norm{\Delta V_\eps}_{L^\infty}M_\eps(t).
\end{equation}
By Lemma \ref{lem:regularized-derivatives},
\begin{equation}\label{eq:DeltaVeps-uniform}
        \sup_{0<\eps\le1}\norm{\Delta V_\eps}_{L^\infty}
        \le C_{d,a}(\norm{\omega}_{L^1}+\norm{\omega}_{L^\infty}).
\end{equation}
Gronwall's lemma and \eqref{eq:initial-approx} give the uniform bound
\begin{equation}\label{eq:Linfty-eps}
        \sup_{0\le t\le T}\norm{\phi_t^\eps}_{L^\infty}
        \le C_T.
\end{equation}
{This maximum-principle argument is classical for aggregation equations with repulsive kernels; see, e.g. \cite{raoul2012nonlocal}.}

\smallskip

\noindent\emph{Step 3: the moment estimate from the Lagrangian flow.}  For
$s\in[0,1]$ and every finite nonnegative measure $\mu$ with finite first moment,
\begin{equation}\label{eq:force-growth}
        |K_{s,\eps}*\mu(x)|
        \le C_s\int_{\R^d}(1+|x-y|)\dd\mu(y)
        \le C_s\bigl(1+|x|+\int |y|\dd\mu(y)\bigr),
\end{equation}
with a constant independent of $\eps$.  Hence, using $\int\phi_t^\eps=1$ and
H\"older's inequality,
\begin{equation}\label{eq:B-growth-moment}
        |B_t^\eps(x)|
        \le C\Bigl(1+|x|+\bigl(M_p^\eps(t)\bigr)^{1/p}
        +M_{p,\omega}^{1/p}\Bigr),
\end{equation}
where $ M_p^\eps(t):=\int_{\R^d}\langle x\rangle^p\phi_t^\eps(x)\dd x,
        M_{p,\omega}:=\int_{\R^d}\langle x\rangle^p\omega(x)\dd x .$
Since $\phi_t^\eps=(X_t^\eps)_\#\phi_{0,\eps}$,
\begin{equation}\label{eq:moment-flow-representation}
        M_p^\eps(t)
        =\int_{\R^d}\langle X_t^\eps(x)\rangle^p
        \phi_{0,\eps}(x)\dd x .
\end{equation}
For a characteristic $X_t^\eps(x)$ we have, for a.e. $t$,
\begin{align}
\frac{\dd}{\dd t}\langle X_t^\eps(x)\rangle^p
&=-p\langle X_t^\eps(x)\rangle^{p-2}
        X_t^\eps(x)\cdot B_t^\eps(X_t^\eps(x)) \notag\\
&\le C\langle X_t^\eps(x)\rangle^{p-1}
        |B_t^\eps(X_t^\eps(x))| .
\label{eq:char-weight-derivative}
\end{align}
Inserting \eqref{eq:B-growth-moment}, integrating with respect to
$\phi_{0,\eps}$, and using \eqref{eq:moment-flow-representation}, gives
\begin{align}
\frac{\dd}{\dd t}M_p^\eps(t)
&\le C\int\langle X_t^\eps\rangle^{p-1}
\Bigl(1+|X_t^\eps|+\bigl(M_p^\eps(t)\bigr)^{1/p}
        +M_{p,\omega}^{1/p}\Bigr)\phi_{0,\eps}\dd x \notag\\
&\le C\Bigl(1+M_p^\eps(t)
        +\bigl(M_p^\eps(t)\bigr)^{(p-1)/p}
        \bigl(M_p^\eps(t)\bigr)^{1/p}
        +M_{p,\omega}^{1/p}\bigl(M_p^\eps(t)\bigr)^{(p-1)/p}\Bigr) \notag\\
&\le C_T\bigl(1+M_p^\eps(t)\bigr),
\label{eq:moment-ode}
\end{align}
where the last line uses Young's inequality and $p>1$.  Gronwall's lemma and
\eqref{eq:initial-moment-approx} imply
\begin{equation}\label{eq:moment-eps}
        \sup_{0<\eps\le1}\sup_{0\le t\le T}M_p^\eps(t)\le C_T .
\end{equation}

\smallskip
\noindent\emph{Step 4: compactness of the measures and passage to the PDE.}
Let $d_{\rm BL}$ be the bounded-Lipschitz metric for narrow convergence.  The
moment bound \eqref{eq:moment-eps} implies tightness of
$\{\phi_t^\eps:0\le t\le T,\ 0<\eps\le1\}$.  By Prokhorov's theorem, together
with lower semicontinuity of moments, the set
\[
        \mathcal K_C:=\left\{\mu\in\calP(\R^d):
        \int\langle x\rangle^p\dd\mu\le C\right\}
\]
is compact for the narrow topology.  Equivalently, in Wasserstein language,
uniform $p$-moment bounds imply relative compactness for all $W_q$ with
$0<q<p$; see the compactness criterion for Wasserstein spaces
\cite[Proposition 7.1.5]{savare2008gradientflows}.  Moreover, for every
$f\in C_b^1(\R^d)$ with $\norm{f}_{L^\infty}+\Lip(f)\le1$,
\begin{equation}\label{eq:time-equicont}
        \left|\frac{\dd}{\dd t}\int f\dd\phi_t^\eps\right|
        =\left|\int \grad f\cdot(-B_t^\eps)\dd\phi_t^\eps\right|
        \le C_T,
\end{equation}
by \eqref{eq:force-growth} and \eqref{eq:moment-eps}.  Hence
$t\mapsto\phi_t^\eps$ is equicontinuous as a curve in the compact metric space
$(\mathcal K_C,d_{\rm BL})$.  Arzela-Ascoli gives a subsequence, not relabeled,
and a narrowly continuous curve $t\mapsto\phi_t$ such that
\begin{equation}\label{eq:eps-limit-phi}
        \sup_{0\le t\le T}d_{\rm BL}(\phi_t^\eps,\phi_t)\to0.
\end{equation}
In addition, \eqref{eq:Linfty-eps} gives, by Banach-Alaoglu and a diagonal
argument over balls,
\begin{equation}\label{eq:weakstar-local}
        \phi^\eps\overset{*}{\rightharpoonup}\phi
        \quad\text{in }L^\infty((0,T)\times B_R)
        \quad\text{for every }R<\infty.
\end{equation}
{
It remains to identify the nonlinear force in the weak formulation.
Fix a compact set \(Q\Subset \mathbb R^d\). We prove that, for every
\(s\in[0,1]\),
\begin{equation}
        K_{s,\varepsilon}*\phi_t^\varepsilon
        \longrightarrow K_s*\phi_t
        \quad\text{in }L^1(0,T;L^\infty(Q)).
\end{equation}
This is the only point in the compactness argument where one has to keep
track of the regularization error.

Let \(\chi\in C_c^\infty(\mathbb R^d)\) satisfy
\(0\leq \chi\leq 1\), \(\chi=1\) on \(B_1(0)\), and
\(\chi=0\) outside \(B_2(0)\). For \(\delta>0\), write
\[
        K_s^{<\delta}(z):=\chi(z/\delta)K_s(z),
        \qquad
        K_s^{\geq\delta}(z):=(1-\chi(z/\delta))K_s(z),
\]
and define \(K_{s,\varepsilon}^{<\delta}\) and
\(K_{s,\varepsilon}^{\geq\delta}\) in the same way, with \(K_s\) replaced
by \(K_{s,\varepsilon}\). We decompose
\begin{equation}\label{eq:322}
\begin{aligned}
K_{s,\varepsilon}*\phi_t^\varepsilon
        -K_s*\phi_t
&=
K_{s,\varepsilon}^{<\delta}*\phi_t^\varepsilon
        -K_s^{<\delta}*\phi_t                                    \\
&\quad+
\bigl(K_{s,\varepsilon}^{\geq\delta}
        -K_s^{\geq\delta}\bigr)*\phi_t^\varepsilon                 \\
&\quad+
K_s^{\geq\delta}*(\phi_t^\varepsilon-\phi_t).
\end{aligned}
\end{equation}

We first estimate the singular part. Since
$ |K_{s,\varepsilon}(z)|\leq C_s(1+|z|^s),
        |K_s(z)|\leq C_s(1+|z|^s),$
and the kernels \(K_{s,\varepsilon}^{<\delta}\) and
\(K_s^{<\delta}\) are supported in \(B_{2\delta}(0)\), the uniform
\(L^\infty\)-bound on \(\phi_t^\varepsilon\), together with the weak-\(*\) lower semicontinuity of the \(L^\infty\)-norm, which gives the \(L^\infty\)-bound for \(\phi_t\), gives
\begin{equation}
\begin{aligned}
\sup_{t\in[0,T]}\sup_{x\in Q}
\left|K_{s,\varepsilon}^{<\delta}*\phi_t^\varepsilon(x)\right|
&\leq
C_T\int_{|z|<2\delta} (1+|z|^s)\,dz,                                \\
\sup_{t\in[0,T]}\sup_{x\in Q}
\left|K_s^{<\delta}*\phi_t(x)\right|
&\leq
C_T\int_{|z|<2\delta} (1+|z|^s)\,dz.
\end{aligned}
\end{equation}
Hence both singular terms tend to zero as \(\delta\downarrow0\), uniformly for
\(t\in[0,T]\) and \(\varepsilon\in(0,1]\).

Next we control the regularization error away from the singularity. For
fixed \(\delta>0\),
\begin{equation}
        \eta_\delta(\varepsilon)
        :=
        \sup_{|z|\geq\delta}
        \frac{|K_{s,\varepsilon}(z)-K_s(z)|}{1+|z|^s}
        \longrightarrow 0
        \qquad\text{as }\varepsilon\downarrow0.
\end{equation}
Indeed, for \(z\neq0\), $K_{s,\varepsilon}(z)
=
(1+s)(|z|^2+\varepsilon)^{(s-1)/2}z,
K_s(z)
=
(1+s)|z|^{s-1}z,$
and the convergence is uniform on \(\{|z|\geq\delta\}\) after division by
\(1+|z|^s\). Therefore, using the uniform \(p\)-moment bound and
\(s\leq1<p\),
\begin{equation}
    \begin{aligned}
\sup_{t\in[0,T]}\sup_{x\in Q}
\left|
\bigl(K_{s,\varepsilon}^{\geq\delta}
        -K_s^{\geq\delta}\bigr)*\phi_t^\varepsilon(x)
\right|
&\leq
\eta_\delta(\varepsilon)
\sup_{t\in[0,T]}\int_{\mathbb R^d}
        \bigl(1+|x-y|^s\bigr)\phi_t^\varepsilon(y)\,dy             \\
&\leq
C_{T,Q}\eta_\delta(\varepsilon)
\longrightarrow0
\qquad\text{as }\varepsilon\downarrow0 .
\end{aligned}
\end{equation}

It remains to pass to the limit in the last term of \eqref{eq:322}. For fixed
\(\delta>0\), the kernel \(K_s^{\geq\delta}\) is continuous and satisfies
the growth bound
\begin{equation}
    |K_s^{\geq\delta}(x-y)|\leq C_{\delta,Q}(1+|y|^s),
        \qquad x\in Q .
\end{equation}
We claim that
\begin{equation}\label{eq:327}
    K_s^{\geq\delta}*\phi_t^\varepsilon
        \longrightarrow
        K_s^{\geq\delta}*\phi_t
        \quad\text{in }L^\infty(0,T;L^\infty(Q)).
\end{equation}
To prove this, let \(R>1\) and choose a cutoff
\(\theta_R\in C_c^\infty(\mathbb R^d)\) with
\(\theta_R=1\) on \(B_R(0)\) and \(\theta_R=0\) outside \(B_{2R}(0)\).
For \(x\in Q\), the functions
$ y\mapsto \theta_R(y)K_s^{\geq\delta}(x-y) $
form an equibounded and equi-Lipschitz family. Since $ \sup_{t\in[0,T]}d_{\mathrm{BL}}
        \bigl(\phi_t^\varepsilon,\phi_t\bigr)\to0,$
we get
\begin{equation}\label{eq:328}
    \sup_{t\in[0,T]}\sup_{x\in Q}
\left|
\int \theta_R(y)K_s^{\geq\delta}(x-y)
        \,d(\phi_t^\varepsilon-\phi_t)(y)
\right|
\longrightarrow0
\qquad\text{as }\varepsilon\downarrow0
\end{equation}
for each fixed \(R\). On the complement of \(B_R(0)\), the uniform
\(p\)-moment bound gives
\begin{equation}\label{eq:329}
    \begin{aligned}
\sup_{t\in[0,T]}\sup_{x\in Q}
\int_{|y|>R}
        |K_s^{\geq\delta}(x-y)|
        \,d\phi_t^\varepsilon(y)
&\leq
C_{Q,\delta}R^{s-p}
        \sup_{t\in[0,T]}\int \langle y\rangle^p
        d\phi_t^\varepsilon(y),                                  \\
\sup_{t\in[0,T]}\sup_{x\in Q}
\int_{|y|>R}
        |K_s^{\geq\delta}(x-y)|
        \,d\phi_t(y)
&\leq
C_{Q,\delta}R^{s-p}
        \sup_{t\in[0,T]}\int \langle y\rangle^p
        d\phi_t(y).
\end{aligned}
\end{equation}
Since \(p>s\), the right-hand sides tend to zero as \(R\to\infty\),
uniformly in \(\varepsilon\) and \(t\). Combining \eqref{eq:328} and
\eqref{eq:329} proves \eqref{eq:327}.

Putting together \eqref{eq:322}--\eqref{eq:327}, we first let
\(\varepsilon\downarrow0\) for fixed \(\delta\), and then let
\(\delta\downarrow0\). This proves
\[
        K_{s,\varepsilon}*\phi_t^\varepsilon
        \longrightarrow
        K_s*\phi_t
        \quad\text{in }L^1(0,T;L^\infty(Q)).
\]
The same argument, with the fixed measure \(\omega(y)\,dy\) in place of
\(\phi_t^\varepsilon\), gives
\[
        K_{a,\varepsilon}*\omega
        \longrightarrow
        K_a*\omega
        \quad\text{in }L^1(0,T;L^\infty(Q)).
\]
Therefore
\begin{equation}
    \label{eq:force-conv-local}
        B_t^\varepsilon
        =
        K_{a,\varepsilon}*\omega
        -
        K_{r,\varepsilon}*\phi_t^\varepsilon
        \longrightarrow
        K_a*\omega
        -
        K_r*\phi_t
        =
        B_t
        \quad\text{in }L^1(0,T;L^\infty(Q)).
\end{equation}
Passing to the limit in the weak formulation of the regularized equation
then gives the distributional formulation \eqref{eq:weaksol}. The mass,
positivity, maximum bound, and moment bound pass to the limit by narrow
lower semicontinuity and weak-\(*\) compactness.
}

\smallskip
\noindent\emph{Step 5: compactness of the flows and the Lagrangian representation.}
Fix $R<\infty$.  We first prove the equi-Lipschitz estimates for the flows on
$[0,T]\times B_R$.  From \eqref{eq:B-growth-moment} and
\eqref{eq:moment-eps}, every characteristic starting in $B_R$ satisfies
\begin{equation}\label{eq:flow-growth-ode}
        \frac{\dd}{\dd t}|X_t^\eps(x)|
        \le C_T(1+|X_t^\eps(x)|),
        \qquad x\in B_R .
\end{equation}
Gronwall's lemma gives
\begin{equation}\label{eq:flow-growth-local}
        \sup_{0<\eps\le1}\sup_{0\le t\le T}\sup_{x\in B_R}
        |X_t^\eps(x)|\le C_{T,R}.
\end{equation}
Consequently, for $0\le s<t\le T$ and $x\in B_R$,
\begin{equation}\label{eq:time-lip-flow}
        |X_t^\eps(x)-X_s^\eps(x)|
        \le \int_s^t|B_\tau^\eps(X_\tau^\eps(x))|\dd\tau
        \le C_{T,R}|t-s|.
\end{equation}
This is equi-Lipschitz continuity in time.

Next we prove equi-Lipschitz continuity in the initial point.  By Lemma
\ref{lem:regularized-derivatives}, the maximum estimate \eqref{eq:Linfty-eps}, and
mass conservation,
\begin{align}
\norm{DB_t^\eps}_{L^\infty(\R^d)}
&\le \norm{D^2\psi_{a,\eps}*\omega}_{L^\infty}
    +\norm{D^2\psi_{r,\eps}*\phi_t^\eps}_{L^\infty} \notag\\
&\le C\bigl(\norm{\omega}_{L^1}+\norm{\omega}_{L^\infty}
    +1+\norm{\phi_t^\eps}_{L^\infty}\bigr)
\le C_T,
\label{eq:uniform-lip-Beps}
\end{align}
uniformly in $t\in[0,T]$ and $\eps\in(0,1]$.  Therefore, if $x,y\in B_R$,
\begin{align}
\frac{\dd}{\dd t}|X_t^\eps(x)-X_t^\eps(y)|
&\le |B_t^\eps(X_t^\eps(x))-B_t^\eps(X_t^\eps(y))| \notag\\
&\le C_T|X_t^\eps(x)-X_t^\eps(y)|.
\end{align}
A second application of Gronwall's lemma yields
\begin{equation}\label{eq:space-lip-flow}
        |X_t^\eps(x)-X_t^\eps(y)|\le e^{C_Tt}|x-y|,
        \qquad x,y\in B_R,
        \quad 0\le t\le T .
\end{equation}
Equations \eqref{eq:time-lip-flow} and \eqref{eq:space-lip-flow} are the
announced equi-Lipschitz estimates on $[0,T]\times B_R$.

Arzela--Ascoli gives, after a diagonal extraction in $R$, locally uniform
convergence $X^\eps\to X$ on $[0,T]\times\R^d$.  The convergence of the forces
can be used in the integral equation because \eqref{eq:force-conv-local} is
local in $L^1_tL^\infty_x$ and the trajectories starting from $B_R$ remain in
$B_{C_{T,R}}$.  Indeed,
\begin{align*}
&\sup_{x\in B_R}\left|\int_0^t
\Bigl[B_\tau^\eps(X_\tau^\eps(x))
       -\bigl(\grad V-K_r*\phi_\tau\bigr)(X_\tau(x))\Bigr]\dd\tau\right|\\
&\qquad\le \int_0^T
\norm{B_\tau^\eps-(\grad V-K_r*\phi_\tau)}_{L^\infty(B_{C_{T,R}+1})}\dd\tau
 +C_T\int_0^T\sup_{x\in B_R}|X_\tau^\eps(x)-X_\tau(x)|\dd\tau,
\end{align*}
which tends to zero.  Passing to the limit in the integral form of
\eqref{eq:regularized-flow} gives
\begin{equation}\label{eq:limiting-characteristic}
        X_t(x)=x-\int_0^t\Bigl(\grad V(X_\tau(x))
        -K_r*\phi_\tau(X_\tau(x))\Bigr)\dd\tau
\end{equation}
for every $x$ in the full-measure set on which the diagonal convergence is
realized.

Finally, for every bounded continuous $f$, $ \int f\dd (X_t^\eps)_{\#}\phi_{0,\eps}
        \to \int f\dd (X_t)_{\#}\phi_0.$
{
Indeed, let \(f\in C_b(\mathbb R^d)\),  we prove that
 $\int_{\mathbb R^d} f(X_t^\eps(x))\phi_{0,\eps}(x)\,dx
 \longrightarrow
 \int_{\mathbb R^d} f(X_t(x))\phi_0(x)\,dx .$
Fix \(R>1\), by the local uniform convergence of the flows, there is $\eta_{\eps,R}:=\sup_{t\in[0,T]}\sup_{x\in B_R}
        |X_t^\eps(x)-X_t(x)|\longrightarrow 0 .$
Moreover, by the flow bound, \(X_t^\eps(B_R)\) and \(X_t(B_R)\) are contained
in a fixed ball \(B_{C_{T,R}}\).  Hence \(f\) is uniformly continuous on this
ball, and if \(\omega_{f,T,R}\) denotes its modulus of continuity there, then
\[
\sup_{t\in[0,T]}\sup_{x\in B_R}
   |f(X_t^\eps(x))-f(X_t(x))|
 \le \omega_{f,T,R}(\eta_{\eps,R})\longrightarrow 0 .
\]
Therefore
\[
\begin{aligned}
&\sup_{t\in[0,T]}
\left|
 \int_{\mathbb R^d} f(X_t^\eps(x))\phi_{0,\eps}(x)\,dx
 -
 \int_{\mathbb R^d} f(X_t(x))\phi_0(x)\,dx
\right|                                                   \\
&\le
 \omega_{f,T,R}(\eta_{\eps,R})\int_{B_R}\phi_{0,\eps}(x)\,dx
 + \|f\|_{L^\infty}\|\phi_{0,\eps}-\phi_0\|_{L^1(B_R)}       \\
&\quad
 + \|f\|_{L^\infty}\int_{B_R^c}\phi_{0,\eps}(x)\,dx
 + \|f\|_{L^\infty}\int_{B_R^c}\phi_0(x)\,dx .
\end{aligned}
\]
The first two terms vanish as \(\eps\downarrow0\), for fixed \(R\).  For the
tail terms, using \(\langle x\rangle^p\ge R^p\) on \(B_R^c\), we obtain
\[
 \int_{B_R^c}\phi_{0,\eps}(x)\,dx
 +
 \int_{B_R^c}\phi_0(x)\,dx
 \le
 R^{-p}\left(
 \sup_{0<\eps\le1}\int_{\mathbb R^d}\langle x\rangle^p\phi_{0,\eps}(x)\,dx
 +
 \int_{\mathbb R^d}\langle x\rangle^p\phi_0(x)\,dx
 \right).
\]
The right-hand side tends to zero as \(R\to\infty\), by
\eqref{eq:initial-moment-approx}.  Hence $\int_{\mathbb R^d} f\,d(X_t^\eps)_{\#}\phi_{0,\eps}
 \rightarrow
 \int_{\mathbb R^d} f\,d(X_t)_{\#}\phi_0 .$
Since \(\phi_t^\eps=(X_t^\eps)_{\#}\phi_{0,\eps}\), and since the preceding
argument shows that $(X_t^\eps)_{\#}\phi_{0,\eps}
        \rightharpoonup (X_t)_{\#}\phi_0$
narrowly, while also \(\phi_t^\eps\rightharpoonup \phi_t\) narrowly, uniqueness
of the narrow limit gives
\begin{equation}\label{eq:lagrangian-limit}
        \phi_t=(X_t)_{\#}\phi_0 .
\end{equation}
}

\smallskip
\noindent\emph{Step 6: uniqueness.}
Let $\phi_t=(X_t)_{\#}\phi_0$ and
$\widetilde\phi_t=(Y_t)_{\#}\phi_0$ be two Lagrangian solutions satisfying the
same $L^\infty$ and moment bounds.  We give the complete localized estimate. 
{The argument is similar to the Lagrangian stability estimates for bounded-density solutions in \cite{carrillo2014derivation}; the localization to the balls $B_R$ and the tail term $\Theta_R$ are needed here because both kernels grow at infinity and the data are not compactly supported.}
The growth estimate \eqref{eq:force-growth} with $\varepsilon=0$ and Gronwall's lemma imply
\begin{equation}\label{eq:flow-growth-unique}
        |X_t(x)|+|Y_t(x)|\le C_T(1+|x|)
        \qquad\text{for }0\le t\le T
\end{equation}
for $\phi_0$-a.e. $x$.  For $R>0$ define
\begin{equation}\label{eq:DR-def}
        D_R(t):=\operatorname*{ess\,sup}_{|x|\le R}
        |X_t(x)-Y_t(x)|
\end{equation}
and
\begin{equation}\label{eq:tail-unique}
        \Theta_R:=\int_{|z|>R}(1+|z|^r)\phi_0(z)\dd z
        +R^r\phi_0(\{|z|>R\}).
\end{equation}
Since the initial $p$-moment is finite and $p>1\ge r$, $\Theta_R\to0$.

Fix $R$ and $x\in B_R$.  Subtracting the two characteristic equations gives
\begin{align}
|X_t(x)-Y_t(x)|
&\le \int_0^t |\grad V(X_s(x))-\grad V(Y_s(x))|\dd s \notag\\
&\quad +\int_0^t \int_{\R^d}
\bigl|K_r(X_s(x)-X_s(z))-K_r(Y_s(x)-Y_s(z))\bigr|
\phi_0(z)\dd z\dd s .
\label{eq:uniq-start}
\end{align}
The first term is bounded by
\begin{equation}\label{eq:uniq-V}
        |\grad V(X_s(x))-\grad V(Y_s(x))|
        \le \norm{D^2V}_{L^\infty}|X_s(x)-Y_s(x)|
        \le CD_R(s),
\end{equation}
where constant $C$ comes from Lemma \ref{lem:regularized-derivatives} with $\varepsilon=0,f=\omega$.
For the interaction term split the $z$-integral into $B_R$ and $B_R^c$.
On $B_R$, if $D_R(s)\le1$, then $ |(X_s(x)-X_s(z))-(Y_s(x)-Y_s(z))|\le2D_R(s).$
Changing variables $y=X_s(z)$ and using $\phi_s=(X_s)_\#\phi_0$, Lemma
\ref{lem:averaged-lip} gives
\begin{align}
&\int_{B_R}\bigl|K_r(X_s(x)-X_s(z))-K_r(Y_s(x)-Y_s(z))\bigr|
\phi_0(z)\dd z \notag\\
&\qquad\le
\int_{\R^d}\sup_{|h|\le2D_R(s)}
        |K_r(X_s(x)-y+h)-K_r(X_s(x)-y)|\phi_s(y)\dd y
\le C_TD_R(s).
\label{eq:uniq-inside}
\end{align}
On $B_R^c$, the growth bound $|K_r(w)|\le C(1+|w|^r)$ and
\eqref{eq:flow-growth-unique} imply, for $x\in B_R$,
\begin{align}
&\int_{B_R^c}\bigl|K_r(X_s(x)-X_s(z))-K_r(Y_s(x)-Y_s(z))\bigr|
\phi_0(z)\dd z \notag\\
&\qquad\le C_T\int_{|z|>R}(1+R^r+|z|^r)\phi_0(z)\dd z
\le C_T\Theta_R .
\label{eq:uniq-tail}
\end{align}
Combining \eqref{eq:uniq-start}--\eqref{eq:uniq-tail} and taking the essential
supremum over $x\in B_R$ yields, as long as $D_R\le1$,
\begin{equation}\label{eq:localized-gronwall}
        D_R(t)\le C_T\int_0^tD_R(s)\dd s+C_T\Theta_R.
\end{equation}
By a standard continuity argument, if $R$ is large enough so that
$C_Te^{C_TT}\Theta_R<1$, then the assumption $D_R\le1$ closes on $[0,T]$.
Gronwall's lemma gives
\begin{equation}\label{eq:DR-tail-bound}
        D_R(t)\le C_Te^{C_TT}\Theta_R,
        \qquad 0\le t\le T.
\end{equation}
Letting $R\to\infty$ gives $X_t(x)=Y_t(x)$ for $\phi_0$-a.e. $x$ and every
$t\in[0,T]$.  Hence $\phi_t=\widetilde\phi_t$, proving uniqueness.

\end{proof}

\begin{proof}[Proof of Proposition \ref{prop:energy-diss}]
Fix $T>0$ and set $  m:=\max\{a,r\}.$
We use the squared-radius approximation constructed in the proof of
Theorem~\ref{thm:wellposed}. Thus
\[
    \psi_{\sigma,\varepsilon}(x)
    :=
    (|x|^2+\varepsilon)^{(1+\sigma)/2},
    \qquad
    K_{\sigma,\varepsilon}
    :=
    \nabla\psi_{\sigma,\varepsilon},
    \qquad \sigma\in\{a,r\},
\]
and $\phi_t^\varepsilon=(X_t^\varepsilon)_\#\phi_{0,\varepsilon}$
denotes the corresponding regularized solution. Set
\[
    B_t^\varepsilon
    :=
    K_{a,\varepsilon}*\omega
    -
    K_{r,\varepsilon}*\phi_t^\varepsilon,
    \qquad
    B_t
    :=
    K_a*\omega-K_r*\phi_t.
\]
By the compactness argument in the proof of Theorem~\ref{thm:wellposed},
\begin{equation}
\label{eq:uniform-narrow-convergence-energy}
    \sup_{0\leq t\leq T}
    d_{\rm BL}(\phi_t^\varepsilon,\phi_t)
    \longrightarrow0,
\end{equation}
see \eqref{eq:eps-limit-phi}, and, for every $R<\infty$,
\begin{equation}
\label{eq:local-force-convergence-energy}
    B^\varepsilon\longrightarrow B
    \qquad\text{in }
    L^1(0,T;L^\infty(B_R)),
\end{equation}
see \eqref{eq:force-conv-local}.

\medskip

\noindent
\textit{Step 1: uniform integrability of the critical $p$-moment.}

As in the proof of Theorem~\ref{thm:wellposed}, choose the
regularized initial data $\phi_{0,\varepsilon}$ by truncating
$\phi_0$, mollifying, and renormalizing. The truncation and
mollification are chosen so that
$\phi_{0,\varepsilon}\to\phi_0
    \text{ in }L^1(\mathbb R^d),$
and, in addition,
\begin{equation}
\label{eq:initial-p-uniform-integrability}
    \lim_{R\to\infty}
    \sup_{0<\varepsilon\leq1}
    \int_{\{|x|>R\}}
    \langle x\rangle^p
    \phi_{0,\varepsilon}(x)\,dx
    =0.
\end{equation}
Indeed, the truncation radii may be chosen to tend to infinity as
$\varepsilon\downarrow0$, while the mollifiers may be chosen with
support shrinking to the origin. Since $\int_{\mathbb R^d}
    \langle x\rangle^p\phi_0(x)\,dx<\infty, $
the $p$-moment of the truncated tail tends to zero; mollification enlarges the support only by the radius of the mollifier and
preserves the $p$-moment up to a uniform multiplicative constant.
Since the normalization factors converge to one,
\eqref{eq:initial-p-uniform-integrability} follows.

We next propagate \eqref{eq:initial-p-uniform-integrability} to
finite times. The estimates \eqref{eq:B-growth-moment} and \eqref{eq:moment-eps} from the proof of
Theorem~\ref{thm:wellposed}, followed by Gronwall's inequality along the
characteristics, give
\begin{equation}
\label{eq:flow-growth-energy}
    |X_t^\varepsilon(x)|
    \leq C_T(1+|x|),
    \qquad
    0\leq t\leq T,
\end{equation}
with $C_T$ independent of $\varepsilon$. Hence
$\langle X_t^\varepsilon(x)\rangle^p
    \leq C_T\langle x\rangle^p.$
Moreover, if $|X_t^\varepsilon(x)|>R$, then
\eqref{eq:flow-growth-energy} implies
$|x|>\frac{R}{C_T}-1.$
Using $\phi_t^\varepsilon=(X_t^\varepsilon)_\#\phi_{0,\varepsilon}$,
we therefore obtain
\begin{align}
    &\sup_{0<\varepsilon\leq1}
    \sup_{0\leq t\leq T}
    \int_{\{|y|>R\}}
    \langle y\rangle^p\,d\phi_t^\varepsilon(y)
    \nonumber\\
    &\qquad\leq
    C_T
    \sup_{0<\varepsilon\leq1}
    \int_{\{|x|>R/C_T-1\}}
    \langle x\rangle^p
    \phi_{0,\varepsilon}(x)\,dx
    \longrightarrow0
    \qquad (R\to\infty).
    \label{eq:propagated-p-uniform-integrability}
\end{align}
Thus the $p$-moments are uniformly integrable on every
finite time interval. The same uniform tail estimate holds for the limiting solution.
Indeed, Theorem~\ref{thm:wellposed} gives
$\phi_t=(X_t)_\#\phi_0$, and the limiting characteristics satisfy $ |X_t(x)|\leq C_T(1+|x|), 0\leq t\leq T.$
Therefore,
\begin{align*}
    \sup_{0\leq t\leq T}
    \int_{\{|y|>R\}}
    \langle y\rangle^p\,d\phi_t(y)
    &\leq
    C_T
    \int_{\{|x|>R/C_T-1\}}
    \langle x\rangle^p\,d\phi_0(x)
    \longrightarrow0
\end{align*}
as $R\to\infty$.

\medskip
\noindent
\textit{Step 2: the regularized energy identity.}

Define
\begin{align}
    E_\varepsilon(\mu)
    &:=
    \iint
    \psi_{a,\varepsilon}(x-y)
    \,d\mu(x)\omega(y)\,dy
    \nonumber\\
    &\qquad
    -\frac12
    \iint
    \psi_{r,\varepsilon}(x-y)
    \,d\mu(x)d\mu(y).
    \label{eq:regularized-energy}
\end{align}
The first variation of $E_\varepsilon$ at $\phi_t^\varepsilon$ is $ F_t^\varepsilon
    :=
    \psi_{a,\varepsilon}*\omega
    -
    \psi_{r,\varepsilon}*\phi_t^\varepsilon,
    \nabla F_t^\varepsilon=B_t^\varepsilon.$
Uniformly in $\varepsilon\in(0,1]$,
\[
    \psi_{\sigma,\varepsilon}(z)
    \leq C(1+|z|^{1+\sigma}),
    \qquad
    |K_{\sigma,\varepsilon}(z)|
    \leq C(1+|z|^\sigma).
\]
Using
\[
    |x-y|^{1+\sigma}
    \leq C_\sigma
    \bigl(|x|^{1+\sigma}+|y|^{1+\sigma}\bigr),
    \qquad
    |x-y|^\sigma
    \leq C_\sigma
    \bigl(1+|x|^\sigma+|y|^\sigma\bigr),
\]
together with the finite-time $p$-moment bounds, we obtain
\begin{equation}
\label{eq:F-B-growth}
    |F_t^\varepsilon(x)|
    \leq C_T(1+|x|^{1+m}),
    \qquad
    |B_t^\varepsilon(x)|
    \leq C_T(1+|x|^m).
\end{equation}

Let $\chi\in C_c^\infty(\mathbb R^d)$ satisfy
$0\leq\chi\leq1$, $\chi=1$ on $B_1$, and
$\chi=0$ outside $B_2$, and put
$\chi_R(x):=\chi(x/R).$
Then $|\nabla\chi_R|\leq C/R$. Multiplying the regularized
continuity equation by $\chi_R F^\varepsilon_t$ and integrating by
parts gives
\begin{align}
    \int_{\mathbb R^d}
    \chi_R F^\varepsilon_t
    \partial_t\phi^\varepsilon_t\,dx
    &=
    -\int_{\mathbb R^d}
    \chi_R|B^\varepsilon_t|^2
    d\phi^\varepsilon_t
    \nonumber\\
    &\qquad
    -
    \int_{\mathbb R^d}
    F^\varepsilon_t
    B^\varepsilon_t\cdot\nabla\chi_R
    d\phi^\varepsilon_t.
    \label{eq:cutoff-ibp}
\end{align}
The last term is supported in
$A_R:=B_{2R}\setminus B_R$. By \eqref{eq:F-B-growth},
\begin{align*}
    \left|
    \int
    F^\varepsilon_t
    B^\varepsilon_t\cdot\nabla\chi_R
    d\phi^\varepsilon_t
    \right|
    &\leq
    \frac{C_T}{R}
    \int_{A_R}
    (1+|x|^{1+m})(1+|x|^m)
    d\phi^\varepsilon_t(x)\\
    &\leq
    C_T
    \int_{A_R}
    \bigl(
        R^{-1}+|x|^m+|x|^{2m}
    \bigr)
    d\phi^\varepsilon_t(x).
\end{align*}
Since $m<1$ and $p\geq1+m>2m$, the uniform $p$-moment bound gives
\[
    \sup_{\varepsilon\in(0,1]}
    \sup_{0\leq t\leq T}
    \left|
    \int
    F^\varepsilon_t
    B^\varepsilon_t\cdot\nabla\chi_R
    d\phi^\varepsilon_t
    \right|
    \longrightarrow0
    \qquad\text{as }R\to\infty.
\]
Thus the integration by parts on $\mathbb R^d$ is justified after letting $R\to\infty$.

For the smooth regularized solution, differentiating
\eqref{eq:regularized-energy} in time and using the symmetry of
$\psi_{r,\varepsilon}$ gives
$  \frac{d}{dt}E_\varepsilon(\phi^\varepsilon_t)
    =
    \int_{\mathbb R^d}
    F^\varepsilon_t
    \partial_t\phi^\varepsilon_t\,dx.$
Therefore, after letting $R\to\infty$ in
\eqref{eq:cutoff-ibp}, $\frac{d}{dt}E_\varepsilon(\phi^\varepsilon_t)
    =
    -\int_{\mathbb R^d}
    |B^\varepsilon_t|^2\,d\phi^\varepsilon_t.$
Integrating in time yields, for every
$0\leq s\leq t\leq T$,
\begin{equation}
\label{eq:regularized-energy-identity}
    E_\varepsilon(\phi^\varepsilon_t)
    +
    \int_s^t\int_{\mathbb R^d}
    |B^\varepsilon_\tau|^2
    d\phi^\varepsilon_\tau\,d\tau
    =
    E_\varepsilon(\phi^\varepsilon_s).
\end{equation}

\medskip
\noindent
\textit{Step 3: convergence of the energy.}

We prove that, for every fixed $t\in[0,T]$,
\begin{equation}
\label{eq:energy-convergence}
    E_\varepsilon(\phi_t^\varepsilon)
    \longrightarrow E(\phi_t)
    \qquad\text{as }\varepsilon\downarrow0.
\end{equation}
We treat the attraction and self-interaction terms separately.

First, \eqref{eq:propagated-p-uniform-integrability} implies that, for every $\sigma\in\{a,r\}$,
\begin{equation}
\label{eq:lower-moment-tail}
    \lim_{R\to\infty}
    \sup_{\substack{0<\varepsilon\leq1\\0\leq t\leq T}}
    \int_{\{|x|>R\}}
    |x|^{1+\sigma}\,d\phi_t^\varepsilon(x)
    =0,
\end{equation}
since $p\geq 1+m\geq 1+\sigma$. Similarly,
\begin{equation}
\label{eq:omega-lower-moment-tail}
    \lim_{R\to\infty}
    \int_{\{|y|>R\}}
    |y|^{1+\sigma}\omega(y)\,dy
    =0,
\end{equation}
since $\omega$ has finite $p$-moment and $1+\sigma\leq p$.
By Step~1, the corresponding estimate also holds for the limiting solution
$\phi_t$, uniformly for $t\in[0,T]$.

Let $\chi_R$ be the cutoff introduced above. We first consider the
attraction term
\[
    A_\varepsilon(t)
    :=
    \iint
    \psi_{a,\varepsilon}(x-y)
    \,d\phi_t^\varepsilon(x)\omega(y)\,dy,
\]
and
\[
    A(t)
    :=
    \iint
    \psi_a(x-y)
    \,d\phi_t(x)\omega(y)\,dy.
\]
Decompose $  A_\varepsilon(t)
    =
    A_{\varepsilon,R}(t)
    +
    \mathcal R_{\varepsilon,R}^{a}(t),$
where $A_{\varepsilon,R}(t)
    :=
    \iint
    \chi_R(x)\chi_R(y)
    \psi_{a,\varepsilon}(x-y)
    \,d\phi_t^\varepsilon(x)\omega(y)\,dy.$
Define $A_R(t)$ analogously with
$\psi_a$ and $\phi_t$ in place of
$\psi_{a,\varepsilon}$ and $\phi_t^\varepsilon$.

For each fixed $R$, the kernels $(x,y)\longmapsto
    \chi_R(x)\chi_R(y)\psi_{a,\varepsilon}(x-y)$
converge uniformly on $\mathbb R^d\times\mathbb R^d$ to
$ (x,y)\mapsto
    \chi_R(x)\chi_R(y)\psi_a(x-y).$
Hence
\begin{align*}
    |A_{\varepsilon,R}(t)-A_R(t)|
    &\leq
    \iint
    \chi_R(x)\chi_R(y)
    |\psi_{a,\varepsilon}(x-y)-\psi_a(x-y)|
    \,d\phi_t^\varepsilon(x)\omega(y)\,dy\\
    &\quad+
    \left|
    \int
    G_R(x)\,d(\phi_t^\varepsilon-\phi_t)(x)
    \right|,
\end{align*}
where $  G_R(x)
    :=
    \chi_R(x)
    \int
    \chi_R(y)\psi_a(x-y)\omega(y)\,dy.$
The first term tends to zero by uniform convergence of
$\psi_{a,\varepsilon}$ on $B_{2R}\times B_{2R}$. Moreover,
$G_R\in C_c(\mathbb R^d)$; hence the second term tends to zero by the
narrow convergence $\phi_t^\varepsilon\rightharpoonup\phi_t$.
Therefore
\begin{equation}
\label{eq:localized-attraction-convergence}
    A_{\varepsilon,R}(t)\longrightarrow A_R(t)
    \qquad\text{for every fixed }R.
\end{equation}

We next control the tails uniformly with respect to $\varepsilon$. Since
 $ 1-\chi_R(x)\chi_R(y)
    \leq
    \mathbf 1_{\{|x|>R\}}
    +
    \mathbf 1_{\{|y|>R\}},$
and
$ \psi_{a,\varepsilon}(x-y)
    \leq
    C\bigl(
        1+|x|^{1+a}+|y|^{1+a}
    \bigr),$
we obtain
\begin{align*}
    |\mathcal R_{\varepsilon,R}^{a}(t)|
    &\leq
    C
    \iint_{\{|x|>R\}\cup\{|y|>R\}}
    \bigl(
        1+|x|^{1+a}+|y|^{1+a}
    \bigr)
    \,d\phi_t^\varepsilon(x)\omega(y)\,dy.
\end{align*}
Using the finite masses of $\phi_t^\varepsilon$ and $\omega$,
together with \eqref{eq:lower-moment-tail} and
\eqref{eq:omega-lower-moment-tail}, it follows that
\begin{equation}
\label{eq:attraction-tail}
    \lim_{R\to\infty}
    \sup_{\substack{0<\varepsilon\leq1\\0\leq t\leq T}}
    |\mathcal R_{\varepsilon,R}^{a}(t)|
    =0.
\end{equation}
The same argument gives $ A_R(t)\longrightarrow A(t)
    \text{ as }R\to\infty.$
Combining this observation with
\eqref{eq:localized-attraction-convergence} and
\eqref{eq:attraction-tail}, we conclude that
\begin{equation}
\label{eq:attraction-convergence}
    A_\varepsilon(t)\longrightarrow A(t).
\end{equation}

We now consider the self-interaction term
$ S_\varepsilon(t)
    :=
    \iint
    \psi_{r,\varepsilon}(x-y)
    \,d\phi_t^\varepsilon(x)d\phi_t^\varepsilon(y),$
and
$ S(t)
    :=
    \iint
    \psi_r(x-y)
    \,d\phi_t(x)d\phi_t(y).$
Set $S_{\varepsilon,R}(t)
    :=
    \iint
    \chi_R(x)\chi_R(y)
    \psi_{r,\varepsilon}(x-y)
    \,d\phi_t^\varepsilon(x)d\phi_t^\varepsilon(y),$
and define $S_R(t)$ analogously.

For fixed $R$, write
\begin{align*}
    |S_{\varepsilon,R}(t)-S_R(t)|
    &\leq
    \iint
    \chi_R(x)\chi_R(y)
    |\psi_{r,\varepsilon}(x-y)-\psi_r(x-y)|
    \,d\phi_t^\varepsilon(x)d\phi_t^\varepsilon(y)\\
    &\quad+
    \left|
    \iint
    H_R(x,y)\,
    d\bigl(
        \phi_t^\varepsilon\otimes\phi_t^\varepsilon
        -
        \phi_t\otimes\phi_t
    \bigr)(x,y)
    \right|,
\end{align*}
where $ H_R(x,y)
    :=
    \chi_R(x)\chi_R(y)\psi_r(x-y)
    \in C_c(\mathbb R^d\times\mathbb R^d).$
The first term tends to zero by the local uniform convergence of the
regularized kernel, whereas the second tends to zero because
\[
    \phi_t^\varepsilon\rightharpoonup\phi_t
    \quad\Longrightarrow\quad
    \phi_t^\varepsilon\otimes\phi_t^\varepsilon
    \rightharpoonup
    \phi_t\otimes\phi_t.
\]
Here, we remark that \eqref{eq:uniform-narrow-convergence-energy} implies $\phi_t^\varepsilon\rightharpoonup\phi_t$, see \cite[Theorem 11.3.3]{dudley2018real}.
Thus
\begin{equation}
\label{eq:localized-self-convergence}
    S_{\varepsilon,R}(t)\longrightarrow S_R(t)
    \qquad\text{for every fixed }R.
\end{equation}

To control the tails, we use
$ \psi_{r,\varepsilon}(x-y)
    \leq
    C\bigl(
        1+|x|^{1+r}+|y|^{1+r}
    \bigr),$
together with \eqref{eq:lower-moment-tail}, we obtain
\begin{equation}
\label{eq:self-tail}
    \lim_{R\to\infty}
    \sup_{\substack{0<\varepsilon\leq1\\0\leq t\leq T}}
    |S_\varepsilon(t)-S_{\varepsilon,R}(t)|
    =0.
\end{equation}
The corresponding tail estimate holds for $S(t)-S_R(t)$ as well.
Therefore
\begin{equation}
\label{eq:self-convergence}
    S_\varepsilon(t)\longrightarrow S(t).
\end{equation}

Finally, by the definition of the energy,
$ E_\varepsilon(\phi_t^\varepsilon)
    =
    A_\varepsilon(t)-\frac12S_\varepsilon(t),
    E(\phi_t)
    =
    A(t)-\frac12S(t).$
Hence \eqref{eq:attraction-convergence} and
\eqref{eq:self-convergence} imply
$  E_\varepsilon(\phi_t^\varepsilon)
    \rightarrow E(\phi_t),$
which proves \eqref{eq:energy-convergence}. The tail estimates above
also show that both terms defining $E(\phi_t)$ are finite.

\medskip
\noindent
\textit{Step 4: convergence of the dissipation.}

Since $ |B_t^\varepsilon(x)|^2
    \leq C_T(1+|x|^{2m})$
and $ 2m<1+m\leq p,$
we have
\begin{align}
    &\sup_{\varepsilon\in(0,1]}
    \int_s^t\int_{|x|>R}
    |B_\tau^\varepsilon(x)|^2
    \,d\phi_\tau^\varepsilon(x)\,d\tau
    \nonumber\\
    &\qquad\leq
    C_T
    \bigl(R^{-p}+R^{2m-p}\bigr)
    \sup_{\varepsilon,\tau\leq T}
    \int\langle x\rangle^p\,d\phi_\tau^\varepsilon(x)
    \longrightarrow0
    \qquad(R\to\infty).
    \label{eq:dissipation-tail-final}
\end{align}
The same estimate holds for $(B_\tau,\phi_\tau)$.

It remains to treat the contribution on compact sets. Write
\begin{align*}
    &\int_s^t\int
    \chi_R|B_\tau^\varepsilon|^2
    \,d\phi_\tau^\varepsilon\,d\tau
    -
    \int_s^t\int
    \chi_R|B_\tau|^2
    \,d\phi_\tau\,d\tau\\
    &=
    \int_s^t\int
    \chi_R
    \bigl(
        |B_\tau^\varepsilon|^2-|B_\tau|^2
    \bigr)
    \,d\phi_\tau^\varepsilon\,d\tau\\
    &\qquad+
    \int_s^t
    \int
    \chi_R|B_\tau|^2
    \,d(\phi_\tau^\varepsilon-\phi_\tau)\,d\tau.
\end{align*}
Since $\phi_\tau^\varepsilon$ is a probability measure and
$0\leq\chi_R\leq1$,
\begin{align*}
    &\left|
    \int_s^t\int
    \chi_R
    \bigl(
        |B_\tau^\varepsilon|^2-|B_\tau|^2
    \bigr)
    \,d\phi_\tau^\varepsilon\,d\tau
    \right|\\
    &\qquad\leq C_{T,R}\int_s^t\|B_\tau^\varepsilon-B_\tau\|_{L^\infty(B_{2R})}\left(\int\chi_R\,d\phi_\tau^\varepsilon\right)\,d\tau
    \\
    &\qquad\leq
    C_{T,R}
    \int_s^t
    \|B_\tau^\varepsilon-B_\tau\|_{L^\infty(B_{2R})}\,d\tau
    \longrightarrow0
\end{align*}
by \eqref{eq:local-force-convergence-energy}.

For the second term, $B_\tau$ is continuous. For positive
exponents, this follows from the continuity of $K_\sigma$; for
$\sigma=0$, it follows from the averaged translation estimate for
the singular force together with the $L^1\cap L^\infty$ bounds on the
densities; see Lemma \ref{lem:averaged-lip}. Thus
$ x\mapsto\chi_R(x)|B_\tau(x)|^2$
belongs to $C_c(\mathbb R^d)$. By
\eqref{eq:uniform-narrow-convergence-energy},
$ \int
    \chi_R|B_\tau|^2\,d\phi_\tau^\varepsilon
    \rightarrow
    \int
    \chi_R|B_\tau|^2\,d\phi_\tau$
for every $\tau$, and the integrands are uniformly bounded on
$[0,T]$. The dominated convergence theorem in time therefore yields convergence
of the second term.

Combining the compact convergence with
\eqref{eq:dissipation-tail-final} yields
\begin{equation}
\label{eq:dissipation-convergence-final}
    \int_s^t\int_{\mathbb R^d}
    |B_\tau^\varepsilon|^2
    \,d\phi_\tau^\varepsilon\,d\tau
    \longrightarrow
    \int_s^t\int_{\mathbb R^d}
    |B_\tau|^2
    \,d\phi_\tau\,d\tau.
\end{equation}

\medskip
\noindent
\textit{Step 5: passage to the limit.}

Passing to the limit $\varepsilon\downarrow0$ in
\eqref{eq:regularized-energy-identity}, using
\eqref{eq:energy-convergence} and
\eqref{eq:dissipation-convergence-final}, gives
\eqref{eq:energy-dissipation-integrated}.

The identity $ E(\phi_t)
    =
    E(\phi_s)
    -
    \int_s^t\int_{\mathbb R^d}
    |B_\tau|^2\,d\phi_\tau\,d\tau$
shows that $t\mapsto E(\phi_t)$ is the primitive of a locally integrable function. Hence
$E(\phi_\cdot)\in AC_{\rm loc}([0,\infty))$, and differentiation
for almost every $t$ gives
\eqref{eq:energy-dissipation-differential}.

Finally, since $\phi_t$ is a probability measure,
$\left(
        \int_{\mathbb R^d}|B_t|\,d\phi_t
    \right)^2
    \leq
    \int_{\mathbb R^d}|B_t|^2\,d\phi_t.$
Taking $s=0$ in \eqref{eq:energy-dissipation-integrated} and using
$E(\phi_T)\geq\inf E$ gives
$\int_0^T
    \left(
        \int_{\mathbb R^d}|B_t|\,d\phi_t
    \right)^2dt
    \leq E(\phi_0)-\inf E.$
This proves \eqref{eq:averaged-dissipation}. Since the right-hand
side of \eqref{eq:averaged-dissipation} tends to zero as $T\to\infty$, there exists a sequence
$t_k\to\infty$ satisfying \eqref{eq:vanishing-dissipation}.
\end{proof}

\begin{proof}[Proof of Theorem \ref{thm:subseq-conv}]
We first verify that \(E\) is bounded from below on
\(\mathcal P(\mathbb R^d)\), following the same argument as in the proof of
Proposition~\ref{prop:moment-anchor}. Recall that
\[
E(\mu)
=
\iint_{\mathbb R^d\times\mathbb R^d}
|x-y|^{1+a}\,d\mu(x)\,\omega(y)\,dy
-\frac12
\iint_{\mathbb R^d\times\mathbb R^d}
|x-z|^{1+r}\,d\mu(x)\,d\mu(z).
\]Suppose first that \(a>r\). Since
\(M:=\int_{\mathbb R^d}\omega(y)\,dy>0\), there exist \(R>0\)
and \(m>0\) such that $\int_{B_R}\omega(y)\,dy=m.$
For \(|x|\ge 2R\) and \(y\in B_R\), we have
\(|x-y|\ge |x|/2\). Consequently, $\int_{\mathbb R^d}|x-y|^{1+a}\omega(y)\,dy
\ge m\,2^{-(1+a)}|x|^{1+a}
\text{ for }|x|\ge 2R.$
Since the left-hand side is nonnegative, decreasing the constant
if necessary gives constants \(c,C>0\) such that
$\int_{\mathbb R^d}|x-y|^{1+a}\omega(y)\,dy
\ge c|x|^{1+a}-C
\text{ for all }x\in\mathbb R^d.$
On the other hand, using
$|x-z|^{1+r}
\le 2^r\bigl(|x|^{1+r}+|z|^{1+r}\bigr),$
and that \(\mu\) is a probability measure, we obtain $\frac12\iint |x-z|^{1+r}\,d\mu(x)d\mu(z)
\le 2^r\int |x|^{1+r}\,d\mu(x).$
Because \(r<a\), for every \(\varepsilon>0\) there is a constant
\(C_\varepsilon<\infty\) such that
$|x|^{1+r}\le \varepsilon |x|^{1+a}+C_\varepsilon
\text{ for all }x\in\mathbb R^d.$
Integrating the preceding estimates against \(\mu\) and choosing
\(\varepsilon\) sufficiently small yields
\[
E(\mu)\ge c'\int_{\mathbb R^d}|x|^{1+a}\,d\mu(x)-C'
\ge -C'>-\infty,
\]
where \(c',C'>0\) are independent of \(\mu\).

Now suppose that \(a=r\), and set \(q:=1+a\in[1,2)\) and $\bar\omega:=M^{-1}\omega(y)\,dy,
M:=\int_{\mathbb R^d}\omega(y)\,dy\ge1.$
For any \(\mu\in\mathcal P(\mathbb R^d)\), we have
\[
\begin{aligned}
E(\mu)
={}&(M-1)\iint |x-y|^q\,d\mu(x)d\bar\omega(y)\\
&+\iint |x-y|^q\,d\mu(x)d\bar\omega(y)
-\frac12\iint |x-z|^q\,d\mu(x)d\mu(z).
\end{aligned}
\]
The first term on the right is nonnegative. To bound the remaining
terms, use the conditional negative definiteness of
\((x,y)\mapsto |x-y|^q\): for every finite signed measure \(\sigma\)
with \(\sigma(\mathbb R^d)=0\), $\iint |x-y|^q\,d\sigma(x)d\sigma(y)\le0$, see Berg--Christensen--Ressel~\cite[Chapter 3, Section 2, Corollary 2.10]{BergChristensenRessel1984}.
Applying this to \(\sigma=\mu-\bar\omega\) and expanding gives
\[
\iint |x-z|^q\,d\mu(x)d\mu(z)
-2\iint |x-y|^q\,d\mu(x)d\bar\omega(y)
+\iint |y-y'|^q\,d\bar\omega(y)d\bar\omega(y')
\le0.
\]
Therefore,
\[
\begin{aligned}
&\iint |x-y|^q\,d\mu(x)d\bar\omega(y)
-\frac12\iint |x-z|^q\,d\mu(x)d\mu(z)\\
&\hspace{2em}\ge
\frac12\iint |y-y'|^q\,d\bar\omega(y)d\bar\omega(y').
\end{aligned}
\]
It follows that
\[
E(\mu)\ge
\frac12\iint |y-y'|^{1+a}\,d\bar\omega(y)d\bar\omega(y')
\ge0,
\]
where we used \(M\ge1\). Thus \(E\) is bounded below in the
equal-exponent case as well. In conclusion, we have $\inf_{\mu\in\mathcal P(\mathbb R^d)}E(\mu)>-\infty.$

Now, we take care of the rest statements.
By Proposition \ref{prop:energy-diss}, choose \(t_k\to\infty\) such that
\eqref{eq:vanishing-dissipation} holds.  The moment bound
\eqref{eq:moment-assumption} gives tightness by Markov's inequality: for every
\(R>0\),
\begin{equation}\label{eq:tightness-markov-longtime}
        \sup_{t\ge0}\phi_t(\{|x|>R\})
        \le
        R^{-q}\sup_{t\ge0}\int_{\R^d}|x|^q\dd\phi_t(x)
        \longrightarrow0
        \qquad\text{as }R\to\infty .
\end{equation}
Therefore, by Prokhorov's theorem, after passing to a subsequence, $  \phi_{t_k}\weak\phi_\infty$
narrowly.  Lower semicontinuity gives
$\int_{\R^d}|x|^q\dd\phi_\infty(x)
        \le
        \liminf_{k\to\infty}
        \int_{\R^d}|x|^q\dd\phi_{t_k}(x)
        <\infty .$

The same moment bound implies uniform integrability of
\(|x|^{q-\delta}\) for every \(0<\delta<q\).  Indeed,
\begin{equation}\label{eq:uniform-integrability-longtime}
\sup_k
\int_{\{|x|>R\}} |x|^{q-\delta}\dd\phi_{t_k}(x)
\le
R^{-\delta}
\sup_k\int_{\R^d}|x|^q\dd\phi_{t_k}(x)
\longrightarrow0 .
\end{equation}
Thus the narrow convergence improves to convergence against every continuous
test function with growth \(O(1+|x|^{q-\delta})\).

We shall also use the corresponding tensor-product consequence.  Namely, if
\(F\in C(\R^d\times\R^d)\) satisfies
$ |F(x,y)|
        \le
        C_F\bigl(1+|x|^{q-\delta}+|y|^{q-\delta}\bigr)$
for some \(0<\delta<q\), then
$\iint F(x,y)\dd\phi_{t_k}(x)\dd\phi_{t_k}(y)
        \rightarrow
        \iint F(x,y)\dd\phi_\infty(x)\dd\phi_\infty(y).$
Indeed, first truncate \(F\) to a compact set and use narrow convergence of
\(\phi_{t_k}\otimes\phi_{t_k}\) to
\(\phi_\infty\otimes\phi_\infty\); then let the truncation radius go to
infinity using the uniform \(q\)-moment bound.

We now show that the approximate zero-flux relation along the sequence
\(\phi_{t_k}\) passes to the narrow limit \(\phi_\infty\).  Let
$ \zeta\in C_c^\infty(\R^d;\R^d),$
and let \(B_S\) be a ball containing \(\supp\zeta\).  From
\eqref{eq:vanishing-dissipation},
\begin{equation}\label{eq:test-diss-zero}
\int_{\R^d}
\zeta(x)\cdot
\bigl(K_a*\omega(x)-K_r*\phi_{t_k}(x)\bigr)
\dd\phi_{t_k}(x)
\longrightarrow0 .
\end{equation}

First consider the attraction term.  Define
$  V(x):=(K_a*\omega)(x)
        =
        \int_{\R^d}K_a(x-y)\omega(y)\dd y .$
For \(x\in B_S\),
$ |K_a(x-y)|\le C_S(1+|y|^a).$
Since \(q>a\) and \(\omega\) has finite \(q\)-moment, see Assumption \ref{ass:compactness}, \(V(x)\) is finite on
\(B_S\).  Moreover, \(V\) is continuous on \(B_S\).  For \(a>0\), this follows
from continuity of \(K_a\), dominated convergence, and the preceding growth
bound.  For \(a=0\), set \(K_0(0)=0\).  Then
 $ K_0(x_j-y)\to K_0(x-y)
        \text{ for a.e. }y$
as \(x_j\to x\), because \(\omega(y)\dd y\) has no atoms.  The integrand is
bounded by \(C\omega(y)\in L^1\), so dominated convergence again gives
continuity of \(V\).  Hence
$  \zeta\cdot V\in C_c(\R^d),$
and narrow convergence gives
\begin{equation}\label{eq:pass-attraction}
        \int_{\R^d}\zeta(x)\cdot V(x)\dd\phi_{t_k}(x)
        \longrightarrow
        \int_{\R^d}\zeta(x)\cdot V(x)\dd\phi_\infty(x).
\end{equation}

It remains to pass to the limit in the repulsive term.  First assume \(r>0\).
Set $G(x,y):=\zeta(x)\cdot K_r(x-y), K_r(0):=0.$
Since \(r>0\), the map \(K_r\) is continuous at the origin after this convention.
Thus \(G\) is continuous on \(\R^d\times\R^d\).  Moreover, for \(x\in B_S\),
\begin{equation}\label{eq:G-growth}
        |G(x,y)|
        \le
        C_S(1+|y|^r).
\end{equation}
Because \(q>r\), choose \(0<\delta<q-r\).  Then
$|G(x,y)|
        \le
        C_S(1+|x|^{q-\delta}+|y|^{q-\delta}),$
and the tensor-product convergence above gives
\[
\iint
\zeta(x)\cdot K_r(x-y)
\dd\phi_{t_k}(y)\dd\phi_{t_k}(x)
\longrightarrow
\iint
\zeta(x)\cdot K_r(x-y)
\dd\phi_\infty(y)\dd\phi_\infty(x).
\]

Now assume \(r=0\).  The direct kernel $ K_0(x-y)=\frac{x-y}{|x-y|}, K_0(0):=0,$
is bounded but discontinuous on the diagonal \(x=y\).  Therefore the preceding
continuity argument cannot be applied directly.  However, the self-interaction
term can be symmetrized using the oddness of \(K_0\).  For every probability
measure \(\mu\),
\[
\iint
        \zeta(x)\cdot K_0(x-y)
        \dd\mu(y)\dd\mu(x)
=
\frac12
\iint
        \bigl(\zeta(x)-\zeta(y)\bigr)\cdot K_0(x-y)
        \dd\mu(x)\dd\mu(y).
\]
Define
\[
        H(x,y):=
        \begin{cases}
        \displaystyle
        \bigl(\zeta(x)-\zeta(y)\bigr)\cdot\dfrac{x-y}{|x-y|},
        & x\ne y,\\[2ex]
        0, & x=y.
        \end{cases}
\]
Since \(\zeta\) is smooth and compactly supported,
$ |\zeta(x)-\zeta(y)|\le \|\nabla\zeta\|_{L^\infty}|x-y|.$
Hence \(H\) is continuous at the diagonal.  Away from the diagonal it is plainly
continuous.  Also
$  |H(x,y)|\le 2\|\zeta\|_{L^\infty},$
so \(H\in C_b(\R^d\times\R^d)\).  Since
$ \phi_{t_k}\otimes\phi_{t_k}
        \weak
        \phi_\infty\otimes\phi_\infty$
narrowly, we get
$\iint H(x,y)
        \dd\phi_{t_k}(x)\dd\phi_{t_k}(y)
        \rightarrow
        \iint H(x,y)
        \dd\phi_\infty(x)\dd\phi_\infty(y).$
Using the symmetrized identity for both \(\phi_{t_k}\) and
\(\phi_\infty\), we conclude that
\[
\iint
        \zeta(x)\cdot K_0(x-y)
        \dd\phi_{t_k}(y)\dd\phi_{t_k}(x)
\longrightarrow
\iint
        \zeta(x)\cdot K_0(x-y)
        \dd\phi_\infty(y)\dd\phi_\infty(x).
\]

Therefore, for every \(0\le r\le1\),
\begin{equation}\label{eq:pass-repulsion}
\iint_{\R^d\times\R^d}
\zeta(x)\cdot K_r(x-y)
\dd\phi_{t_k}(y)\dd\phi_{t_k}(x)
\longrightarrow
\iint_{\R^d\times\R^d}
\zeta(x)\cdot K_r(x-y)
\dd\phi_\infty(y)\dd\phi_\infty(x).
\end{equation}
Passing to the limit in \eqref{eq:test-diss-zero}, using
\eqref{eq:pass-attraction} and \eqref{eq:pass-repulsion}, gives
$\int_{\R^d}
        \zeta(x)\cdot
        \bigl(K_a*\omega(x)-K_r*\phi_\infty(x)\bigr)
        \dd\phi_\infty(x)
        =
        0.$
Since \(\zeta\in C_c^\infty(\R^d;\R^d)\) was arbitrary, this is exactly $ \phi_\infty\bigl(K_a*\omega-K_r*\phi_\infty\bigr)=0$
in the sense of vector-valued distributions.  Thus \(\phi_\infty\) is a
zero-flux stationary state.
\end{proof}

\subsection{Miscellany}

\begin{remark}[Dini derivative]\label{rmk:dini}

We recall the corresponding maximum-rule argument for
\(R(t)\).  Set $ K:=\operatorname{supp}\phi_0 .$
Since \(\phi_t=(X_t)_{\#}\phi_0\), we can write
$  R(t)=\max_{x\in K}|X_t(x)|.$
For each fixed \(x\in K\), define
$ F_x(t):=|X_t(x)|.$
At times when \(R(t)>0\), every maximizer satisfies \(X_t(x)\neq0\), and hence
\(F_x\) is differentiable at such active points, with
\[
        \frac{\dd}{\dd t}F_x(t)
        =
        \frac{X_t(x)}{|X_t(x)|}\cdot \dot X_t(x).
\]

We now prove the upper Dini estimate.  Choose \(h_n\downarrow0\) along a
subsequence realizing the upper limit, and choose \(x_n\in K\) such that $ R(t+h_n)=|X_{t+h_n}(x_n)|.$
Since \(K\) is compact, after passing to a subsequence, $  x_n\to x_*\in K .$
By the continuity of the flow and of \(R(t)\), we obtain
$ |X_t(x_*)|=R(t),$
so \(x_*\) is an active maximizer at time \(t\).  Moreover,
\[
\begin{aligned}
\frac{R(t+h_n)-R(t)}{h_n}
&=
\frac{F_{x_n}(t+h_n)-R(t)}{h_n} \\
&\le
\frac{F_{x_n}(t+h_n)-F_{x_n}(t)}{h_n},
\end{aligned}
\]
because \(F_{x_n}(t)\le R(t)\).  Letting \(n\to\infty\), and using the
continuity of the vector field along the flow, the right-hand side converges to
$ \frac{\dd}{\dd t}F_{x_*}(t)
        =
        \frac{X_t(x_*)}{|X_t(x_*)|}\cdot \dot X_t(x_*).$
Therefore
\[
        D^+R(t)
        \le
        \frac{X_t(x_*)}{|X_t(x_*)|}\cdot \dot X_t(x_*).
\]
Taking the maximum over all active maximizers gives
\[
        D^+R(t)
        \le
        \max_{\substack{x\in\operatorname{supp}\phi_0\\ |X_t(x)|=R(t)}}
        \frac{X_t(x)}{|X_t(x)|}\cdot \dot X_t(x).
\]

For completeness, we recall the maximum-rule argument for the diameter.  Set $ \overline K:=\operatorname{supp}\phi_0\times\operatorname{supp}\phi_0 .$
Since \(\phi_t=(X_t)_{\#}\phi_0\), we can write
$ D(t)^2
        =
        \max_{(x,z)\in\overline K}|X_t(x)-X_t(z)|^2 .$
For each fixed pair \((x,z)\in\overline K\), define
$ F_{x,z}(t):=|X_t(x)-X_t(z)|^2 .$
Then \(F_{x,z}\) is differentiable in \(t\), and
$ \frac{\dd}{\dd t}F_{x,z}(t)
        =
        2\left\langle
        \dot X_t(x)-\dot X_t(z),
        X_t(x)-X_t(z)
        \right\rangle .$

Choose \(h_n\downarrow0\) along a subsequence realizing the upper limit in
\(D^+D(t)^2\), and choose \((x_n,z_n)\in\overline K\) such that
$  D(t+h_n)^2
        =
        |X_{t+h_n}(x_n)-X_{t+h_n}(z_n)|^2 .$
Since \(\overline K\) is compact, after passing to a subsequence,
$  (x_n,z_n)\to(x_*,z_*)\in\overline K .$
By the continuity of the flow and of \(D(t)\), we obtain
$  |X_t(x_*)-X_t(z_*)|=D(t),$
so \((x_*,z_*)\) is a diameter pair at time \(t\).  Moreover,
\[
\begin{aligned}
\frac{D(t+h_n)^2-D(t)^2}{h_n}
&=
\frac{F_{x_n,z_n}(t+h_n)-D(t)^2}{h_n}  \\
&\le
\frac{F_{x_n,z_n}(t+h_n)-F_{x_n,z_n}(t)}{h_n},
\end{aligned}
\]
because \(F_{x_n,z_n}(t)\le D(t)^2\).  Letting \(n\to\infty\), and using the
continuity of the vector field along the flow, the right-hand side converges to $ \frac{\dd}{\dd t}F_{x_*,z_*}(t)
        =
        2\left\langle
        \dot X_t(x_*)-\dot X_t(z_*),
        X_t(x_*)-X_t(z_*)
        \right\rangle .$
Therefore
\[
        D^+D(t)^2
        \le
        2\left\langle
        \dot X_t(x_*)-\dot X_t(z_*),
        X_t(x_*)-X_t(z_*)
        \right\rangle .
\]
Taking the maximum over all diameter pairs gives
\[
        D^+D(t)^2
        \le
        \max_{\substack{x,z\in\operatorname{supp}\phi_0\\
        |X_t(x)-X_t(z)|=D(t)}}
        2\left\langle
        \dot X_t(x)-\dot X_t(z),
        X_t(x)-X_t(z)
        \right\rangle .
\]

\end{remark}

\begin{example}[Particle escape with background mass larger than one]
\label{ex:particle-escape-mass-larger-than-one}
Let \(d=1\), let
$ a=r=s>1,$
and choose a number
$ 1<m<2^{s-1}.$
Let the background be
$ \omega=m\delta_0,$
so that
$ \omega(\mathbb R)=m>1.$
Take the initial particle measure
$  \phi_0=\frac12\delta_{-X_0}+\frac12\delta_{X_0},
        \qquad X_0>0.$
We look for a symmetric particle solution of the form
\[
        \phi_t=\frac12\delta_{-X(t)}+\frac12\delta_{X(t)},
        \qquad X(t)>0.
\]
Since $ K_s(x)=(1+s)\operatorname{sgn}(x)|x|^s,$
the velocity of the right particle is
\[
\begin{aligned}
        \dot X(t)
        &=
        -K_s*\omega(X(t))+K_s*\phi_t(X(t))  \\
        &=
        -m(1+s)X(t)^s
        +
        \frac12(1+s)(2X(t))^s .
\end{aligned}
\]
Therefore $  \dot X(t)
        =
        (1+s)\bigl(2^{s-1}-m\bigr)X(t)^s .$
By the choice \(1<m<2^{s-1}\), the coefficient is positive. Hence the particles
move outward. Solving the scalar ODE gives
\[ X(t)
        =
        \left[
        X_0^{1-s}
        -
        (s-1)(1+s)\bigl(2^{s-1}-m\bigr)t
        \right]^{-1/(s-1)} .\]
Thus
$ X(t)\to+\infty
        \text{ as }t\uparrow T_*,$
where
$ T_*
        =
        \frac{X_0^{1-s}}
        {(s-1)(1+s)(2^{s-1}-m)}.$
Consequently,
$\operatorname{diam}(\operatorname{supp}\phi_t)=2X(t)$
becomes unbounded in finite time, even though \(\omega(\mathbb R)=m>1\).
This shows that the uniform compactness result for \(0\le r\le1\) cannot be
extended to the superlinear equal-exponent regime \(a=r>1\).
\end{example}

\begin{remark}[Distributional Laplacian of the exterior radial solution]\label{rmk:104}

We provide the details leading to the boundary term in the
distributional Laplacian of the exterior Poisson solution. Let \(R>0\)
and let \(f:(R,\infty)\to\R\) be locally integrable and satisfy
\begin{equation}\label{eq:appendix-f-integrability}
    \int_R^\infty s|f(s)|\,\dd s<\infty.
\end{equation}
Suppose that \(g_R\) solves
\begin{equation}\label{eq:appendix-radial-poisson}
    g_R''(r)+\frac{2}{r}g_R'(r)=f(r),
    \qquad r>R,
\end{equation}
subject to
\begin{equation}\label{eq:appendix-radial-boundary-data}
    g_R(R)=0,
    \qquad
    \lim_{r\to\infty}g_R(r)=0.
\end{equation}
Define the zero extension \(G_R:\R^3\to\R\) by
\begin{equation}\label{eq:appendix-zero-extension}
    G_R(x)
    =
    \begin{cases}
        g_R(|x|),& |x|>R,\\
        0,& |x|\leq R.
    \end{cases}
\end{equation}

For a radial function \(G_R(x)=g_R(|x|)\), the three-dimensional
Laplacian satisfies $  \Delta G_R(x)
    =
    g_R''(r)+\frac{2}{r}g_R'(r)
    =
    \frac{1}{r^2}\bigl(r^2g_R'(r)\bigr)', r=|x|>R.$
Consequently, \eqref{eq:appendix-radial-poisson} is equivalent to
\begin{equation}\label{eq:appendix-radial-ode}
    \bigl(r^2g_R'(r)\bigr)'=r^2f(r),
    \qquad r>R.
\end{equation}
Integrating \eqref{eq:appendix-radial-ode} from \(R\) to \(r\) gives $ r^2g_R'(r)
    =
    C_R+\int_R^r s^2f(s)\,\dd s,$
where $ C_R=R^2g_R'(R+).$
It follows that
\begin{equation}\label{eq:appendix-g-prime}
    g_R'(r)
    =
    \frac{C_R}{r^2}
    +
    \frac{1}{r^2}\int_R^r s^2f(s)\,\dd s.
\end{equation}

We next determine \(C_R\) from the two boundary conditions in
\eqref{eq:appendix-radial-boundary-data}. Since
\(g_R(\infty)=g_R(R)=0\), we have
\begin{equation}\label{eq:appendix-integral-g-prime-zero}
    0
    =
    g_R(\infty)-g_R(R)
    =
    \int_R^\infty g_R'(u)\,\dd u.
\end{equation}
Substituting \eqref{eq:appendix-g-prime} into
\eqref{eq:appendix-integral-g-prime-zero} yields
\begin{equation}\label{eq:appendix-CR-intermediate}
    0
    =
    C_R\int_R^\infty\frac{\dd u}{u^2}
    +
    \int_R^\infty
    \frac{1}{u^2}
    \left(\int_R^u s^2f(s)\,\dd s\right)\dd u.
\end{equation}
The first integral on the right-hand side is \(1/R\). For the second
one, Fubini's theorem gives
\begin{equation}\label{eq:appendix-fubini-calculation}
\begin{aligned}
    \int_R^\infty
    \frac{1}{u^2}
    \left(\int_R^u s^2f(s)\,\dd s\right)\dd u
    &=
    \int_R^\infty
    s^2f(s)
    \left(\int_s^\infty\frac{\dd u}{u^2}\right)\dd s\\
    &=
    \int_R^\infty sf(s)\,\dd s.
\end{aligned}
\end{equation}
The exchange of integrals is justified by
\eqref{eq:appendix-f-integrability}, since
\begin{equation}\label{eq:appendix-fubini-absolute}
\begin{aligned}
    \int_R^\infty
    \frac{1}{u^2}
    \left(\int_R^u s^2|f(s)|\,\dd s\right)\dd u
    &=
    \int_R^\infty
    s^2|f(s)|
    \left(\int_s^\infty\frac{\dd u}{u^2}\right)\dd s\\
    &=
    \int_R^\infty s|f(s)|\,\dd s
    <\infty.
\end{aligned}
\end{equation}
Combining \eqref{eq:appendix-CR-intermediate} and
\eqref{eq:appendix-fubini-calculation}, we obtain
\begin{equation}\label{eq:appendix-CR-formula}
    C_R
    =
    -R\int_R^\infty sf(s)\,\dd s.
\end{equation}
Since \(C_R=R^2g_R'(R+)\), this proves
\begin{equation}\label{eq:appendix-normal-derivative}
    g_R'(R+)
    =
    -\frac{1}{R}\int_R^\infty sf(s)\,\dd s.
\end{equation}

We now compute the distributional Laplacian of the zero extension
\eqref{eq:appendix-zero-extension}. Let
\(\zeta\in C_c^\infty(\R^3)\). By the definition of the
distributional Laplacian,
\begin{equation}\label{eq:appendix-distribution-pairing}
    \langle\Delta G_R,\zeta\rangle
    =
    \int_{\{|x|>R\}}G_R(x)\Delta\zeta(x)\,\dd x.
\end{equation}
Applying Green's second identity in the exterior domain
\(\Omega_R:=\{x\in\R^3:|x|>R\}\), and using
\(\Delta G_R=f(|x|)\) there, gives
\begin{equation}\label{eq:appendix-green-identity}
\begin{aligned}
    \langle\Delta G_R,\zeta\rangle
    &=
    \int_{\Omega_R}f(|x|)\zeta(x)\,\dd x\\
    &\quad+
    \int_{\partial\Omega_R}
    \left(
        G_R\partial_{\nu_{\Omega_R}}\zeta
        -
        \zeta\partial_{\nu_{\Omega_R}}G_R
    \right)\dd\mathcal H^2.
\end{aligned}
\end{equation}
There is no boundary contribution at infinity because \(\zeta\) has
compact support. On the inner boundary \(\partial\Omega_R=\partial
B_R\), the outward unit normal of the exterior domain is
\(\nu_{\Omega_R}=-e_r\). Moreover, \(G_R=0\) on \(\partial B_R\) and
\begin{equation}\label{eq:appendix-exterior-normal-derivative}
    \partial_{\nu_{\Omega_R}}G_R
    =
    -g_R'(R+).
\end{equation}
Therefore, the boundary term in
\eqref{eq:appendix-green-identity} reduces to
\begin{equation}\label{eq:appendix-boundary-term}
    \int_{\partial B_R}
    \zeta\,g_R'(R+)\,\dd\mathcal H^2.
\end{equation}
We conclude that
\begin{equation}\label{eq:appendix-distributional-laplacian}
    \Delta G_R
    =
    f(|x|)\mathbf 1_{\{|x|>R\}}\,\dd x
    +
    g_R'(R+)\,
    \mathcal H^2\llcorner\partial B_R
    \qquad\text{in }\mathcal S'(\R^3).
\end{equation}
Substituting \eqref{eq:appendix-normal-derivative} into
\eqref{eq:appendix-distributional-laplacian} gives the explicit
formula
\begin{equation}\label{eq:appendix-distributional-laplacian-explicit}
    \Delta G_R
    =
    f(|x|)\mathbf 1_{\{|x|>R\}}\,\dd x
    -
    \frac{1}{R}
    \left(\int_R^\infty sf(s)\,\dd s\right)
    \mathcal H^2\llcorner\partial B_R
    \qquad\text{in }\mathcal S'(\R^3).
\end{equation}

In particular, suppose that
\(\omega=f(|x|)\,\dd x\) and that the stationary candidate is defined
by \(\phi_R=\omega-\Delta G_R\). Then
\eqref{eq:appendix-distributional-laplacian-explicit} yields
\begin{equation}\label{eq:appendix-phi-explicit}
    \phi_R
    =
    f(|x|)\mathbf 1_{B_R}\,\dd x
    +
    \frac{1}{R}
    \left(\int_R^\infty sf(s)\,\dd s\right)
    \mathcal H^2\llcorner\partial B_R.
\end{equation}
Thus the absolutely continuous part of \(\phi_R\) is the restriction
of the background density to \(B_R\), while the jump of the normal
derivative of the zero extension generates a singular measure on
\(\partial B_R\). If \(f\geq0\), the coefficient of this surface
measure is nonnegative by
\eqref{eq:appendix-normal-derivative}.

Finally, the total mass of the candidate in
\eqref{eq:appendix-phi-explicit} is
\begin{equation}\label{eq:appendix-phi-mass}
    \phi_R(\R^3)
    =
    4\pi\int_0^R r^2f(r)\,\dd r
    +
    4\pi R\int_R^\infty rf(r)\,\dd r.
\end{equation}
Hence the requirement that \(\phi_R\) be a probability measure selects
\(R\) through
\begin{equation}\label{eq:appendix-radius-condition}
    4\pi\int_0^R r^2f(r)\,\dd r
    +
    4\pi R\int_R^\infty rf(r)\,\dd r
    =
    1.
\end{equation}
\end{remark}

\end{document}